\documentclass[12pt,a4paper]{amsart}
\usepackage{amsfonts}
\usepackage{amssymb}
\usepackage{amsmath}
\usepackage{amsthm}

\usepackage{cite}
\usepackage{enumitem}
\usepackage{tikz}
\usepackage{subfigure}
\usepackage{MnSymbol}

\theoremstyle{plain}
\newtheorem{thm}{Theorem}

\newtheorem{lem}{Lemma}
\newtheorem{prop}{Proposition} 
\newtheorem{cor}{Corollary}

\newcommand{\vecII}[2]{
\ensuremath{
\begin{pmatrix}
#1 \\ #2 \\
\end{pmatrix}}}
\newcommand{\matII}[4]{
\ensuremath{ 
\begin{pmatrix}  
#1 & #2 \\
#3 & #4 \\
\end{pmatrix}}}

\newcommand{\q}[1]{``#1''}

\providecommand{\sm}{\setminus}
\providecommand{\N}{\mathbb{N}}
\providecommand{\R}{\mathbb{R}}
\providecommand{\Z}{\mathbb{Z}}
\providecommand{\C}{\mathbb{C}}
\providecommand{\eps}{\varepsilon}
\providecommand{\ov}{\overline}
\providecommand{\dx}{\,dx}
\providecommand{\dy}{\,dy}

\providecommand{\dH}{\,d\mathcal{H}}

\providecommand{\skp}[2]{\langle#1,#2\rangle}
\providecommand{\Bigskp}[2]{\Big\langle#1,#2\Big\rangle}

\DeclareMathOperator{\diver}{div}

\DeclareMathOperator{\supp}{supp}
\DeclareMathOperator{\sign}{sign}
\DeclareMathOperator{\sgn}{sgn}
\DeclareMathOperator{\diag}{diag}

\DeclareMathOperator{\Real}{Re}
\DeclareMathOperator{\Jac}{Jac}
\DeclareMathOperator{\Imag}{Im}  

\renewcommand{\qed}{\hfill $\Box$}

\def\Xint#1{\mathchoice
   {\XXint\displaystyle\textstyle{#1}}%
   {\XXint\textstyle\scriptstyle{#1}}%
   {\XXint\scriptstyle\scriptscriptstyle{#1}}%
   {\XXint\scriptscriptstyle\scriptscriptstyle{#1}}%
   \!\int} 
\def\XXint#1#2#3{{\setbox0=\hbox{$#1{#2#3}{\int}$} 
     \vcenter{\hbox{$#2#3$}}\kern-.5\wd0}}
\def\avint{\Xint-}

\begin{document}

\allowdisplaybreaks 

\title[The limiting absorption principle for periodic differential operators]{The limiting
absorption principle for periodic differential operators and applications to nonlinear Helmholtz equations}

\author{Rainer Mandel}
\address{~\hfill\break 
R. Mandel \hfill\break
Karlsruhe Institute of Technology \hfill\break
Institute for Analysis \hfill\break
Englerstra{\ss}e 2 \hfill\break
D-76131 Karlsruhe, Germany}
\email{~\hfill\break Rainer.Mandel@kit.edu}
\date{}


\begin{abstract}
  We prove an $L^p$-version of the limiting absorption principle for a class of periodic
  elliptic differential operators of second order. The result is applied to the construction of nontrivial
  solutions of nonlinear Helmholtz equations with periodic coefficient functions.
\end{abstract}
 
\maketitle
\allowdisplaybreaks

% % 
% Open: $W^{2,p}$-limiting absorption principle, Cald�ron-Zygmund theory for Bloch functions (estimates for the
% nonresonant part), Open: Stein-Tomas

\section{Introduction} 

%   \r{   Error:
%   $\frac{2(d+1)}{d+3}< q<\frac{2d}{d+1}$,implies $\frac{q(d+1)}{d+1-2q}\r{<}\frac{2q}{2d-q(d+1)}$, so the
%   Stein-Tomas bound is still better. \\
%   Adaptation of Stein-Tomas to Floquet-Bloch instead of the Kachkovskii-version does not seem obvious even if 
%   $Uf(x,k)=\check{f(x-\cdot)}(k)$. \\
%   }
 
  In this paper we study elliptic partial differential equations of the form 
  \begin{gather} \label{eq:linear_problem}
    Lu -\lambda u =f \quad\text{in }\R^d
  \end{gather}
  where $L\psi:= -\diver(A(\cdot)\nabla \psi)+V(\cdot)\psi$ is a Schr\"odinger-type operator with periodic
  coefficient functions that are sufficiently regular. For $\lambda$ outside the spectrum of the selfadjoint
  operator $L:L^2(\R^d)\supset H^2(\R^d)\to L^2(\R^d)$ this equation is invertible, i.e., a unique solution
  $u\in H^2(\R^d)$ of \eqref{eq:linear_problem} exists. What about $\lambda$ inside the spectrum of $L$? This
  issue is much more delicate and a general answer for large classes of operators is missing. There is, however, a
  general strategy called \q{limiting absorption principle} how to find nontrivial solutions
  of~\eqref{eq:linear_problem} for such~$\lambda$.
  On an abstract level, any such limiting absorption principle is characterized by a class of
  coefficient functions $A,V$ and real function spaces $X,Y$ such that for all $f\in Y$ and
  $\eps\in\R\sm\{0\}$ there is a unique solution $u^\eps\in X+iX$ of the perturbed equation 
  \begin{equation}\label{eq:linear_problem_eps}
    Lu - (\lambda + i\eps) u = f \quad\text{in }\R^d
  \end{equation}
  such that $u^\eps$ converges as $\eps\to 0^\pm$ to a solution $u^\pm \in X+iX$ of \eqref{eq:linear_problem}
  in a suitable topology.
%   Typically one observes
%   that $u^+,u^-$ are different and complex-valued solutions of \eqref{eq:linear_problem}, which
%   additionally satisfy a so-called radiation condition at infinity that distinguishes it from other
%   solutions. In this respect the radiation condition plays the role of a boundary condition. 
%   In the context of the Helmholtz equation (i.e. $A\equiv \id_{d\times d},V\equiv 0$) the functions $u^\pm$
%   satisfy Sommerfeld's outgoing respectively ingoing radiation condition 
%   $$
%     |x|^{\frac{1-d}{2}} \Big( \frac{\partial}{\partial |x|} \mp  i \sqrt{\lambda}\Big)u(x)
%     \to 0 \qquad\text{as  }|x|\to\infty. 
%   $$
  Let us give some examples for Schr\"odinger operators of the form $L=-\Delta+V(x)$ in $\R^3$.
  
  \medskip
  
  One of the first results on limiting absorption principles for such operators is due to Odeh \cite{Odeh_LAP}
  who proved uniform convergence of the $u^\eps$ for square integrable\footnote{Odeh requires the right hand
  side to be \q{integrable}, but probably \q{square integrable} is meant in view of the fact that he speaks
  of a unique $L^2$-solution of \eqref{eq:linear_problem_eps}.} right hand sides $f$ with compact support
  provided the potential $V$ decays sufficiently fast at infinity in an averaged sense. Another famous result is due to
  Agmon (Theorem 4.1 in \cite{Agmon_spectral}) who used differently weighted $L^2-$spaces $X$ and $Y$ and so-called short range
  potentials satisfying $V(x)=O(|x|^{-1-\delta})$ as $|x|\to\infty$ for some $\delta>0$. A
  generalization to Helmholtz equations in unbounded and asymptotically conic manifolds was recently proved 
  by Rodnianski and Tao \cite{RodTao_LAP}.
  Further versions of the limiting absorption principle in Morrey-Campanato spaces, again for evanescent
  potentials, can be found in~\cite{CacDanLuc_helmholtz} or~\cite{PerVega_Morrey}. Goldberg and Schlag \cite{GolSch_LAP}
  proved an $L^p$-version of the limiting absorption principle ($X=L^4(\R^3),Y=L^{4/3}(\R^3)$) for potentials
  $V\in L^r(\R^3)\cap L^{3/2}(\R^3)$ with $r>\frac{3}{2}$. Each of these results relies on
  the decay of the potential $V$, which ensures that the resolvent of $-\Delta+V(x)-\lambda-i\eps$ resembles
  the one of $-\Delta-\lambda-i\eps$ as far as the asymptotic properties at infinity are concerned. We stress
  that a control of the global regularity and integrability of the functions $u^\eps$ represents the main
  difficulty since convergence on compact sets can be proved under very mild assumptions on $V$. For
  instance, in~1962 Eidus~\cite{Eidus_LAP} proved a convergence result in $H^2_{loc}(\R^3)$ whenever $V$ is
  bounded from below and locally bounded from above.
  Being interested in global regularity for solutions of periodic problems, we need to take a different
  approach. The main tool of our analysis is Floquet-Bloch theory, which provides a qualitative
  description of the spectrum of elliptic periodic differential operators. As we will see, combining this
  approach with suitable assumptions on the so-called band structure of $L$ leads to a new limiting absorption
  principle. In our analysis we mainly take advantage of the papers by Guti\'{e}rrez~\cite{Gut_nontrivial} and
  Radosz~\cite{Rad_LAP}. The first-mentioned paper provides an $L^p$-version of the limiting absorption
  principle for the Helmholtz operator $-\Delta-\lambda$, while the second paper contains the main ideas 
  how Floquet-Bloch analysis may be used in order to establish a limiting absorption principle for periodic
  problems. Our contribution is to combine the methods from both papers in order to prove an
  $L^p$-version for the limiting absorption principle in the periodic setting. Accordingly,
  both papers are of fundamental importance for this paper, so we provide some details.
   
  \medskip
  
  In \cite{Gut_nontrivial} Theorem~6 Guti\'{e}rrez shows that for all $\lambda>0$ the family of resolvent
  operators $(-\Delta-\lambda-i\eps)^{-1}:L^p(\R^d)\to L^q(\R^d;\C)$ is equibounded with respect to
  $\eps\in\R\sm\{0\}$ provided $d\geq 3$ and $p,q$ are chosen suitably, see~\eqref{eq:pq_Gutierrez}. Here the
  task is to analyze the functions
  $$
    u^\eps 
    := (-\Delta-\lambda-i\eps)^{-1}f 
    = \mathcal{F}^{-1} \left(\frac{\hat f(\cdot)}{|\cdot|^2-\lambda -  i\eps}\right).
  $$
  Guti\'{e}rrez' a priori estimates allow to pass to a weak limit of the $u^\eps$ in $L^q(\R^d;\C)$ as
  $\eps\to 0^\pm$ and the limit functions $u^+,u^-\in L^q(\R^d;\C)$ are given by
  \begin{align} \label{eq:upm_formula}
    \begin{aligned}
    u^\pm(x)
     &= \int_{\R^d}  \frac{i}{4}(2\pi|x-y|)^{\frac{2-d}{2}}H^{(1)}_{\frac{d-2}{2}}(x-y) f(y)\,dy   \\
     &= (2\pi)^{-\frac{d}{2}} \left( p.v. \int_{\R^d} \frac{\hat f(\xi)}{|\xi|^2-\lambda} 
     e^{i\skp{x}{\xi}}\,d\xi + i \pi \int_{\{|\xi|^2=\lambda\}}  \frac{\hat
     f(\xi)}{2\sqrt\lambda} e^{i\skp{x}{\xi}}  \dH^{d-1}(\xi)\right),
     \end{aligned}
  \end{align}
  where $H^{(1)}_{(d-2)/2}:\R\to\C$ denotes the Hankel function of the first kind, see~(11) in
  \cite{EvWe_dual}.  The formula from the second line  follows from Lemma~5.1 in~\cite{Ruiz_LN}. 
  It shows some similarities with the formula 
  obtained by Radosz in the case of a periodic Schr\"odinger operator $L=-\Delta+V(x)$, see Theorem~2.13
  in~\cite{Rad_LAP}. Using Floquet-Bloch theory~\cite{Bloch_Quantenmechanik,Floq_Sur_les_eq} Radosz analyzed
  the convergence of the functions $u^\eps(\lambda,\cdot):=(L-\lambda+i\eps)^{-1}f$ as $\eps\to 0^\pm$ and
  determined complex-valued functions $u^+,u^-$ satisfying
  $$
    \int_{\R^d\times I} u^\pm(\lambda,x)\big((L-\lambda)\phi(\lambda,\cdot)\big)(x) \,d(x,\lambda)
    = \int_{\R^d\times I} f(x)\phi(\lambda,x) \,d(x,\lambda) 
    \quad\text{for all }\phi\in C_0^\infty(I\times \R^d), 
  $$
  where $I\subset\R$ is a sufficiently small interval containing a \q{regular frequency}
  $\lambda\in\sigma(L)$, cf. Definition~1.1 in \cite{Rad_LAP}. More precisely, she shows in Theorem~1.2 that
  the functions $u^\eps$ converge to some $u^\pm$ as $\eps\to 0^\pm$ in the space $L^2(I,Z)$ where $Z$ is a
  suitably weighted $L^2-$space. Nonetheless, Radosz' results are weaker than one may hope for in view of
  Guti\'{e}rrez' results for constant potentials.  First of all, it is expected that a convergence result holds
  true for every fixed regular frequency $\lambda$ in the spectrum of $L$, which cannot be deduced from
  convergence in $L^2(I,Z)$. Furthermore, the topology of the weighted $L^2-$space $Z$ is rather coarse given
  that the weight function is assumed to have some decay at infinity, see p.255-256 and Definition~2.7
  in~\cite{Rad_LAP}. As a consequence, Radosz' techniques do not allow to control the decay of the functions
  $u^\eps(\lambda,\cdot)$ and $u^\pm(\lambda,\cdot)$ at infinity. These shortcomings were our motivation to
  look for a limiting absorption principle that may substitute Guti\'{e}rrez' results \cite{Gut_nontrivial} when
  the differential operator $L$ has periodic instead of constant coefficient functions. Our
  Theorem~\ref{thm_LAP} provides such a new result for a class of differential operators $L$ and
  regular frequencies $\lambda$ satisfying the assumptions (A1),(A2),(A3) that we are going to introduce and
  motivate next.
  
  \medskip 
  
  Our first assumption says that we deal with uniformly elliptic partial differential equations of second
  order in divergence form with $\Z^d$-periodic coefficient functions so that Floquet-Bloch theory is
  applicable. Clearly, by a change of coordinates, other periodicities can be dealt with, too. So we require
  the following:
  \begin{itemize}
    \item[(A1)] $L\psi =-\diver(A(\cdot)\nabla\psi)+V(\cdot)\psi$ for $\Z^d$-periodic coefficient functions
    $A\in C^1_b(\R^d,\R^{d\times d})$ and $V\in L^\infty(\R^d)$ such that $A(x)$ is symmetric
    and $\skp{\xi}{A(x)\xi}\geq c|\xi|^2$ holds for some $c>0$ and all $x,\xi\in\R^d$.
  \end{itemize}
  Under this assumption the operator $L$ is selfadjoint on $\R^d$ with domain $H^2(\R^d)$ and its
  spectrum has a so-called band structure. This means that the
  spectrum of $L$ is the union of infinitely many bands $\lambda_s(B)$ where the band functions $\lambda_s$ are continuous and
  $B=[-\pi,\pi]^d$ is the so-called Brillouin zone, named after L\'{e}on Brillouin in honor of his
  contributions to the study of wave propagation in periodic media~\cite{Bri_wave}. 
  The relation between the band functions $\lambda_s$ and the operator $L$ is given by the following
  $k$-dependent selfadjoint quasiperiodic eigenvalue problems on the periodicity cell $\Omega:=(0,1)^d$:
  \begin{align} \label{eq:FB_ev_problem}
   \begin{aligned}
   L\psi &= \lambda \psi \qquad\qquad\quad\,\text{in }\Omega, \\
   \psi(x+n) &= e^{i\skp{k}{n}}\psi(x) \qquad\text{for }x\in\R^d\text{ and all }n\in\Z^d. 
   \end{aligned}
 \end{align}
 For every $k\in B$ there is an orthonormal basis $\{\psi_s(\cdot,k) : s\in\Z^d\}$ in $L^2(\Omega;\C)$ consisting
 of eigenfunctions of \eqref{eq:FB_ev_problem} with associated eigenvalues $\{\lambda_s(k):s\in\Z^d\}$ so that
 the band structure takes the form
 \begin{equation} \label{eq:bandstructure}
   \sigma(L) = \bigcup_{s\in\Z^d} \lambda_s(B) = \bigcup_{s\in\Z^d} \{\lambda_s(k):k\in B\}.
 \end{equation}
 A proof of \eqref{eq:bandstructure} may be found in Lemma~4 and Lemma~5 in the paper bei
 Odeh and Keller~\cite{OdKe_Bloch}. Notice that their result is formulated for continuous and
 $\Z^d-$periodic potentials $V$, but extends to bounded ones as in (A1). We will use that the functions
 $k\mapsto \psi_s(\cdot,k)\in L^2(\Omega;\C)$ can be chosen to be measurable, see Lemma 5.3~b) 
 in~\cite{BerShu_Schroedinger}. 
 %Moreover, we may extend the $\psi_s$ continuously to $\R^d\times B$ by
 %quasiperiodicity, i.e., by defining $\psi_s(x+n,k)=e^{i\skp{k}{n}}\psi_s(x,k)$ for $x\in\Omega,n\in\Z^d$,
 %see~\eqref{eq:FB_ev_problem}.
 A very subtle point concerns the labeling of the eigenpairs $(\psi_s(\cdot,k),\lambda_s(k))$. A common way to
 do this is to use $\N_0$ instead of $\Z^d$ as an index set and to order the eigenvalues by requiring $\lambda_j(k)\leq
 \lambda_{j+1}(k)$ for all ${j\in\N_0}$. This approach is  used for instance in \cite{OdKe_Bloch} or
 in Eastham's book, see Chapter~6 in~\cite{Eas_TheSpectral}. The advantage of this numbering is two-fold:
 Firstly, it is intuitive and secondly, the $\Z^d-$periodicity and Lipschitz continuity of the band functions
 immediately follow from the min-max-characterization of eigenvalues.
 %, see for instance Theorem~1.1 and Proposition~1.4 in~\cite{ConVan_Fourier}.  
 In this paper, however, we do not use this labeling. The reason 
 is that for this labeling Lipschitz continuity is the best regularity one may in general hope for. Indeed, 
 it is possible that the bands $\lambda_s(B)$ intersect each other transversally so that the
 crossings destroy every kind of differentiability property (but not the Lipschitz continuity) of the
 ordered band functions $\lambda_1\leq \lambda_2,\ldots$, see~p.143~\cite{Kato_short_intro}. 
 This phenomenon is illustrated schematically in Figure XIII.15 in~\cite{RS_analysis_of_operators} in the
 one-dimensional setting. A numerical example for $d=2$ and $L=-\Delta+V(x)$ with a concrete potential~$V$ may be found on p.863 in~\cite{DoUe_CME}.
 We choose the index set~$\Z^d$ for the numbering of the orthonormal basis, which is motivated by the explicit
 example of a constant potential where the Floquet-Bloch eigenpairs $(\psi_s(\cdot,k),\lambda_s(k))$ 
 are given by   
 \begin{equation} \label{eq:eigenpairs_constantpotential}
    \psi_s(x,k) = e^{i\skp{k+2\pi s}{x}},\quad \lambda_s(k) =|k+2\pi s|^2 
   \qquad\text{for }k\in B,s\in\Z^d,x\in \Omega,
 \end{equation}
 see (6.8.1),(6.8.2) in~\cite{Eas_TheSpectral}. So one finds that $\psi_s,\lambda_s$ are smooth with
 $\psi_s(x,k+2\pi n)=\psi_{s+n}(x,k)$, $\lambda_s(k+2\pi n)=\lambda_{s+n}(k)$ for all $n\in\Z^d$. 
 We conclude that with our choice of the index set smoothness may be gained at the expense of
 $\Z^d$-periodicity with respect to the quasimomenta~$k$. We will say more on regularity issues below.  
 
%  For further details concerning the spectral
%   theory of periodic operators we refer to Chapter~3 in~\cite{Plum_Oberwolfach} and \cite{Eas_TheSpectral}.
%  
 \medskip 
 
 The band functions $\lambda_s$ satisfy the estimates
 \begin{equation} \label{eq:lowerbounds_lambdas}
   c|s|^2-C \leq \lambda_s(k) \leq C|s|^2+C \qquad (s\in\Z^d,\, k\in B)
 \end{equation}
 for some $c,C>0$ independent of $k$. 
 Notice that $|s|^2$ has to be replaced by $|s|^{2/d}$ when $\N_0$ or $\Z$ is used as an index set. We
 quickly recall why this is true. In the case $L=-\Delta$ Theorem~6.3.1
 in~\cite{Eas_TheSpectral} shows that the $j$th largest eigenvalue among the $\lambda_s(k)$ can be enclosed
 between the $j$-th Neumann and the $j$-th Dirichlet eigenvalue. Since the 
 asymptotics for both eigenvalue sequences are given by Weyl's law, \eqref{eq:lowerbounds_lambdas} follows
 for this special case.
 For differential operators $L$ as in $(A1)$ one has $c\cdot(-\Delta)-C\leq L\leq C\cdot(-\Delta+1)$ for some
 $c,C>0$ in the sense of symmetric operators so that \eqref{eq:lowerbounds_lambdas} results from Courant's min-max
 characterization for the eigenvalues of selfadjoint compact operators and the corresponding result for $-\Delta$ mentioned
 above. For more information about the qualitative properties of the eigenpairs
 $(\psi_s(\cdot,k),\lambda_s(k))$ in a one-dimensional setting we refer to Theorem XIII.89 and Theorem
 XIII.90 in~\cite{RS_analysis_of_operators} or Chapter~2.8 in~\cite{BerShu_Schroedinger}. 
 Important tools from Floquet-Bloch analysis are the Floquet-Bloch transform  $U$ and its inverse
 $U^{-1}$ that allow to transfer problems from $\R^d$ to $k$-dependent problems on the periodicity cell
 $\Omega=(0,1)^d,k\in B=[-\pi,\pi]^d$ and vice versa. It is given by
  \begin{align} \label{eq:Floqet_Bloch_transform}
    \begin{aligned}
    U&:L^2(\R^d;\C) \to L^2(\Omega\times B;\C), \quad f\mapsto \Big[ (x,k)\mapsto \sum_{n\in\Z^d}
    f(x-n)e^{ink} \Big], \\
    U^{-1}&: L^2(\Omega\times B;\C)\to L^2(\R^d;\C), \quad g\mapsto \Big[ x\mapsto \frac{1}{\sqrt{|B|}} \int_B
    g(x,k)\,dk   \Big]
    \end{aligned}  
  \end{align}
  where $g$ is to be understood quasiperiodically extended via
  the formula $g(x+n,k)=e^{i\skp{k}{n}}g(x,k)$ \,$(n\in\Z^d)$ from $\Omega\times B$ to $\R^d\times B$, see
  Lemma~2 and Lemma~3  in \cite{OdKe_Bloch}. The Floquet-Bloch transform is an isometry, see Theorem~2.2.5
  in~\cite{Kuc_Mathematics} or Corollary~2 in~\cite{OdKe_Bloch}.

\medskip 

 With the above preparations we may now introduce and discuss the precise regularity assumptions that we 
 to impose on the Floquet-Bloch eigenpairs $(\psi_s(\cdot,k),\lambda_s(k))$ from above. In the case of the
 trivial potential, see~\eqref{eq:eigenpairs_constantpotential}, the functions
 $\Lambda:\R^d\to\R,\Psi:\Omega\times\R^d\to\C$ defined by 
 \begin{equation}\label{eq:defn_LambdaPsi}
   \Lambda(k+2\pi s):=\lambda_s(k),\qquad \Psi(x,k+2\pi s):=\psi_s(x,k) \qquad \text{for }x\in \Omega,k\in
   B,s\in\Z^d
 \end{equation}
 are real analytic  
 and the so-called Fermi surfaces (or isoenergetic surfaces)
 \begin{equation} \label{eq:defn_Ftau}
   F_\tau:=\{k\in\R^d:\Lambda(k)=\tau\} 
 \end{equation} 
 are spheres of radius $\sqrt{\tau}$ for all positive $\tau$, i.e., for all $\tau$ in the interior the
 spectrum $[0,\infty)$. In the general case, our assumption (A2) on the Fermi surfaces of the operator $L$
 will ensure that for $\tau$ close to a given frequency $\lambda\in\sigma(L)$ the associated Fermi surfaces $F_\tau$ show a
 somewhat similar behaviour. More precisely, we will require them to be compact, sufficiently smooth and to
 have positive Gaussian curvature in each point of the surface. From the physical point of view it is
 reasonable to assume that at least small periodic perturbations of constant potentials have this property,
 which we will actually prove in the two-dimensional case, see Lemma~\ref{lem:example}.
 In some textbooks and papers the term \q{Fermi surface} is used
 differently. There it is the uniquely defined subset of the Brillouin zone $B=[-\pi,\pi]^d$ that contains a
 $2\pi\Z^d$-translate of a point from $F_\tau$. In other words, it is given as follows: 
 \begin{equation} \label{eq:defn_calFtau}
   \mathcal{F}_\tau = \{ k\in B: \lambda_s(k)= \tau \text{ for some }s\in\Z^d\}.
 \end{equation}
 In the physical literature definition~\eqref{eq:defn_Ftau} is called the extended-zone scheme,
 while~\eqref{eq:defn_calFtau} corresponds to the reduced-zone scheme.  
 The following statement about the $\mathcal{F}_\tau$ is taken literally from S\'{o}lyom's book
 \cite{Solyom:Fundamentals}, page 89:  \vspace{0.2cm}
 
 \begingroup
\leftskip1em
\rightskip\leftskip
  \noindent
   \q{\ldots However, the presence of a periodic potential can drastically
 distort the spherical shape of the Fermi surface -- and, as we shall see, it can even disappear. For a
 relatively small number of electrons only the states at the bottom of the lowest-lying band are occupied.
 The Fermi surface is then a simply connected continuous surface that deviates little from the spherical
 shape. When the number of electrons is increased, the surface may cease to be simply connected \ldots In
 such cases more than one band can be partially filled. The Fermi surface separating occupied and unoccupied
 states must then be given for each of these -- hence the Fermi surface is made up of several pieces.}
\par
\endgroup \vspace{0.2cm}
%  ''\ldots However, the presence of a periodic potential can drastically
%  distort the spherical shape of the Fermi surface -- and, as we shall see, it can even disappear. For a
%  relatively small number of electrons only the states at the bottom of the lowest-lying band are occupied.
%  The Fermi surface is then a simply connected continuous surface that deviates little from the spherical
%  shape. When the number of electrons is increased, the surface may cease to be simply connected \ldots In
%  such cases more than one band can be partially filled. The Fermi surface separating occupied and unoccupied
%  states must then be given for each of these -- hence the Fermi surface is made up of several pieces.''
%    \\
 Transferred to our situation this means that for small $\tau$ one typically observes that
 $\mathcal{F}_\tau=F_\tau$ has a spherical shape. For larger $\tau$, however, $\mathcal{F}_\tau\neq F_\tau$ is
 possible and $\mathcal{F}_\tau$ may be disconnected. Indeed, this phenomenon can be easily
 verified for the constant potential $V\equiv 0$, which is again based 
 on~\eqref{eq:eigenpairs_constantpotential}.
 For $\tau>\sqrt\pi$ the sphere $F_\tau=\{k\in\R^d:\Lambda(k)=|k|^2=\tau\}$ does not fit into the Brillouin
 zone $B=[-\pi,\pi]^d$ and thus $\mathcal{F}_\tau$ becomes disconnected. It is however remarkable that the Fermi
 surfaces $F_\tau$ according to our definition from \eqref{eq:defn_Ftau} keep their shape regardless of the
 precise value of $\tau>0$. This makes us believe that, firstly, the sets $F_\tau$ are actually more
 meaningful and physical than the $\mathcal{F}_\tau$.
 Notice that the fact of the $\mathcal{F}_\tau$ becoming disconnected for $\tau>\sqrt\pi$ does not produce any
 physical effects; the Helmholtz equation $-\Delta u -\tau u = f$ for $\tau$ bigger or smaller than $\sqrt\pi$
 may be transformed into each other by a simple rescaling so that the qualitative description of the solutions
 does not change. Secondly, assuming a spherical shape in terms of positive Gaussian curvature also makes
 sense from a physical point of view. Our assumptions for the Fermi surfaces concern their shape as well as
 their regularity. 

 \begin{itemize}
   \item[(A2)] For $\Lambda,\Psi,F_\lambda$ defined in~\eqref{eq:defn_LambdaPsi},\eqref{eq:defn_Ftau} and
   an open set $U\subset\R^d$ the following holds:
   \begin{itemize}
     \item[(a)] $F_\lambda\subset U$ with 
%       $\sup_{x\in\Omega} \|\Psi(x,\cdot)\|_{C^{N-1,\beta}(\ov
%      U)}<\infty$, $\Lambda\in C^{N,\beta}(\ov U)$ for some $N\geq 3,N>\frac{d+1}{2},\beta>0$.
      $\Psi\in L^\infty(\Omega,C^{N-1,\beta}(\ov U))<\infty$, $\Lambda\in C^{N,\beta}(\ov U)$ for $N\geq
      3,N>\frac{d+1}{2},\beta>0$.
     \item[(b)] $F_\lambda$ is a closed, compact, regular hypersurface
     with positive Gaussian curvature.
  \end{itemize}
 \end{itemize}
 We stress that the regularity assumptions on the functions $\Lambda,\Psi$ are only imposed on a
 neighbourhood $U$ of the Fermi surface $F_\lambda$.
 %and only for finitely many indices $s\in\Z^d$ as follows from~\eqref{eq:lowerbounds_lambdas}.
 At first sight this seems to be a technical
 point, but in fact it is known for $d=2$ that $\Lambda$ can only be an entire function if $V$ is constant,
 see Theorem~4.4.6 in~\cite{Kuch_Floq}.
 In Lemma~\ref{lem:example} we will present a comparatively simple situation where (A2) holds. In this case
 the surface will even be analytic due to the analyticity of $\Lambda$ on $F_\lambda$. 
 Notice that the band functions $\lambda_s$ can be shown to be real analytic as long as they do not intersect, 
 see for instance Theorem~2 and Remark~(iii) in~\cite{OdKe_Bloch}.
 In~\cite{Wil_BlochWaves} Wilcox proves that for all $s\in\Z^d$ the mappings $k\mapsto \psi_s(\cdot,k)\in
 C(\ov\Omega)$ are holomorphic on $B\sm Z_s$ where $Z_s$ is a closed null set, but this regularity result
 is not sufficient for the verification of (A2). 
%  As far as the band functions $\lambda_s$ are concerned, there are global
%  regularity results that allow to continue the $\lambda_s$ analytically in a certain sense
%  (see Chapter 3.5.4 in \cite{BerShu_Schroedinger} and in particular Theorem 5.2) but we did not see how such
%  tools can be used in order to verify assumption (A2) for a larger class of operators.  
 
 \medskip
  
 In Figure~\ref{Fig:Fermi_surfaces} a few Fermi surfaces (and, for computational reasons, translates of
 it) are plotted numerically for an almost constant potential (left) and for a strongly oscillating one
 (right). The figure on the left suggest that assumption (A2)(b) is satisfied for all depicted
 frequencies~$\tau$. The Fermi surfaces on the right hand side are more complicated and for some $\tau$ more
 than one connected components of the Fermi surface can be found as well as parts with negative Gaussian
 curvature. So in this case the geometry of the Fermi surfaces does not seem to be covered by (A2)(b). The
 author thanks T.Dohnal (University of Dortmund) for providing these pictures. The regularity assumption for
 $F_\lambda$ requires $\nabla\Lambda\neq 0$ on $F_\lambda$ and frequencies $\lambda\in\sigma(L)$ with this
 property are called regular. As mentioned in Remark~2.2 of~\cite{Rad_LAP} almost all frequencies in
 $\Lambda(U)\subset \sigma(L)$ are regular. Indeed, Sard's Lemma and $\Lambda\in C^1(U)$ imply that the set of
 irregular frequencies in $\Lambda(U)$, which is $\Lambda(\{k\in U: \nabla\Lambda(k)=0\})$, is a null set.
 For the constant potential all frequencies $\tau>0$ are regular. We mention that (A2) implies that
 $F_\lambda$ is embedded (see Corollary 5.14 in \cite{Lee_Introduction}) and that all Fermi surfaces $F_\tau$
 for $\tau\approx\lambda$ also satisfy (A2)(b). Assumption (A2) will allow us to analyze the properties of
 certain integrals over the Fermi surfaces that may be interpreted as a generalized version of Herglotz
 waves, which are known to play a fundamental role in the study of Helmholtz equations, see~\cite{Ruiz_LN}
 for more details in this direction. Let us mention that Herglotz waves also appear in Guti\'{e}rrez' proof
 of the limiting absorption principle for the Helmholtz operator \cite{Gut_nontrivial} so that it may not
 surprise that such integrals are involved in our analysis. We refer to the end of
 Section~\ref{sec:proof_prop_I} for more details.
 
 \begin{figure}[!h] 
    \centering  
    \subfigure[\small{Fermi surfaces $F_\tau$ for the potential ${V(x,y)=0.2\sin(2\pi x)^2\cos(2\pi y)}$ and
    ${\tau=5,15,30,40}$ (red/black/green/blue)}]{
      \includegraphics[scale=.44]{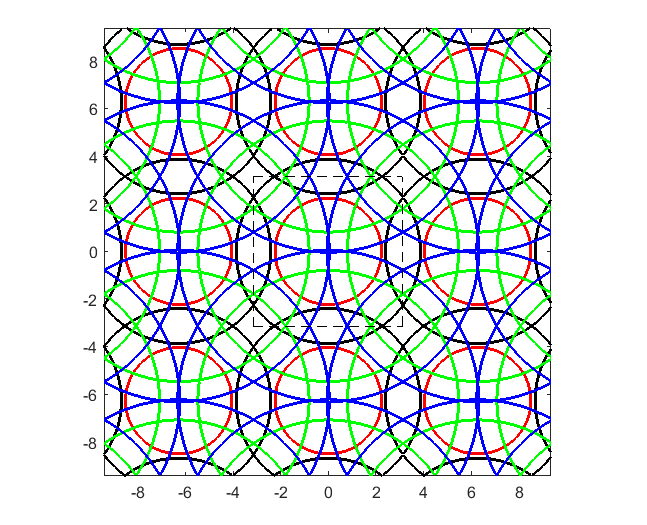}
    }   \qquad 
    \subfigure[\small{Fermi surfaces $F_\tau$ for the potential ${V(x,y)=10\sin(2\pi x)^2\cos(2\pi y)}$ and
    ${\tau=5,15,30,40}$ (red/black/green/blue)}]{ 
      \includegraphics[scale=.44]{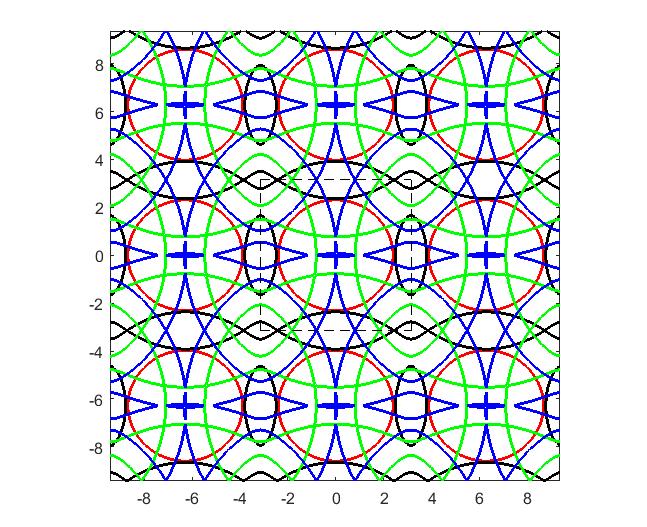} 
    } 
    \caption{Fermi surfaces}
     \label{Fig:Fermi_surfaces} 
  \end{figure} 
 
 \medskip
 
 Our last assumption concerns the eigenfunctions $\psi_s(\cdot,k)$ introduced above. 
 \begin{itemize}
   \item[(A3)] There is a $C>0$ such that $\|\psi_s(\cdot,k)\|_{L^\infty(\Omega;\C)}\leq C$ for all
   $s\in\Z^d,k\in B$.
 \end{itemize}
 Again we deduce from~\eqref{eq:eigenpairs_constantpotential} that this assumption holds for the
 constant potential, but further sufficient conditions are provided in Lemma~\ref{lem:example}.  
 
 \medskip

 We finally come to our main result, which is the limiting absorption principle for periodic differential
 operators satisfying the assumptions (A1),(A2),(A3). It is formulated in terms of the resolvent operators
 \begin{equation} \label{eq:Reps_defn}
   \mathcal{R}^\eps(\lambda):=\mathcal{R}(\lambda+i\eps) := (L-\lambda-i\eps)^{-1} \qquad
   \text{for }\eps\in\R\sm\{0\} 
 \end{equation}
 that we will consider as bounded linear operators from $L^p(\R^d)$ to $L^q(\R^d;\C)$ for $p,q$ according to
 the following inequalities:
  \begin{align}\label{eq:admissible_pq}
      \begin{aligned}
     & 1\leq p < \frac{2(d+1)}{d+3}, &&\,\,\frac{2dp}{2+p(d-3)}<q\begin{cases}
        < \frac{pd}{d-2p} &, p\leq \frac{d}{2} \\
        \leq \infty &, p>\frac{d}{2}
      \end{cases} \qquad\text{or}\\
     &\frac{2(d+1)}{d+3}\leq p<\frac{2d}{d+1}, &&\frac{2p}{2d-p(d+1)}<q\begin{cases}
        < \frac{pd}{d-2p} &, p\leq \frac{d}{2} \\
        \leq \infty &, p>\frac{d}{2}
      \end{cases}.
    \end{aligned}
  \end{align}
  An equivalent set of conditions, more in the spirit of a Riesz diagram, are given
  by~\eqref{eq:pq_nonresonant},\eqref{eq:resonant_Rieszdiagrm_conditions}. Our limiting absorption principle
  for periodic differential operators reads as follows.

  \begin{thm}\label{thm_LAP}
    Let $d\in\N,d\geq 2, p,q$ satisfy \eqref{eq:admissible_pq}  and let the assumptions (A1),(A2),(A3)
    hold for some $\lambda\in\sigma(L)$. Then the family of resolvent operators
    $\mathcal{R}^\eps(\lambda):L^p(\R^d)\to L^q(\R^d;\C)$ from~\eqref{eq:Reps_defn} is equibounded and there
    exist bounded linear operators $\mathcal{R}^\pm(\lambda):L^p(\R^d)\to L^q(\R^d;\C)$ such that 
    $$ 
      \mathcal{R}^\eps(\lambda)\to \mathcal{R}^\pm(\lambda) \quad\text{as }\eps\to 0^\pm 
    $$
    in the operator norm. For all $f\in L^p(\R^d)$ the functions
    $\mathcal{R}^\pm(\lambda)f \in W^{2,p}(\R^d;\C)+ W^{2,q}(\R^d;\C)$ 
    are strong solutions of $Lu-\lambda u = f$ in $\R^d$.
  \end{thm}
  
  %\r{Riesz diagram, globale Regularit\"at nicht zufriedenstellend, aber cald�ron-Zygmund theorie nicht da}
  
  Additionally, the functions $\mathcal{R}^\pm(\lambda)f$ are expected to satisfy a generalized form  
  of Sommerfeld's radiation condition at infinity. Similarly, the farfield expansions of these functions are
  of interest and generalized versions of the corresponding results for constant potentials (see for instance
  Proposition~2.7 and Proposition~2.8 in \cite{EvWe_dual}) are expected to hold. We
  mention that our conditions on~$p,q$ are probably non-optimal given that Guti\'{e}rrez' limiting absorption principle from
  \cite{Gut_nontrivial} holds for the larger range of exponents $p,q\in [1,\infty]$ satisfying the
  inequalities
  \begin{equation} \label{eq:pq_Gutierrez}
    \frac{1}{p}> \frac{d+1}{2d}, \qquad \frac{1}{q}<\frac{d-1}{2d},\qquad \frac{2}{d+1}\leq
    \frac{1}{p}-\frac{1}{q}\leq \frac{2}{d}. 
  \end{equation}
  The reason for this comes from a different interpolation procedure for the resonant part of the
  resolvent operators $\mathcal{R}^\eps(\lambda)$, as we will see later. Notice that the crucial estimates for
  the Helmholtz operator are based on the Stein-Tomas theorem and, as far as we know, no equivalent of this
  result is known in the context of Floquet-Bloch theory. We hope that future research will make it possible
  to extend our results to all exponents satisfying~\eqref{eq:pq_Gutierrez}. 
  
%   \medskip
%   
%   Another open problem concerns the natural conjecture that the operators 
%   $\mathcal{R}^\pm(\lambda)$ are not only bounded from $L^q(\R^d)$ to $L^p(\R^d)$ but even from $L^q(\R^d)$ to
%   $W^{2,p}(\R^d)$. We expect this result to be true for two reasons: Firstly, Guti\'{e}rrez' as well as our
%   bounds allow for exponents $p$ up to $\frac{dq}{d-2q}$, which is the exponent coming from Sobolev's
%   imbedding $W^{2,q}(\R^d)\to L^{dq/(d-2q)}(\R^d)$. Secondly, at least in a radial setting it is known that
%   for any given solution of a prototype nonlinear Helmholtz equation the first and second order derivatives
%   have the same decay properties as the function itself (see Theorem~1.2~(iii) in~\cite{MaMoPe_oscillating}).
%   Hence, $u\in W^{2,q}(\R^d)$ is a natural guess. Unfortunately, we did not see a way how to prove such a
%   result with the methods presented here.
  
  \medskip
  
  Finally, we discuss an application of  the limiting absorption principle from Theorem~\ref{thm_LAP}. We
  study real-valued solutions of the nonlinear Helmholtz equation
  \begin{equation} \label{eq:NLH}
      Lu-\lambda u = \pm\Gamma(x)|u|^{q-2}u \quad\text{in }\R^d
  \end{equation} 
  where $L,\lambda$ satisfy the assumptions of the theorem and $\Gamma\in L^\infty(\R^d)$ is a positive
  $\Z^d-$periodic function. In the case $L=-\Delta$ and $\lambda>0$ Evequoz and Weth \cite{EvWe_dual} showed
  that \eqref{eq:NLH} admits a dual variational formulation in $L^{q'}(\R^d)$ for
  $\frac{2(d+1)}{d-1}\leq q\leq \frac{2d}{d-2}$ that relies on the selfdual estimates for the
  associated resolvent operators $\mathcal{R}^\pm(\lambda):L^{q'}(\R^d)\to L^q(\R^d;\C)$.
  For $q$ in the interior of this interval they proved the
  existence of a mountain pass critical point in $L^{q'}(\R^d)$ of the associated dual functional and thus the
  existence of a nontrivial dual ground state solution of \eqref{eq:NLH} belonging to $L^q(\R^d)$. This
  solution even lies in $W^{2,r}(\R^d)$ for all $r\in [q,\infty)$.
  One of the major limitations in their approach is the specific form of the linear operator $L$, which is due to
  the fact that only in this case the mapping properties of the resolvent-type operators
  $\mathcal{R}^\pm(\lambda)$ are known (thanks to Guti\'{e}rrez' results we mentioned above). We refer to the
  beginning of Section~2 in~\cite{EvWe_dual} for the details. Given that the selfdual estimates 
  $\mathcal{R}^\pm(\lambda):L^{q'}(\R^d)\to L^q(\R^d;\C)$ from Theorem~\ref{thm_LAP} hold for
  $\frac{2(d+1)}{d-1}< q<\frac{2d}{d-2}$, we may apply the same variational techniques provided the linear
  operator satisfies (A1),(A2),(A3).
  
  \begin{cor}\label{cor_NLH} 
    Let  $d\in\N,d\geq 2, \frac{2(d+1)}{d-1}<q<\frac{2d}{d-2}$ and let the assumptions (A1),(A2),(A3)
    hold for some $\lambda\in\sigma(L)$, let $\Gamma\in L^\infty(\R^d)$ be positive and $\Z^d-$periodic. Then
    the nonlinear Helmholtz equation~\eqref{eq:NLH} has a nontrivial solution $u\in
    W^{2,r}(\R^d)$ for all $r\in [q,\infty)$.
  \end{cor}
  
  As in \cite{Eveq_plane,EvWe_dual} the existence of infinitely many nontrivial solutions may be shown by
  invoking the Symmetric Mountain Pass Theorem under the assumption that $\Gamma$ is evanescent at infinity so that the
  associated dual functional satisfies the Palais-Smale condition, cf. Lemma~5.2 in~\cite{EvWe_dual} for the
  case $d\geq 3$ and p.10 in \cite{Eveq_plane} for the case $d=2$. Finally, we provide a class of nontrivial
  periodic operators $L$ and frequencies $\lambda\in\sigma(L)$ for which the assumptions (A1),(A2),(A3) hold
  so that Theorem~\ref{thm_LAP} and Corollary~\ref{cor_NLH} apply.
  
  \begin{lem}\label{lem:example}
     Let $d=2,\mu_1,\mu_2\in\R$ and $\eps>0$. Then there is a $\delta>0$ such that the two-dimensional
     Schr\"odinger operator $L=-\Delta+V_1(x_1)+V_2(x_2)$ satisfies (A1),(A2),(A3) at all frequencies
     $\lambda\in (\mu_1+\mu_2+\eps,\mu_1+\mu_2+\pi^2-\eps)$ provided $V_1,V_2$ are $1$-periodic, piecewise
     continuous with $\|V_1-\mu_1\|_{L^\infty([0,1])},\|V_2-\mu_2\|_{L^\infty([0,1])}<\delta$.  
  \end{lem}
  
%   We mention that 
%   our statement concerning the global regularity of the solution is weaker than the corresponding claims
%   in~\cite{EvWe_dual,Eveq_plane} because a replacement for the bootstrap procedure from Theorem~4.4
%   in~\cite{EvWe_dual} has not been established yet.
  
  \medskip
  
  The paper is organized as follows: In Section~\ref{sec:estimates} we analyze the mapping properties of the
  resolvent operators $\mathcal{R}^\eps(\lambda)$ for $\eps\in\R\sm\{0\}$ and identify the limit operators
  $\mathcal{R}^\pm(\lambda)$ as $\eps\to 0^\pm$. This will be done by splitting $\mathcal{R}^\eps(\lambda)$
  into a nonresonant and a resonant part the analysis of which is substantially different. We mention that this splitting already
  appears in the work of Radosz~\cite{Rad_Diss,Rad_LAP}.   
  The estimates from Section~\ref{sec:estimates} will then be used in Section~\ref{sec:proofs} where
  Theorem~\ref{thm_LAP} and Corollary~\ref{cor_NLH} are proved. Two results from Section~\ref{sec:estimates}
  with long and technical proofs will be discussed in Section~\ref{sec:proof_prop_I} and
  Section~\ref{sec:proof_prop_II}. In Section~\ref{sec:Proof_Lemma} we finally prove Lemma~\ref{lem:example}.
  Throughout the paper $c,C>0$ will denote positive numbers that may change from line to line.

\section{Estimates} \label{sec:estimates}
  
  Throughout this section we make use of the assumptions of Theorem~\ref{thm_LAP}. Following the strategy
  outlined above we split up the resolvent operators
  according to $\mathcal{R}^\eps(\lambda)= \mathcal{R}^\eps_1(\lambda)+
  \mathcal{R}_2^\eps(\lambda)$ where $\mathcal{R}^\eps_1(\lambda),\mathcal{R}^\eps_2(\lambda)$ define
  linear and bounded operators between appropriate Lebesgue spaces that converge as $\eps\to 0^\pm$. 
%   Here the word
%   equibounded refers to the fact that the operator norms of these operators remain bounded as $\eps$ tends to
%   zero, which will later enable us to define resolvent-type operator
%   $\mathcal{R}^+(\lambda),\mathcal{R}^-(\lambda)$ as $\eps\to 0^\pm$. 
  In order to prove this assertion we first provide a representation formula for the resolvent using the
  eigenfunction expansion for the eigenvalue problems~\eqref{eq:FB_ev_problem} on the periodicity cell
  $\Omega=(0,1)^d$. With the aid of the Floquet-Bloch transform and the notation from  the first section
  we get the following result.
   
 \begin{prop}\label{prop:representation_formula}
   Let $\eps\in\R\sm\{0\}$. Then for all $f\in C_0^\infty(\R^d)$  we have 
   $$
     (\mathcal{R}^\eps(\lambda)f)(x)= \int_{\R^d} K^\eps(x,y)f(y)\,dy
   $$
   where the kernel function $K^\eps\in L^2_{loc}(\R^d\times\R^d;\C)$ is given by
  \begin{equation} \label{eq:def_Keps}
    K^\eps(x,y) 
    =  \avint_B\sum_{s\in\Z^d} \frac{\psi_s(x,k)\ov{\psi_s(y,k)}}{\lambda_s(k)-\lambda-i\eps}  \,dk.
    %= \sum_{s\in\Z^d}  |B|^{-1} \int_B  \frac{\psi_s(x,k)\ov{\psi_s(y,k)}}{\lambda_s(k)-\lambda-i\eps}  \,dk,
  \end{equation} 
 \end{prop}
 \begin{proof}
   We set $u^\eps:=\mathcal{R}^\eps(\lambda)f$. Then $u^\eps\in H^2(\R^d;\C)$
   satisfies 
   $$
     L u^\eps  - (\lambda+i\eps)u^\eps = f \quad\text{in }\R^d
   $$
   in the strong sense.  Now we apply the Floquet-Bloch transform which commutes with the differential
   operator $L$ thanks to periodicity assumption (A1). So for all $k\in B$ the function
   $Uu^\eps(\cdot,k)\in H^2(\Omega;\C)$ solves the Floquet-Bloch boundary value
   problem~\eqref{eq:FB_ev_problem}. Since $(\psi_s(\cdot,k))_{s\in\Z^d}$ is an orthonormal basis in
   $L^2(\Omega;\C)$ consisting of eigenfunctions for this problem with eigenvalues $\lambda_s(k)$, we get for
   all $k\in B$ and almost all $x\in \Omega$
   \begin{align} \label{eq:resolvent+Floquet}
     (Uu^\eps)(x,k)
     = \sum_{s\in\Z^d}\frac{\skp{(Uf)(\cdot,k)}{\psi_s(\cdot,k)}_{L^2(\Omega;\C)}}{\lambda_s(k)-\lambda-i\eps} 
     \psi_s(x,k).
   \end{align}
   Notice that for every given $k\in B$  this series converges in $L^2(\Omega;\C)$
   thanks to~\eqref{eq:lowerbounds_lambdas}. Since $f$ has compact support, we get
   \begin{align} \label{eq:Umformung}
    \begin{aligned}
      \skp{Uf(\cdot,k)}{\psi_s(\cdot,k)}_{L^2(\Omega;\C)}
      &= \int_{\Omega} Uf(y,k)\ov{\psi_s(y,k)} \,dy \\
      &= |B|^{-1/2} \int_{\Omega} \sum_{n\in\Z^d}  f(y+n)e^{-i\skp{k}{n}} \ov{\psi_s(y,k)} \,dy  \\
      &= |B|^{-1/2} \sum_{n\in\Z^d} \int_{\Omega} f(y+n)\ov{\psi_s(y+n,k)} \,dy  \\
      &= |B|^{-1/2} \int_{\R^d} f(y)\ov{\psi_s(y,k)} \,dy,
    \end{aligned} 
  \end{align}
  and thus  
  \begin{align*}
    (Uu^\eps)(x,k)
     &=    \sum_{s\in\Z^d}\frac{|B|^{-1/2} \int_{\R^d}
     f(y)\ov{\psi_s(y,k)} \,dy}{\lambda_s(k)-\lambda-i\eps} \psi_s(x,k)   \\
     &=   |B|^{-1/2} \int_{\R^d} \sum_{s\in\Z^d}\frac{\psi_s(x,k)\ov{\psi_s(y,k)}
     }{\lambda_s(k)-\lambda-i\eps} f(y) \,dy  
  \end{align*}
  for all $k\in B$ and almost all $x\in \Omega$. Finally, we apply the inverse
  Floquet-Bloch transform given by \eqref{eq:Floqet_Bloch_transform} and get from Fubini's Theorem
   \begin{align*}
     u^\eps(x)
     &= |B|^{-1/2}\int_B Uu^\eps(x,k) \,dk \\
     &= \int_{\R^d} \left( \avint_B 
      \sum_{s\in\Z^d}\frac{\psi_s(x,k)\ov{\psi_s(y,k)}
     }{\lambda_s(k)-\lambda-i\eps}   \,dk \right) f(y) \,dy \\
     &= \int_{\R^d} K^\eps(x,y)f(y)\,dy,
   \end{align*}
   which is all we had to show.
 \end{proof}

  We note that an explicit formula for $K^\eps$ does not seem to be available
  except for the special case of the Helmholtz operator $L-\lambda=-\Delta-\lambda$ for $\lambda>0$,
  see~\eqref{eq:upm_formula}. The representation formula from
  Proposition~\ref{prop:representation_formula} in fact holds for more general functions~$f$.
  Based on estimates involving $K^\eps$ we will see that the integral representation for
  $\mathcal{R}^\eps(\lambda)f$ also makes sense for $f\in L^p(\R^d)$ if $p$ is chosen suitably. To see
  this, we split the sum and the integration into one part where $\lambda_s(k)-\lambda$ is bounded away from
  zero and a second part where $\lambda_s(k)-\lambda$ is close to zero. We will call the associated operators
  the nonresonant part (indexed by~1) or the resonant part (indexed by~$2$) of the resolvent, respectively.
  For $\rho>0$ such that $F_\tau\subset U$ for $|\tau-\lambda|\leq \rho$ we can choose a cutoff function
  $\chi\in C_0^\infty(\R^d)$ with the properties
  \begin{equation}\label{eq:Def_chi}
    0\leq \chi\leq 1,\qquad \supp(\chi)\subset B_\rho(0),\qquad  \chi\equiv 1 \text{ on }B_{\rho/2}(0).
 \end{equation}  
  In the proof of Proposition~\ref{prop:resonant_pointwise_decay} it will become clear, how $\rho$ should be
  chosen depending only on the geometry of the Fermi surfaces $F_\tau$ for $\tau\approx\lambda$.  
  Then the splitting  
 $$
   \mathcal{R}^\eps(\lambda) f
   = \mathcal{R}^\eps_{1}(\lambda)f + \mathcal{R}^{\eps}_{2}(\lambda)f 
 $$
 holds for $f\in C_0^\infty(\R^d)$ where the operators on the right hand side are  defined via
 \begin{align}\label{eq:Reps12_defn}
   \begin{aligned}
   (\mathcal{R}^\eps_{1}(\lambda)f)(x)  &:=  \int_{\R^d} K^\eps_{1}(x,y) f(y) \,dy, \\
    (\mathcal{R}^\eps_{2}(\lambda)f)(x)  &:= \int_{\R^d} K^\eps_{2}(x,y)f(y)\,dy
    \end{aligned} 
  \end{align}
  and $K^\eps_{1},K^\eps_2\in L^2_{loc}(\R^d\times\R^d;\C)$ are given by 
  \begin{align} \label{eq:def_Keps_12}
    \begin{aligned}
    K^\eps_{1}(x,y)&:=  \avint_B \sum_{s\in\Z^d}
    \frac{1-\chi(\lambda_s(k)-\lambda)}{\lambda_s(k)-\lambda-i\eps}  \psi_s(x,k)\ov{\psi_s(y,k)}\,dk, \\
    K^\eps_{2}(x,y) 
    &:=  \avint_B \sum_{s\in\Z^d} 
    \frac{\chi(\lambda_s(k)-\lambda)}{\lambda_s(k)-\lambda-i\eps}  \psi_s(x,k)\ov{\psi_s(y,k)}\,dk.
    \end{aligned} 
  \end{align}
  In view of \eqref{eq:lowerbounds_lambdas} we find that $K^\eps_{2}(x,y)$ should be seen as a finite
  sum of singular terms (as $\eps\to 0^\pm$) whereas $K^\eps_{1}(x,y)$ is an infinite series of regular terms.
  Moreover, we observe
  $$
    K_j^\eps(x,y)=\ov{K_j^{-\eps}(y,x)},\quad K_j^\eps(x+m,y)=K_j^\eps(x,y-m) \quad\text{for }
    \eps\in\R\sm\{0\},x,y\in\R^d,m\in\Z^d \;\;(j=1,2).
  $$
  Using the assumptions (A1),(A2),(A3) we will show that, roughly speaking, the resonant part is responsible
  for low decay rates at infinity because it maps into Lebesgue spaces $L^q(\R^d;\C)$ with certain
  exponents $q>2$. On the contrary the nonresonant part will give the upper bound for $q$ from
  \eqref{eq:admissible_pq}. In the following we study the mapping properties of
  $\mathcal{R}^\eps_{1}(\lambda),\mathcal{R}^\eps_{2}(\lambda)$ for small $|\eps|$ that will be used in
  Section~\ref{sec:proofs} when we prove Theorem~\ref{thm_LAP} and Corollary~\ref{cor_NLH}.
  
\subsection{Estimates for the nonresonant part}

  Using the equiboundedness of the eigenfunctions $\psi_s$ from assumption (A3) we first prove an
  estimate for the family of sequences $(\skp{Uh(\cdot,k)}{\psi_s(\cdot,k)}_{L^2(\Omega;\C)})_{s\in\Z^d}$
  where $k$ ranges over the Brillouin zone $B$ and $h\in L^{r'}(\R^d)$ for some $r\in [2,\infty]$. For notational
  convenience we suppress $k$ as well as the index $s\in\Z^d$ of these sequences. These estimates involve the
  Banach spaces $L^r(B\times \Z^d)$ for $r\in [2,\infty]$ which we define to be the Lebesgue space with
  exponent~$r$ induced by the product of the Lebesgue measure on $B\subset\R^d$ and the counting measure on $\Z^d$. The corresponding norm is given by
  \begin{align*}
    \|(\skp{Uh}{\psi_s}_{L^2(\Omega;\C)})\|_{L^r(B\times\Z^d;\C)}
    %&:= \Big(\int_B   \|(\skp{Uh(\cdot,k)}{\psi_s(\cdot,k)}_{L^2(\Omega;\C)})\|_{l^r(\Z^d)}^r \,dk
    %\Big)^{1/r}   \\
    &:= \Big( \int_B \sum_{s\in\Z^d} |\skp{Uh(\cdot,k)}{\psi_s(\cdot,k)}_{L^2(\Omega;\C)}|^r \,dk \Big)^{1/r}
  \\
  \intertext{for $2\leq r<\infty$ and}
    \|(\skp{Uh}{\psi_s}_{L^2(\Omega;\C)})\|_{L^\infty(B\times\Z^d;\C)}
    &:= \sup_{(k,s)\in B\times\Z^d} |\skp{Uh(\cdot,k)}{\psi_s(\cdot,k)}_{L^2(\Omega;\C)}|.
  \end{align*}
  Here, $\sup$ stands for the essential supremum. In view of \eqref{eq:Umformung} the following result
  resembles the Hausdorff-Young inequality for Fourier series.

  \begin{prop}\label{prop:Hausdorff_Young}
    There is a $C>0$ such that for all $k\in B$ and $2\leq r\leq \infty$ and $h\in L^{r'}(\R^d)$
    $$
      \|(\skp{Uh}{\psi_s}_{L^2(\Omega;\C)})\|_{L^r(B\times\Z^d;\C)}
      \leq C \|h\|_{L^{r'}(\R^d)}.
    $$
  \end{prop}
  \begin{proof}
    For $r=2$ we have the identity
    \begin{align*}
      \|(\skp{Uh}{\psi_s}_{L^2(\Omega;\C)})\|_{L^2(B\times\Z^d;\C)}^2
      %&= \int_B \|(\skp{Uh(\cdot,k)}{\psi_s(\cdot,k)}_{L^2(\Omega;\C)})\|_{l^2(\Z^d)}^2 \,dk \\
      &= \int_B \sum_{s\in\Z^d} |\skp{Uh(\cdot,k)}{\psi_s(\cdot,k)}_{L^2(\Omega;\C)}|^2 \,dk \\
      &=\int_B  \|Uh(\cdot,k)\|_{L^2(\Omega;\C)}^2 \,dk \\
      &= \|Uh\|_{L^2(\Omega\times B;\C)}^2 \\
      &= \|h\|_{L^2(\R^d)}^2.
    \end{align*}
    Here we used that the functions $\psi_s(\cdot,k)$ form an orthonormal basis in $L^2(\Omega;\C)$ and that
    the Floquet-Bloch transform $U:L^2(\R^d;\C)\to L^2(\Omega\times B;\C)$ is an isometry. In
    the proof of the inequality for $r=\infty$ we use (A3), so let $C>0$ be given with $|\psi_s(x,k)|\leq C$
    for all $x\in\Omega,k\in B,s\in\Z^d$. Then we get from
    \eqref{eq:Floqet_Bloch_transform}
    \begin{align*}
       \|(\skp{Uh}{\psi_s})\|_{L^\infty(B\times\Z^d;\C)}
      &= \sup_{(k,s)\in B\times \Z^d} |\skp{Uh(\cdot,k)}{\psi_s(\cdot,k)}_{L^2(\Omega;\C)}| \\
      &\leq C  \sup_{(k,s)\in B\times \Z^d} \int_\Omega |Uh(x,k)|\,dx \\
      &\leq C \int_\Omega \sum_{n\in\Z^d} |h(x-n)|\,dx \\
      &\leq C \|h\|_{L^1(\R^d)}.
    \end{align*}
    Interpolating both estimates yields the result.
  \end{proof}

  Next we use the estimates from Proposition~\ref{prop:Hausdorff_Young} to prove some mapping properties of
  the nonresonant part of the resolvent operator.

  \begin{lem}\label{lem:mapping_properties_nonresonant}
    Let $p,q$ satisfy
    \begin{align} \label{eq:pq_nonresonant}
      d\geq 2 \quad\text{and}\quad 0\leq \frac{1}{p}-\frac{1}{q}<\frac{2}{d},\quad 1\leq p\leq 2\leq
      q\leq \infty. 
%       \qquad\text{or} \\
%       &d=2 \quad\text{and}\quad 0\leq \frac{1}{q}-\frac{1}{p}<1,\quad 1\leq q\leq 2\leq
%       p\leq \infty.  
    \end{align} 
    Then there is a $C>0$ such that for all $\eps\in\R$ and $f\in C_0^\infty(\R^d)$ the following estimates
    hold
    \begin{align*}
      \|\mathcal{R}^\eps_{1}(\lambda)f\|_{L^q(\R^d;\C)}  
      &\leq C\|f\|_{L^p(\R^d)}, \\
      \|\mathcal{R}^\eps_{1}(\lambda)f-\mathcal{R}^0_{1}(\lambda)f\|_{L^q(\R^d;\C)} 
      &\leq C\eps\|f\|_{L^p(\R^d)}.
    \end{align*}
  \end{lem}
  \begin{proof}
    Applying \eqref{eq:Umformung} to $f\in C_0^\infty(\R^d)$ we get
    \begin{align*}
      &\int_{\R^d} g(x) (\mathcal{R}^\eps_{1}(\lambda)f)(x)\,dx \\ 
      &= \int_{\R^d} g(x) \Big(\int_{\R^d} K^\eps_{1}(x,y)f(y)\,dy\Big)\,dx \\
      &=  \avint_B \sum_{s\in\Z^d} \alpha_s^\eps(k) 
      \Big(\int_{\R^d} f(y)\ov{\psi_s(y,k)} \,dy\Big)   \Big(\int_{\R^d} g(x)\psi_s(x,k)  \,dx\Big)  \,dk   \\
      &=  \int_B \sum_{s\in\Z^d} \alpha_s^\eps(k) 
      \skp{Uf(\cdot,k)}{\psi_s(\cdot,k)}_{L^2(\Omega;\C)} 
      \overline{\skp{Ug(\cdot,k)}{\psi_s(\cdot,k)}_{L^2(\Omega;\C)}} \,dk 
    \end{align*}
    where  
    $$
      \alpha_s^\eps(k) := \frac{1-\chi(\lambda_s(k)-\lambda)}{\lambda_s(k)-\lambda-i\eps}.
    $$
    Now let $r>\frac{d}{2}$ be given by $\frac{1}{p}-\frac{1}{q}=\frac{1}{r}$.     
    Due to \eqref{eq:lowerbounds_lambdas} we find a $C>0$ such that for all $\eps\in\R$  we have 
    $\|(\alpha_s^\eps)\|_{L^r(B\times\Z^d;\C)}\leq C$. So  H\"older's inequality and
    Proposition~\ref{prop:Hausdorff_Young} (we have $p',q\geq 2$) yield
    \begin{align*}
      &\int_B \sum_{s\in\Z^d} \big| 
      \alpha_s^\eps(k) \skp{Uf(\cdot,k)}{\psi_s(\cdot,k)}_{L^2(\Omega;\C)} 
      \overline{\skp{Ug(\cdot,k)}{\psi_s(\cdot,k)}_{L^2(\Omega;\C)}}\big|\,dk \\
      &\leq  \|(\alpha_s^\eps)\|_{L^r(B\times\Z^d;\C)}
      \|(\skp{Uf}{\psi_s}_{L^2(\Omega;\C)})\|_{L^{p'}(B\times\Z^d;\C)}
      \|(\skp{Ug}{\psi_s}_{L^2(\Omega;\C)})\|_{L^q(B\times\Z^d;\C)}  \\
      &\leq C \|g\|_{L^{q'}(\R^d)}\|f\|_{L^p(\R^d)}.
    \end{align*}
    This entails
    $$
       \Big|\int_{\R^d} g(x) (\mathcal{R}^\eps_{1}(\lambda)f)(x)\,dx\Big| 
       \leq   C \|g\|_{L^{q'}(\R^d)}\|f\|_{L^p(\R^d)} 
    $$
    for all $f,g\in C_0^\infty(\R^d)$. The same estimates hold for $g$ replaced by $i g$. Therefore, since
    $C_0^\infty(\R^d;\C)$ is dense in $L^{q'}(\R^d;\C)$ and the dual of $L^{q'}(\R^d;\C)$ is $L^q(\R^d;\C)$
    because of $q\geq 2$, we get the first asserted estimate. The same way we get the second estimate from the
    bound $\|(\alpha_s^\eps-\alpha_s^0)\|_{L^r(B\times\Z^d;\C)}\leq C\eps$.
\end{proof}

\subsection{Estimates for the resonant part}
 
  Now we discuss  the mapping properties of the integral operator 
  $$
    (\mathcal{R}^\eps_{2}(\lambda)f)(x) = \int_{\R^d} K^\eps_{2}(x,y)f(y)\,dy
  $$ 
  where $K^\eps_{2}$ was defined in~\eqref{eq:def_Keps_12}. Our first result is a
  pointwise estimate for the kernel function, which is the most difficult result in this paper. 
  Its proof is based on a refinement of the method of (non-)stationary phase and its
  application to decay estimates for oscillatory integrals over nicely curved hypersurfaces in $\R^d$, namely
  the Fermi surfaces $F_\lambda$ described by assumption (A2). 
    
 \begin{prop} \label{prop:resonant_pointwise_decay}
   There are $\rho>0$ and measurable functions $K_2^\pm:\R^d\times\R^d\to\C$ and a $C>0$ such that for all 
   $\eps\in\R\sm\{0\},x,y\in\R^d,m\in\Z^d$ we have $K_2^\pm(x,y)=\ov{K_2^\mp(y,x)}$,
   $K^\pm(x+m,y)=K^\pm(x,y-m)$ as well as
   \begin{align}\label{eq:resonant_pointwise_est}
     \begin{aligned}
     |K_2^\eps(x,y)| &\leq  C (1+|x-y|)^{\frac{1-d}{2}},\\
     |K_2^{\eps}(x,y)-K_2^\pm(x,y)| &\leq C \eps^\beta (1+|x-y|)^{\frac{1-d}{2}} \quad\text{as }\eps\to
     0^\pm.
     \end{aligned}
   \end{align}
 \end{prop}
 
 The proof of Proposition
 \ref{prop:resonant_pointwise_decay} is very long, so we prefer to present it later in
 Section~\ref{sec:proof_prop_I}. The estimate \eqref{eq:resonant_pointwise_est} already yields some mapping
 properties of $\mathcal{R}^\eps_{2}(\lambda)$ between Lebesgue spaces, but those are not strong enough to
 prove Theorem~\ref{thm_LAP}. As in the
 proof of Theorem~2.2 in~\cite{KeRuSo_uniform} or Theorem~6 in~\cite{Gut_nontrivial} an estimate based on
 spectral properties has to be added in order to improve them via interpolation, i.e., with the aid of the
 Riesz-Thorin Interpolation Theorem. In~\cite{KeRuSo_uniform, Gut_nontrivial} this strategy
 applies in the context of elliptic differential operators with constant coefficients. For instance in the
 case $L=-\Delta-1$ one finds that the kernel function associated with the operator $L-i\eps$ is given by
 $K^\eps(x,y)=\Phi^\eps(x-y)$ with $\mathcal{F}(\Phi^\eps)(\xi)=(|\xi|^2-1-i\eps)^{-1}$. The estimates for
 the nonresonant part $(||\xi|^2-1|\geq c>0)$ of the associated integral operator are based on Bessel
 potential estimates -- their counterpart in the periodic setting was presented in the previous section.
 The resonant part  of the kernel function $K_2^\eps(x,y)$, which corresponds to $||\xi|^2-1|\leq c$ in the
 case of the Helmholtz operator, is estimated differently.
 It is split into infinitely many pieces $K^{\eps,j}_{2}$ that only depend on the
 behaviour of $K^\eps_2$ in the dyadic annuli $2^{j-1}\leq |x-y|<2^j$ for $j\in\N$. The $j$-dependent mapping
 properties of these infinitely many integral operators result from the pointwise decay of $K^\eps_2$ and
 from estimates based on the Stein-Tomas theorem, see for instance (36) in \cite{Gut_nontrivial} for the
 decomposition into annular regions and Lemma~1 in \cite{Gut_nontrivial} for the resulting $j$-dependent
 estimates on these regions. For the Stein-Tomas theorem we refer to \cite{Tom_restriction} or p.375,p.414 for $d\geq 3,d=2$ in
 \cite{Ste_harmonic_analysis}.
 
 \medskip
 
 In the case of general periodic elliptic differential operators the Fourier transform is not suitable and a
 replacement for the above-mentioned estimates has to be found. In our
 situation it turns out that estimates for the Floquet-Bloch transforms (similar to the ones in the
 paper~\cite{Kach_ST_theorem}) of $K^{\eps}_2(x,y)$ for $(x,y)$ in the $j$th dyadic shell are helpful. These
 dyadic shells should be seen as the the analogues of the annuli used in the constant coefficient case. More precisely, we define the grid
 points $R_0:=\{0\},R_j:= \{m\in\Z^d : 2^{j-1}\leq |m_i|< 2^j \text{ for }i=1,\ldots,d\}$ and then, for each
 $j\in\N_0$ and $\eps\in\R\sm\{0\}$,
  \begin{align} \label{eq:def_Repsj}
    \begin{aligned}
    (\mathcal{R}^{\eps,j}_{2}(\lambda)f)(x) 
    &:= \int_{\R^d} K^{\eps,j}_{2}(x,y)f(y)\,dy 
    \qquad\text{where } \\
    K^{\eps,j}_{2}(x,y) &:=  K^\eps_{2}(x,y)1_{R_j}([x]-[y]). 
  \end{aligned}
  \end{align}
  Here, $[x]:=([x_1],\ldots,[x_d])\in\Z^d$ denotes the componentwise floor function. This definition
  assures $K^{\eps,j}_{2}(x+m,y)=K^{\eps,j}_{2}(x,y-m)$ for all $x,y\in\R^d,m\in\Z^d$ so that
  $K^{\eps,j}_2$ inherits this important symmetry property from $K^\eps_{2}$. Analogously, we define 
  \begin{align} \label{eq:def_Rpmj}
    \begin{aligned}
    (\mathcal{R}^{\pm,j}_{2}(\lambda)f)(x) 
    &:= \int_{\R^d} K^{\pm,j}_{2}(x,y)f(y)\,dy 
    \qquad\text{where } \\
    K^{\pm,j}_{2}(x,y) &:=  K^\pm_{2}(x,y)1_{R_j}([x]-[y]).
  \end{aligned}
  \end{align}
  First we provide the
  estimates based on the pointwise bounds from Proposition~\ref{prop:resonant_pointwise_decay}.
  
\begin{prop}\label{prop:L1_linfty_estimate}
  There is a $C>0$ such that we have for all $\eps\in\R\sm\{0\}$ and $1\leq p\leq q\leq \infty$ and all $f\in
  L^p(\R^d)$
  \begin{align*}
    \|\mathcal{R}^{\eps,j}_{2}(\lambda)f\|_{L^q(\R^d;\C)} 
    &\leq C 2^{j(\frac{1+d}{2}+\frac{d}{q}-\frac{d}{p})} \|f\|_{L^p(\R^d)} &&\text{for all } j\in\N_0
    \text{ and } \\
    \|\mathcal{R}^{\eps,j}_{2}(\lambda)f-\mathcal{R}^{\pm,j}_{2}(\lambda)f\|_{L^q(\R^d;\C)} 
    &\leq C \eps^\beta 2^{j(\frac{1+d}{2}+\frac{d}{q}-\frac{d}{p})} \|f\|_{L^p(\R^d)} &&\text{for all }
    j\in\N_0 \text{ as }\eps\to 0^\pm.
  \end{align*}
\end{prop}
\begin{proof} 
  We only show the first estimate, the proof of the second being similar.
  For $x,y\in\R^d$ such that $[x]-[y]\in R_j$ we have the inequality $c\cdot 2^j\leq |x-y|\leq C\cdot 2^j$ for
  some positive $c,C$. In particular, Proposition~\ref{prop:resonant_pointwise_decay} gives 
  $$
    |K^{\eps,j}_{2}(x,y)|\leq  C  (1+|x-y|)^{\frac{1-d}{2}} 1_{|x-y|\leq C 2^j} 
    \leq C 2^{\frac{j(1-d)}{2}} 1_{|x-y|\leq C 2^j} \qquad\text{for all } j\in\N_0.
  $$ 
  Hence, Young's convolution inequality yields the desired estimate.
\end{proof}

  We continue with an $L^2-L^2-$estimate for $\mathcal{R}^{\eps,j}_{2}(\lambda)$ 
  based on a pointwise estimate of the Floquet-Bloch transform of the kernel function
  $K^{\eps,j}_{2}(\cdot,y)$ which relies on the regularity assumptions for the Fermi surfaces from
  assumption (A2). Since it is quite long, we defer the proof to Section~\ref{sec:proof_prop_II}.  

\begin{prop} \label{prop:estimate_FBtransform_kernel}
  For all  $\delta>0$ there is a $C_{\delta}>0$ such that for all $\eps\in\R\sm\{0\}$ we have
  \begin{align*}
    \sup_{x,y\in\Omega, l\in B} \big| U(K^{\eps,j}_{2}(\cdot,y))(x,l)\big| 
    &\leq C_{\delta} 2^{j(1+\delta)}  &&\text{for all } j\in\N_0 \text{ and } \\
    \sup_{x,y\in\Omega, l\in B} \big| U(K^{\eps,j}_{2}(\cdot,y)-K^{\pm,j}_2(\cdot,y))(x,l)\big| 
    &\leq C_\delta \eps^\beta 2^{j(1+\delta)}  &&\text{for all } j\in\N_0 \text{ as }\eps\to 0^\pm.
  \end{align*}
\end{prop}

 This leads to the following uniform estimates.

\begin{prop} \label{prop:L2_l2_estimate}
  For all $\delta>0$ there is $C_{\delta}>0$ such that for all $\eps\in\R\sm\{0\},f\in L^2(\R^d)$ we have 
  \begin{align*}
    \| \mathcal{R}^{\eps,j}_{2}(\lambda)f\|_{L^2(\R^d;\C)}
    %\leq \| \mathcal{R}^{\eps,j}_{2}f(x+\cdot)\|_{l^2(\Z^d)}
    &\leq C_{\delta} 2^{j(1+\delta)} \|f\|_{L^2(\R^d)} &&\text{for all } j\in\N_0 \text{ and } \\
    \| \mathcal{R}^{\eps,j}_{2}(\lambda)f-\mathcal{R}^{\pm,j}_2(\lambda)f\|_{L^2(\R^d;\C)}
    %\leq \| \mathcal{R}^{\eps,j}_{2}f(x+\cdot)\|_{l^2(\Z^d)}
    &\leq C_\delta \eps^\beta 2^{j(1+\delta)} \|f\|_{L^2(\R^d)} &&\text{for all } j\in\N_0 \text{ as }\eps\to
    0^\pm.
  \end{align*}
\end{prop}
\begin{proof}
  Again, we only prove the first estimate since it relies on the first inequality from
  Proposition~\ref{prop:estimate_FBtransform_kernel} in the same way as the second estimate relies on the
  second inequality from Proposition~\ref{prop:estimate_FBtransform_kernel}. First we recall the convolution
  formula for the Floquet-Bloch transform. By the quasiperiodicity of the eigenfunctions we have  $K^{\eps,j}_{2}(x+n,y)=K^{\eps,j}_{2}(x,y-n)$ for all
  $x,y\in\R^d,n\in\Z^d$, see \eqref{eq:def_Repsj}. This yields the following formula for $x\in\Omega,l\in B$
  and $f\in C_0^\infty(\R^d)$:
  \begin{align*}
    U( \mathcal{R}^{\eps,j}_{2}(\lambda) f)(x,l) 
    &= U\Big( \int_{\R^d} K^{\eps,j}_{2}	(\cdot,y)f(y)\,dy\Big)(x,l)  \\
    &= |B|^{-1/2} \sum_{m\in\Z^d} e^{iml} \int_{\R^d} K^{\eps,j}_{2}(x-m,y)f(y)\,dy \\
    &= |B|^{-1/2} \sum_{m,n\in\Z^d} e^{i(m-n)l} e^{inl} \int_{\Omega} K^{\eps,j}_{2}(x-m,y-n)f(y-n)\,dy \\
    &= |B|^{-1/2} \sum_{m,n\in\Z^d} \int_{\Omega} e^{i(m-n)l}  K^{\eps,j}_{2}(x-(m-n),y)  f(y-n) e^{inl} \,dy \\
    &= |B|^{1/2} \int_{\Omega} U(K^{\eps,j}_{2}(\cdot,y))(x,l)   Uf(y,l) \,dy \\
  \intertext{
  and hence by Proposition~\ref{prop:estimate_FBtransform_kernel}
  }
    \big| U( \mathcal{R}^{\eps,j}_{2}(\lambda) f)(x,l) \big|
   % &\leq C\int_{\Omega} |U(K^{\eps,j}_{2}(\cdot,y))(x,l)||Uf(y,l)| \,dy \\
    &\leq |B|^{1/2} C_\delta 2^{j(1+\delta)} \int_{\Omega} |Uf(y,l)| \,dy. \\
  \intertext{
  Taking the $L^2-$norm over $\Omega\times B$ and using the isometry property of the Floquet
  transform as well as H\"older's inequality we arrive at
  }
    \| \mathcal{R}^{\eps,j}_{2}(\lambda)f \|_{L^2(\R^d;\C)}
  %  &(\leq \|(\mathcal{R}^{\eps,j}_{2}(\lambda)f)(x+\cdot)\|_{l^2(\Z^d)}) \\
    &= \| U(\mathcal{R}^{\eps,j}_{2}(\lambda)f) \|_{L^2(\Omega\times B;\C)} \\
    &\leq  (|B||\Omega|)^{1/2} C_{\delta} 2^{j(1+\delta)}    
      \Big(\int_B \Big( \int_\Omega |Uf(y,l)|\,dy \Big)^2\,dl\Big)^{1/2} \\ 
    &\leq (2\pi)^{d/2}C_{\delta} 2^{j(1+\delta)}   \|Uf\|_{L^2(\Omega\times B;\C)} \\ 
    &=  (2\pi)^{d/2} C_{\delta} 2^{j(1+\delta)}  \|f\|_{L^2(\R^d)}.
  \end{align*} 
  By density of $C_0^\infty(\R^d)$ in $L^2(\R^d)$ the result follows.
\end{proof}

  %Think about a better estimate for $\int_B \Big( \int_\Omega |Uf(y,l)|\,dy \Big)^2\,dl$ using
  %Stein-Tomas. Necessary?

  By interpolation we deduce the following estimates.

\begin{lem}\label{lem:resonant_mappingproperties}
  Assume that $p,q\in [1,\infty]$ satisfy
   \begin{align}\label{eq:resonant_admissibleexponents}
      \begin{aligned}
      &1\leq p \leq  \frac{2(d+1)}{d+3}, \quad &&\frac{2dp}{2+p(d-3)}<q\leq \infty \qquad\text{or} \\
      &\frac{2(d+1)}{d+3}< p < \frac{2d}{d+1},  \quad  &&\frac{2p}{2d-p(d+1)}<q\leq \infty.
    \end{aligned}
  \end{align}
  Then there are $C>0>\gamma$ such that we have for all $\eps\in\R\sm\{0\}$ and $f\in C_0^\infty(\R^d)$
  \begin{align*}
    \|\mathcal{R}^{\eps,j}_{2}(\lambda)f\|_{L^q(\R^d;\C)}
    %\leq \|(\mathcal{R}^{\eps,j}_{2}(\lambda)f)(x+\cdot)\|_{l^p(\Z^d)}
    &\leq C 2^{\gamma j} \|f\|_{L^p(\R^d)} &&\text{for all } j\in\N_0 \text{ and }  \\
    \|\mathcal{R}^{\eps,j}_{2}(\lambda)f-\mathcal{R}^{\pm,j}_{2}(\lambda)f\|_{L^q(\R^d;\C)}
    %\leq \|(\mathcal{R}^{\eps,j}_{2}(\lambda)f)(x+\cdot)\|_{l^p(\Z^d)}
    &\leq C \eps^\beta 2^{\gamma j} \|f\|_{L^p(\R^d)}&&\text{for all } j\in\N_0 \text{ as }\eps\to 0^\pm. 
  \end{align*}
%    with 
%   $$
%     \|\mathcal{R}^{\eps}_{2}(\lambda)f\|_{L^p(\R^d)}
%     \leq \frac{C}{2^\gamma-1} \|f\|_{L^q(\R^d)}. 
%   $$
\end{lem}
\begin{proof}
   Let $p,q$ satisfy~\eqref{eq:resonant_admissibleexponents}.
   By Proposition~\ref{prop:L1_linfty_estimate} and Proposition~\ref{prop:L2_l2_estimate}, for all $\delta>0$
   and $\tilde p,\tilde q$ such that $1\leq \tilde p\leq p$, $q\leq\tilde q\leq \infty$ the following
   estimates hold for $f\in C_0^\infty(\R^d)$
  \begin{align*}
    \|\mathcal{R}^{\eps,j}_{2}(\lambda)f\|_{L^2(\R^d;\C)}
    &\leq C_\delta 2^{j(1+\delta)}\|f\|_{L^2(\R^d)},\\
    \|\mathcal{R}^{\eps,j}_{2}(\lambda)f\|_{L^{\tilde q}(\R^d;\C)}
    &\leq C_\delta 2^{j(\frac{1+d}{2}+\frac{d}{\tilde q}-\frac{d}{\tilde p})}\|f\|_{L^{\tilde p}(\R^d)}.
  \end{align*}
  Interpolating these estimates provides the estimate
  \begin{align*}
    \|\mathcal{R}^{\eps,j}_{2}(\lambda)f\|_{L^q(\R^d;\C)}
    &\leq C_\delta 2^{j(1+\delta+\theta(\frac{d-1}{2}-\delta+\frac{d}{\tilde
    q}-\frac{d}{\tilde p}))}\|f\|_{L^p(\R^d)}
  \end{align*}
  provided $\theta\in [0,1]$ and $\tilde p,\tilde q$ are chosen according to $\frac{1}{p} =
  \frac{\theta}{\tilde p} + \frac{1-\theta}{2}, \frac{1}{q} = \frac{\theta}{\tilde q} + \frac{1-\theta}{2}$.
  Solving the latter equation for $\tilde q\in [q,\infty]$ was arbitrary, we find
  \begin{align*}
    \|\mathcal{R}^{\eps,j}_{2}(\lambda)f\|_{L^q(\R^d;\C)}
    &\leq C_\delta 2^{j(1-\frac{d}{2}+\frac{d}{q}+\delta+\theta(d-\frac{1}{2}-\delta-\frac{d}{\tilde
    p}))}\|f\|_{L^p(\R^d)}
  \end{align*}
  provided $\frac{1}{p}=\frac{\theta}{\tilde p}+\frac{1-\theta}{2}$ and
  $1\geq \theta\geq \frac{q-2}{2}$. Solving the equation for $\tilde p\in [1,p]$ we get
  \begin{align*}
    \|\mathcal{R}^{\eps,j}_{2}(\lambda)f\|_{L^q(\R^d;\C)}
    &\leq C_\delta 2^{j(1+\frac{d}{q}-\frac{d}{p}\theta \frac{d-1}{2} + (1-\theta)\delta)}\|f\|_{L^p(\R^d)}
  \end{align*}
  provided $1\geq \theta\geq \max\{\frac{q-2}{q},\frac{2-p}{p}\}$. Choosing now $\delta>0$
  sufficiently small and $\theta$ smallest possible, we observe that a negative $\gamma$ with the required
  properties exists provided $$
    1+\frac{d}{q}-\frac{d}{p} + \max\Big\{\frac{q-2}{q},\frac{2-p}{p}\Big\} \cdot \frac{d-1}{2}<0 
  $$
  or equivalently
  \begin{equation}\label{eq:resonant_Rieszdiagrm_conditions}
    \frac{d+1}{2} - \frac{d}{p}+\frac{1}{q}<0,\quad
    \frac{3-d}{2} + \frac{d}{q}-\frac{1}{p}<0.
  \end{equation}
  These inequalities are equivalent to~\eqref{eq:resonant_admissibleexponents} so that the result is proved.
\end{proof}

\section{Proofs of Theorem \ref{thm_LAP} and Corollary~\ref{cor_NLH}} \label{sec:proofs}

  \noindent \textit{Proof of Theorem~\ref{thm_LAP}:}\;  
%   We first discuss the mapping properties of the
%   resolvent operators $\mathcal R_\lambda^\eps$. From Lemma~\ref{lem:mapping_properties_nonresonant} and
%   Lemma~\ref{lem:resonant_mappingproperties} we get $\|\mathcal{R}^\eps(\lambda)f\|_{L^q(\R^d;\C)} \leq C
%   \|f\|_{L^p(\R^d)}$ for $p,q$ as in \eqref{eq:admissible_pq} and $f\in L^p(\R^d)$. Moreover, the function
%   $u^\eps:= \mathcal{R}^\eps(\lambda)f$ is a strong solution of $Lu-(\lambda +i\eps)u = f$. Classical
%   $W^{2,p}-$estimates for uniformly elliptic second order operators (see for instance Theorem C.1.3.(iii) in
%   \cite{LorBer_analytical})  therefore yield $\|u^\eps\|_{W^{2,q}(\R^d;\C)}\leq C\|f\|_{L^p(\R^d)}$
%   
%   The first part of the proof implies $u^\eps\to u^\pm :=\mathcal{R}^\pm(\lambda)f$ as $\eps\to 0^\pm$ in
%   $L^q(\R^d;\C)$ and hence we obtain for all test functions $g\in C_0^\infty(\R^d)$
  The first step of the proof is the definition of the
  operators $\mathcal{R}^\pm(\lambda)$. In view of the results of the previous chapter it is
  reasonable to define for $f\in C_0^\infty(\R^d)$
  $$
    \mathcal{R}^\pm(\lambda)f
    := \mathcal{R}_1^0(\lambda) f + \sum_{j=0}^\infty \mathcal{R}_2^{\pm,j}(\lambda) f,
  $$
  see~\eqref{eq:Reps12_defn} and \eqref{eq:def_Rpmj}. For $p,q$ as in \eqref{eq:admissible_pq} these mappings
  satisfy an estimate of the form 
  $$
    \|\mathcal{R}^\pm(\lambda)f\|_{L^q(\R^d;\C)} \leq C \|f\|_{L^p(\R^d)}
  $$
  for a positive number $C$ independent of $f$, see Lemma~\ref{lem:mapping_properties_nonresonant} and
  Lemma~\ref{lem:resonant_mappingproperties}. Since $C_0^\infty(\R^d)$ is dense in $L^p(\R^d)$,
  $\mathcal{R}^\pm(\lambda)$ extend to bounded linear operators (denoted with the same symbol) from
  $L^p(\R^d)$ to $L^q(\R^d;\C)$. The same lemmas provide the equiboundedness of the bounded linear operators
  $\mathcal{R}^{\eps}(\lambda):L^p(\R^d)\to L^q(\R^d;\C)$ as well as 
  $$
    \|\mathcal{R}^{\eps}(\lambda)f - \mathcal{R}^\pm(\lambda)f\|_{L^q(\R^d;\C)}  
    \leq C \eps^\beta \|f\|_{L^p(\R^d)}\qquad\text{as }\eps\to 0^\pm. 
  $$
  From this we deduce $\mathcal{R}^{\eps}\to \mathcal{R}^{\pm}$ as $\eps\to 0^\pm$ in the operator norm. 
  
  \medskip 
  
  We now show that $\mathcal{R}^\pm(\lambda)$ define a resolvent-type operators for $L-\lambda$.
  For $f\in L^p(\R^d)$ we set $u^\eps:= \mathcal{R}^\eps(\lambda)f\in W^{2,p}_{loc}(\R^d;\C)\cap L^q(\R^d;\C)$
  so that $u^\eps$ is a complex-valued strong solution of $Lu-(\lambda +i\eps)u = f$ for every $\eps\in\R\sm\{0\}$. The first part
  of the proof implies $u^\eps\to u^\pm :=\mathcal{R}^\pm(\lambda)f$ as $\eps\to 0^\pm$ in $L^q(\R^d;\C)$ and hence we obtain for all test functions
  $g\in C_0^\infty(\R^d)$
  \begin{align*}
    \int_{\R^d} f g   
    = \int_{\R^d} (L-\lambda-i\eps)u^\eps g   
    = \int_{\R^d} u^\eps (L-\lambda-i\eps)g   
    \to \int_{\R^d} u^\pm  (L-\lambda)g  \qquad\text{as }\eps\to 0^\pm. 
  \end{align*}
  As a consequence, $u^\pm$ is a distributional solution of the linear elliptic PDE $(L-\lambda)u=f$ on~$\R^d$
  and therefore (see for instance Theorem~2 in~\cite{Man_note_on}) it satisfies this PDE in the strong sense
  as an element of $W^{2,p}_{loc}(\R^d;\C)\cap L^q(\R^d;\C)$. 
  
  \medskip
  
  It is left to prove that $\mathcal{R}^\pm(\lambda)f$ lies in $W^{2,p}(\R^d;\C)+W^{2,q}(\R^d;\C)$. To this
  end set $L_0\psi:=-\diver(A\nabla\psi)$ and write $u^\eps=v^\eps+w$ where 
  $$
    v^\eps:= (L_0+1)^{-1}\big((1-V-\lambda-i\eps)u^\eps\big)\in W^{2,q}(\R^d;\C),\qquad
    w:= (L_0+1)^{-1}f \in W^{2,p}(\R^d).  
  $$
  The $W^{2,p}$-estimates for $L_0+1$ from~Theorem C.1.3.(iii) in \cite{LorBer_analytical} and
  $V\in L^\infty(\R^d)$ imply
  \begin{align*}
    \|v^\eps\|_{W^{2,q}(\R^d;\C)} 
    + \|w\|_{W^{2,p}(\R^d)}
    &\leq C (\|u^\eps\|_{L^q(\R^d;\C)} +  \|f\|_{L^p(\R^d)})
    \leq C \|f\|_{L^p(\R^d)}, \\
    \|v^\eps-v^\delta\|_{W^{2,q}(\R^d;\C)}
    &\leq C\|u^\eps-u^\delta\|_{L^q(\R^d;\C)}
    \leq o(1) \|f\|_{L^p(\R^d)} 
  \end{align*}
  as $\eps,\delta \to 0^+$ or $\eps,\delta \to 0^-$. Hence, we get 
  $$
    \mathcal{R}^\pm(\lambda)f
    = \lim_{\eps\to 0^\pm} u^\eps 
    = \lim_{\eps\to 0^\pm} v^\eps + w 
    \in W^{2,q}(\R^d;\C)+W^{2,p}(\R^d).
  $$
   \qed
  
  \medskip
  
  We notice that the operators $\mathcal{R}^\pm(\lambda)$ are defined as integral operators with a kernel
  function 
  \begin{equation}\label{eq:def_Kpm}
    K^\pm(x,y)
    := K^0_1(x,y)+K^\pm_2(x,y)
    := K^0_1(x,y)+ \sum_{j=1}^\infty K^{\pm,j}_2(x,y)
  \end{equation} 
  where the integral has to be understood in the sense of an oscillatory integral, i.e. 
  $$
    (\mathcal{R}^\pm(\lambda)f)(x)
    = \int_{\R^d} K^0_1(x,y)f(y)\,dy + \sum_{j=0}^\infty \int_{\R^d} K_2^{\pm,j}(x,y)f(y)\,dy.  
  $$
  We will use $K^\pm(x,y)=\ov{K^\mp(y,x)}$ for all $x,y\in\R^d$ as well as $K^\pm(x+m,y)=K^\pm(x,y-m)$ for all
  $x,y\in\R^d$ and $m\in\Z^d$, which follows from the corresponding properties of each of the summands
  in~\eqref{eq:def_Kpm} described in the lines after~\eqref{eq:def_Keps_12}.
 
 \medskip
 
  \noindent \textit{Proof of Corollary~\ref{cor_NLH}:}\; As pointed out in the introduction, the idea for the
  proof of this result is completely due to Evequoz and Weth \cite{EvWe_dual}. We quickly review in which way our
  construction of the resolvent from Theorem~\ref{thm_LAP} makes it possible to use their methods. Following
  their notation we set for $f\in L^{q'}(\R^d)$ with $\frac{2(d+1)}{d-1}<q<\frac{2d}{d-2}$
  \begin{align} \label{eq:def_boldR}
    \begin{aligned}
    \bold{R}f
    &:= \frac{1}{2}\Real\big(\mathcal{R}^+(\lambda)f + \mathcal{R}^-(\lambda)\big)f 
    = \int_{\R^d}K^*(x,y)f(y)\,dy, \qquad\text{where } \\
    K^*(x,y) 
    &:= \frac{1}{2}\Real\big(K^+(x,y)+K^-(x,y)\big) = \frac{1}{2}\Real \big(K^+(x,y)+K^+(y,x)\big).
    \end{aligned} 
  \end{align} 
  By Theorem~\ref{thm_LAP} this formula defines a bounded linear operator
  from $L^{q'}(\R^d)$ to $L^q(\R^d)$ and $\bold{R}f$ is a real-valued strong
  solution of $Lu-\lambda u=f$ by Theorem~\ref{thm_LAP}. By construction, we moreover have
  \begin{equation}\label{eq:symmetries_Kstar}
    K^*(x,y) = K^*(y,x),\quad K^*(x+m,y)=K^*(x,y-m)\qquad\text{for all } m\in\Z^d,\,x,y\in\R^d.
  \end{equation}
  A nontrivial solution $u\in L^q(\R^d)$ of the nonlinear Helmholtz equation $Lu-\lambda u = \pm \Gamma
  |u|^{q-2}u$ may then be solved by proving the existence of a nontrivial function $v\in L^{q'}(\R^d)$ such
  that
  \begin{equation}\label{eq:dual_equation}
    \Gamma^{-1/(q-1)} |v|^{q'-2}v = \pm  \bold{R}v.
  \end{equation}
  %see~(47) in~\cite{EvWe_dual}. 
  Exploiting the first equation in \eqref{eq:symmetries_Kstar} we conclude that
  the equation~\eqref{eq:dual_equation} is variational and its Euler functional $J:L^{q'}(\R^d)\to\R$ is given
  by 
  $$
    J(v) = \frac{1}{q'}\|  \Gamma^{-1/q} v\|_{L^{q'}(\R^d)}^{q'} \mp \frac{1}{2} \int_{\R^d}  v
    \bold{R}v.
  $$
  %cf. (48) in \cite{EvWe_dual}. 
  This functional is continuously differentiable and has the mountain
  pass geometry, see Lemma~4.2 in~\cite{EvWe_dual}. The only point in the verification of
  this lemma that is not so obvious, is the existence of nontrivial functions $z_+,z_-\in L^{q'}(\R^d)$
  such that 
  \begin{equation} \label{eq:MP_verification}
    \mp \int_{\R^d}  z_\pm \bold{R}z_\pm <0.
  \end{equation}
  In order to find such a function we adapt the idea from Lemma~3.1 in~\cite{MaMoPe_oscillating}.  
  We choose 
  $$
    K_\pm 
    := \{k\in\R^d : \delta \leq \pm(\Lambda(k)-\lambda)\leq 2\delta\} 
  $$
  where $\Lambda:\R^d\to\R$ has the properties described by (A2) and $\delta>0$ is chosen so small that
  $K_\pm$ has positive measure and $K_\pm\subset U$ for $U$ as in (A2).
  This is possible due to $\nabla\Lambda\neq 0$ on $U$ and the Implicit Function Theorem. Then we
  define $z_\pm$ via
  \begin{equation}\label{eq:MP_verificationII}
    z_\pm := 1_{B_R(0)} y_\pm,\qquad
    U(y_\pm)(x,k) := \sum_{s\in\Z^d} 1_{K_\pm}(k+2\pi s)\psi_s(x,k) 
  \end{equation}
  where $R$ will be chosen sufficiently large.
  From \eqref{eq:resolvent+Floquet} we get  
  \begin{align*}
    &\;  \mp \int_{\R^d}  y_\pm \mathcal{R}^+ y_\pm \\
    &= \lim_{\eps\to 0^+} \Big[ \mp   \int_{\R^d} y_\pm
    \mathcal{R}^{\eps}(\lambda) y_\pm \Big]  \\
    &=  \lim_{\eps\to 0^+} \Big[\mp  \int_{\Omega} \int_B \overline{(U y_\pm)(x,k)}
    \cdot U(\mathcal{R}^{\eps}(\lambda)y_\pm )(x,k)\,dk \,dx \Big)  \\ 
    &=  \lim_{\eps\to 0^+} \Big[  \mp\int_{\Omega} \int_{B} \sum_{s\in\Z^d} 1_{K_\pm}(k+2\pi
    s)\overline{\psi_s(x,k)} \cdot  \sum_{t\in\Z^d}
    \frac{\skp{U(y_\pm)(\cdot,k)}{\psi_t(\cdot,k)}_{L^2(\Omega;\C)}}{\lambda_t(k)-\lambda-i\eps}
    \psi_t(x,k) \,dk \,dx \Big]  \\
    &=   \lim_{\eps\to 0^+}  \Big[\mp \sum_{s\in\Z^d} \int_{\Omega} \int_{B} 1_{K_\pm}(k+2\pi s) 
    \frac{|\psi_s(x,k)|^2}{\lambda_s(k)-\lambda-i\eps} \,dk \,dx \Big]  \\
    &\stackrel{\eqref{eq:defn_LambdaPsi}}{=}   \lim_{\eps\to 0^+}  \Big[\mp \int_{\Omega} \int_{K_\pm}  
    \frac{|\Psi(x,k)|^2}{\Lambda(k)-\lambda-i\eps} \,dk \,dx \Big]  \\
    &=  \mp \int_{\Omega} \int_{K_\pm} \frac{|\Psi(x,k)|^2}{\Lambda(k)-\lambda} \,dk \,dx 
    < 0.
  \end{align*}
  In the second last equality we used that $\{\psi_t(\cdot,k):t\in\Z^d\}$ is an orthonormal basis of
  $L^2(\Omega;\C)$. The calculations for the integral of $y_\pm \mathcal R^- y_{\pm}$ are exactly the
  same, for it suffices to replace $\eps\to 0^+$ by $\eps\to 0^-$. So, the
  definition of $\bold{R}$ from~\eqref{eq:def_boldR} implies 
  $$
    \mp \int_{\R^d} y_\pm \bold{R}y_\pm   < 0.
  $$
  Choosing now $R$ large enough (but finite) in \eqref{eq:MP_verificationII} we get~\eqref{eq:MP_verification}
  as well as $z_\pm\in L^{q'}(\R^d)$ by the explicit formula for $U^{-1}$ from
  \eqref{eq:Floqet_Bloch_transform}.
  So the Mountain Pass Theorem provides a Palais-Smale sequence for $J$ at its mountain pass level $c>0$,
  which is defined as in Section~6 of \cite{EvWe_dual}.
  This sequence is bounded and using the periodicity of $\Gamma$ as well as \eqref{eq:symmetries_Kstar}
  we get from the \q{nonvanishing property} (see Theorem~3.1 in \cite{EvWe_dual}) that, up to
  translation, the Palais-Smale sequence converges weakly to a nontrivial solution $v\in L^{q'}(\R^d)$ of
  \eqref{eq:dual_equation} which has the right energy level $c$. As in~\cite{EvWe_dual} this provides an
  $L^q(\R^d)$-solution $u$ of \eqref{eq:NLH} and it remains to discuss its global regularity.
  
  \medskip
  
  First we claim $u\in L^q(\R^d)\cap L^\infty(\R^d)$. This follows as in the first part of the proof
  of~Lemma~4.3~\cite{EvWe_dual}, where the corresponding result is proved for the Helmholtz operator
  $-\Delta-1$ instead of $L$. Notice that the method used there is based on a kind of Moser iteration, which
  remains valid for general linear elliptic second order operators such as $L$.  So we have
  $$
    (L_0+1)u 
    = f + (1-V-\lambda)u \in \big(L^{q'}(\R^d)\cap L^\infty(\R^d)\big) +  \big( L^q(\R^d)\cap  L^\infty(\R^d)\big) 
    = L^q(\R^d)\cap L^\infty(\R^d)
  $$
  and Theorem C.1.3.(iii) in \cite{LorBer_analytical} implies $u\in W^{2,r}(\R^d)$ for all $r\in [q,\infty)$. \qed

%   $$
%     -\Delta  u^\eps+ u^\eps = (1-V(x)+\lambda+i\eps)u^\eps+f
%   $$
%   in the strong sense. Since the terms on the right hand side belong to $L^p(\R^d),L^q(\R^d)$ respectively, we
%   have $u^\eps=v^\eps+w$ for the following functions:
%   $$
%     v^\eps:= (-\Delta+1)^{-1}((1-V(x)+ \lambda+i\eps)u^\eps),\qquad
%      w := (-\Delta+1)^{-1}(f).
%   $$
%   From Bessel potential estimates (see for instance Theorem 3 on p.135 in \cite{Ste_SingInt}) we get
%   $v^\eps\in W^{2,p}(\R^d),w\in W^{2,q}(\R^d)$ with  
%   $$
%     \|v^\eps\|_{W^{2,p}(\R^d)}  + \|w\|_{W^{2,q}(\R^d)}
%     \leq C \big( \|1-V+\lambda+i\eps\|_{L^\infty(\R^d)}\|u^\eps\|_{L^p(\R^d)} + \|f\|_{L^q(\R^d)} \big) 
%     \leq C \|f\|_{L^q(\R^d)},
%   $$
%   where we have used \eqref{eq:thm_resolvent_estimate}. In particular, the sequences  
%   $(v^\eps)_{\eps\to 0^+},(v^\eps)_{\eps\to 0^-}$ are bounded in the reflexive separable Banach
%   space $W^{2,p}(\R^d)$ and thus contain weakly convergent subsequences in this space. Our aim is to shows
%   that the limit of either subsequence does not depend on the choice of the subsequence. 
%   

\section{Proof of Proposition~\ref{prop:resonant_pointwise_decay}} \label{sec:proof_prop_I}

The proof of Proposition~\ref{prop:resonant_pointwise_decay} uses the method of stationary phase (p.348ff.
\cite{Ste_harmonic_analysis}) in order to derive the pointwise bounds for $K^\eps_2(x,y)$. The crucial
observation is that in the definition of this kernel function, see~\eqref{eq:def_Keps},
the integration takes place over those regions which correspond to
a foliation by Fermi surfaces $(F_\tau)$. These hypersurfaces have positive
Gaussian curvature by (A2) so that we may prove decay estimates for integrals of the form 
$$
  \int_{\R} \frac{\chi(\lambda-\tau)}{\lambda-\tau-i\eps} \Big( \int_{F_\tau} h(k) e^{i\sigma
  \skp{v}{k}} \,d\mathcal{H}^{d-1}(k) \Big)\,d\tau
$$
by the method of stationary phase. As we will see later, such estimates yield pointwise bounds for
$K^\eps_2(x,y)$ when $\sigma=|x-y|$ and $v=\tfrac{x-y}{|x-y|}$. We recall that $\chi$ is chosen to
satisfy~\eqref{eq:Def_chi} for some $\rho>0$ that we will define later and which will only depend on the data
from (A1),(A2). The main technical difficulties come from the
fact that our estimates have to be uniform with respect to $\eps$ and that the presence of the singular
prefactor requires to estimate both $a(\lambda)$ and $a(\lambda+t)-a(\lambda)$ where  
\begin{equation} \label{eq:integral_a}
  a(\tau) = \chi(\tau-\lambda)\int_{F_\tau} h(k) e^{i\sigma \skp{v}{k}} \,d\mathcal{H}^{d-1}(k)
\end{equation}
for $\tau\in (\lambda-\rho,\lambda+\rho)$. This fact will be proved first. 
%\r{$N$ as in the proposition}

\begin{prop} \label{prop:Plemelj_formula}
  Let $\lambda\in\R,\rho\in (0,\infty]$ and assume that $a:[\lambda-\rho,\lambda+\rho]\to\R$ is measurable
  such that $|a(\lambda+t)-a(\lambda)|\leq \omega(|t|)$ where $t\mapsto
  \omega(t)/t$ is integrable over $(0,\rho)$. Then the following
  inequalities hold for $\eps>0$:
  \begin{align*}
    &\text{(i)}\qquad \Big| \int_{\lambda-\rho}^{\lambda+\rho} \frac{a(\tau)}{\tau-\lambda\mp i\eps} \,d\tau  
    - p.v.\int_{\lambda-\rho}^{\lambda+\rho} \frac{a(\tau)}{\tau-\lambda}\,d\tau 
    \mp i\pi a(\lambda)\Big| \\
    &\qquad\qquad \leq   \int_0^\rho 
    \frac{2\eps}{\sqrt{t^2+\eps^2}} \frac{\omega(t)}{t}\,dt
    + (\pi-2\arctan(\rho/\eps))|a(\lambda)|, \\
    &\text{(ii)}\qquad
    \Big| \int_{\lambda-\rho}^{\lambda+\rho} \frac{a(\tau)}{\tau-\lambda\mp i\eps} \,d\tau \Big| 
    \leq  2\pi \Big( \int_0^\rho
      \frac{\omega(t)}{t}\,dt + |a(\lambda)|\Big). 
%      \\
%     &\text{(iii)}\quad \Big| \int_{\R} \frac{a(x)}{x-\lambda-i\eps} \,dx - \int_{\R}
%     \frac{a(x)}{x-\lambda-i\eps'} \,dx %- 2\pi ia(\lambda)\sign(\eps')1_{\eps'\eps<0} 
%     \Big| \leq  C_\delta (\eps-\eps')^\delta \int_0^\infty \frac{\omega(t)}{t^{1+\delta}}\,dt 
      % \\
%     &\text{(4)}\quad \int_\R \Big| \int_{\R} \frac{a(x)}{x-\tau-i\eps} \,dx - \int_{\R}
%     \frac{a(x)}{x-\lambda-i\eps'} \,dx  \Big|^2 \,d\tau = o(1) \quad\text{as }\eps,\eps'\to 0
  \end{align*}
% provided $a\in L^2(\R^d)$.
\end{prop}  
\begin{proof}
  Without loss of generality we assume $\lambda=0$. Then we
  have
  \begin{align*}
    & \Big| \int_{-\rho}^\rho \frac{a(\tau)}{\tau\mp i\eps} \,d\tau  - p.v.\int_{-\rho}^\rho \frac{a(\tau)}{\tau}\,d\tau
    \mp i\pi a(0)\Big| \\
    &= \Big| \lim_{r\to 0^+} \int_{r<|\tau|<\rho} \Big( \frac{a(\tau)}{\tau\mp i\eps} - \frac{a(\tau)}{\tau} 
    \mp \frac{i\eps (\tau\pm i\eps) a(0)}{\tau(\tau^2+\eps^2)} \Big) \,d\tau 
    \mp i(\pi-2\arctan(\rho/\eps))a(0)  \Big| \\
    &\leq  \liminf_{r\to 0^+} \int_{r<|\tau|<\rho} \Big| \frac{a(\tau)}{\tau\mp i\eps} - \frac{a(\tau)}{\tau} 
    \mp \frac{i\eps (\tau\pm i\eps) a(0)}{\tau(\tau^2+\eps^2)} \Big| \,d\tau 
    + (\pi-2\arctan(\rho/\eps))|a(0)| \\
    %&=  \liminf_{r\to 0^+} \int_{r<|\tau|<\rho} \Big| \pm \frac{i\eps(\tau\pm
    %i\eps)a(\tau)}{\tau(\tau^2+\eps^2)} \mp \frac{i\eps (\tau\pm
    %i\eps) a(0)}{\tau(\tau^2+\eps^2)} \Big| \,d\tau 
    %+ (\pi-2\arctan(\rho/\eps))|a(0)|\\
    &=  \liminf_{r\to 0^+} \int_{r<|\tau|<\rho}   \frac{\eps|a(\tau)-a(0)|}{|\tau|\sqrt{\tau^2+\eps^2}}  \,d\tau
    + (\pi-2\arctan(\rho/\eps))|a(0)|  \\
    &\leq  \int_0^\rho \frac{2\eps}{\sqrt{t^2+\eps^2}} \frac{\omega(t)}{t} \,dt +
    (\pi-2\arctan(\rho/\eps))|a(0)|.
  \end{align*}  
  This proves (i) and (ii) follows from  
  \begin{align*}
    \Big| \int_{-\rho}^\rho\frac{a(\tau)}{\tau\mp i\eps} \,d\tau  \Big|
    &\leq \Big|p.v.\int_{-\rho}^\rho \frac{a(\tau)}{\tau}\,d\tau -i\pi a(0)\Big| 
      +  \int_0^\rho \frac{2\eps}{\sqrt{t^2+\eps^2}} \frac{\omega(t)}{t} \,dt 
      +  (\pi-2\arctan(\rho/\eps))|a(0)| \\
    &\leq \int_{-\rho}^\rho \frac{|a(\tau)-a(0)|}{|\tau|}\,d\tau + \pi|a(0)| +  2\int_0^\rho
    \frac{\omega(t)}{t} \,dt 
     + (\pi-2\arctan(\rho/\eps))|a(0)|\\ 
    &\leq 2\pi \Big(\int_0^\rho \frac{\omega(t)}{t} \,dt + |a(0)|\Big).
  \end{align*}
\end{proof}

Variants of the above result are usually attributed to Plemelj and Sokhotski.  In order to derive estimates for $a$ 
as in~\eqref{eq:integral_a} we will perform a change of coordinates
in order to reduce the estimates over the Fermi surfaces $F_\tau$ to estimates over open subsets of
$\R^{d-1}$ where $\tau-\lambda\in I:= [-\rho,\rho]$ where $\rho>0$ is chosen later.
The estimates over those pieces of the $F_\tau$ where the phase function $k\mapsto \skp{v}{k}$ is
nonstationary will be estimated with the aid of the following result.

\begin{prop}\label{prop:nonstationary_phase}
  Let $K\subset\R^{d-1}$ be a compact set, $\alpha,\beta\in (0,1)$  and let $\Phi\in
  C^{0,\beta}(I;W^{N,\infty}(K))$ satisfy $|\nabla\Phi_t|\geq c>0$ on $K$ for all $t\in I$. Then there
  is a $C>0$ such that for $|\sigma|\geq 1$ and  $f\in C^{0,\beta}(I;W^{N-1,1}(\R^{d-1}))$ with
  $\supp(f_t)\subset K$ we have
  \begin{align*}
    \Big| \int_{\R^{d-1}} f_t(x) e^{i\sigma \Phi_t(x)} \,dx \Big| 
    &\leq C|\sigma|^{1-N} %\|\phi_t\|_{W^{N,\infty}(\R^{d-1})}^N  
    \|f_t\|_{W^{N-1,1}(\R^{d-1})}, \\
    \Big| \int_{\R^{d-1}} f_t(x) e^{i\sigma \Phi_t(x)} \,dx - \int_{\R^{d-1}} f_0(x) e^{i\sigma \Phi_0(x)}
    \,dx \Big|  
    &\leq C|t|^\beta|\sigma|^{\alpha+1-N}\|f\|_{C^{0,\beta}(I;W^{N-1,1}(\R^{d-1}))}.
  \end{align*}
\end{prop}
\begin{proof}
  Without loss of generality we assume $\skp{\nabla\Phi_t(x)}{\xi}\geq c>0$ on $K$ for some
  unit vector $\xi\in S^{d-1}$ and all $t\in I$, otherwise consider a partition of unity of a
  suitable covering of $K\times I$ where the corresponding inequalities hold for unit vectors
  $\xi^1,\ldots,\xi^M$ for some $M\in\N$. We define the linear differential operators $D_t$ and the formal
  adjoints $D_t^*$ via 
  $$
    (D_t\psi)(x) := \frac{1}{i\sigma} \frac{\skp{\nabla \psi(x)}{\xi}}{
     \skp{\nabla\Phi_t(x)}{\xi}},\qquad
    (D_t^*\psi)(x) = \frac{i}{\sigma} \Bigskp{\nabla\Big( \frac{\psi(\cdot)}{
    \skp{\nabla\Phi_t(\cdot)}{\xi}}\Big)(x)}{\xi}.
  $$
  This definition is motivated by   $D_t(e^{i\sigma \Phi_t})=e^{i\sigma\Phi_t}$. By induction one 
  proves
  $$
    ((D_t^*)^{N-1}\psi)(x) 
    = \Big(\frac{i}{\sigma}\Big)^{N-1} \frac{P_N(\psi(x),\ldots,\nabla^{N-1}
    \psi(x),\nabla\Phi_t(x),\ldots,\nabla^N\Phi_t(x))}{\skp{\nabla\Phi_t(x)}{\xi}^N} 
  $$ 
  and $P_N$ is a polynomial of degree $N$ that is 1-homogeneous with respect to the $\psi$-components,
  because $(D_t^*)^{N-1}$ is linear, and $N-1$-homogeneous with respect to the $\Phi_t$-components. 
  %Moreover, its coefficients remain bounded as $\xi$ varies over $S^{d-1}$. 
  Therefore, integrating by parts $N-1$ times gives
  \begin{align} \label{eq:prop_nonstationary_phase_I}
    \begin{aligned}
    \Big|\int_{\R^{d-1}} f_t(x) e^{i\sigma \Phi_t(x)} \,dx\Big|
    &= \Big|\int_{\R^{d-1}} f_t(x) D_t^{N-1}(e^{i\sigma \Phi_t(\cdot)})(x) \dx  \Big|\\
    &= \Big|\int_{\R^{d-1}} ((D_t^*)^{N-1} f_t)(x) e^{i\sigma \Phi_t(x)} \dx \Big|\\
    &\leq C%\rho^{-N-1}|
    |\sigma|^{1-N}   
    \int_{\R^{d-1}} \big(|\nabla\Phi_t|+\ldots+|\nabla^N\Phi_t|\big)^{N-1} \big(|f_t|+\ldots+|\nabla^{N-1}
    f_t|\big)  \\
    &\leq C%\rho^{-N-1}
    |\sigma|^{1-N}
    %\| |\nabla\phi_t|\|_{W^{N,\infty}}^N 
    \|f_t\|_{W^{N-1,1}(\R^{d-1})} 
  \end{aligned}
  \end{align}
  and the first inequality is proved. The proof of the second inequality is similar. Proceeding as above we get
  \begin{align*}
    &\Big|\int_{\R^{d-1}} f_t(x) e^{i\sigma \Phi_t(x)} \,dx
    -  \int_{\R^{d-1}} f_0(x) e^{i\sigma \Phi_0(x)} \,dx\Big| \\
    &=  \Big| \int_{\R^{d-1}} e^{i\sigma\Phi_0(x)} \big(
      ((D_t^*)^{N-1} f_t)(x) - ((D_0^*)^{N-1} f_0)(x)  \big)\dx \Big|\\
    &+ \Big|\int_{\R^{d-1}} \big( e^{i\sigma(\Phi_0(x)-\Phi_t(x))}-1\big)((D_t^*)^{N-1}
    f_t)(x)e^{i\sigma\Phi_t(x)}\,dx  \Big|.
  \end{align*}
  The first integral is estimated as follows:
  \begin{align*}
    &\Big| \int_{\R^d} e^{i\sigma\Phi_0(x)} \big(
      ((D_t^*)^{N-1} f_t)(x) - ((D_0^*)^{N-1} f_0)(x)  \big)\dx \Big| \\
%    &\leq  |\sigma|^{1-N} \int_{\R^d}
%       \Big| \frac{P_N(f_t(x),\ldots,\nabla^{N-1}
%     f_t(x),\nabla\Phi_t(x),\ldots,\nabla^N\Phi_t(x))}{\skp{\nabla\Phi_t(x)}{\xi}^N} \\  
%     &\qquad\qquad\qquad -   \frac{P_N(f_0(x),\ldots,\nabla^{N-1}
%     f_0(x),\nabla\Phi_0(x),\ldots,\nabla^N\Phi_0(x))}{\skp{\nabla\Phi_0(x)}{\xi}^N} 
%     \Big| \dx  \\
%     &=  |\sigma|^{1-N} \int_{\R^d} \Big|\int_0^t
%       \frac{d}{ds} \Big(  \frac{P_N(f_s(x),\ldots,\nabla^{N-1}
%     f_s(x),\nabla\Phi_s(x),\ldots,\nabla^N\Phi_s(x))}{\skp{\nabla\Phi_s(x)}{\xi}^N}   
%     \Big)  \ds\Big| \dx \\
    &\leq  C|\sigma|^{1-N} \int_{\R^{d-1}} \big(|\nabla \Phi_t|+|\nabla \Phi_0|+|\nabla^N \Phi_t|+|\nabla^N
    \Phi_0|\big)^{N-1}\big(|f_t-f_0|+\ldots+|\nabla^{N-1} (f_t-f_0)|\big) \\
    &\; +  C|\sigma|^{1-N} \int_{\R^{d-1}} \big(|\nabla \Phi_t|+|\nabla \Phi_0|+|\nabla^N \Phi_t|+|\nabla^N
    \Phi_0|\big)^{N-2} \cdot \\
    &\qquad\qquad\qquad  \big(|\nabla (\Phi_t-\Phi_0)| + \ldots + |\nabla^N (\Phi_t- \Phi_0)|\big)
    \big(|f_t|+|f_0|+\ldots+|\nabla^{N-1} f_t|+|\nabla^{N-1}f_0|\big) \\
    &\leq C|t|^\beta|\sigma|^{1-N}\|f\|_{C^{0,\beta}(I;W^{N-1,1}(\R^{d-1}))}.
  \end{align*}
  The estimate for the second integral follows from the estimate ~\eqref{eq:prop_nonstationary_phase_I} and
  the global $\beta$-H\"older-continuity of sine and cosine.
\end{proof}

 While the above proposition will be used for the estimates of integrals over those regions where the phase
 is nonstationary, the following propositions deal with the resonant parts of the Fermi surfaces. 
 %The aim is to write the phase
 %functions $\Phi_t$ as quadratic 2-homogeneous polynomials around the points of stationary phase. 
 To this end we use the Fourier transform 
 $$
   \mathcal{F}f(\xi) := \hat f(\xi) := (2\pi)^{\frac{1-d}{2}}  \int_{\R^{d-1}}
   f(x)e^{-i\skp{x}{\xi}}\,dx, 
  $$
  which, as usual, is defined for all Schwartz functions in $\mathcal{S}(\R^{d-1})$ and, as an isometry on
  $L^2(\R^{d-1})$, is as well defined for all tempered distributions in $\mathcal{S}'(\R^{d-1})$.
  The dual pairing will be denoted by the symbol $\skp{\cdot}{\cdot}_{\mathcal{S}'(\R^{d-1})}$. First we calculate the Fourier transform of the
  tempered distribution given by the function $x\mapsto e^{i\sigma \skp{x}{Ax}}$. Since we did not find a
  reference for these computations, we present the proof of this well-known result.

\begin{prop} \label{prop:FT_Fresnelphase}
  Let $\sigma>0$ and $A\in\R^{(d-1)\times (d-1)}$ symmetric and invertible. Then we have 
  $$
    \mathcal{F}(e^{i\sigma \skp{x}{Ax}})(\xi)
    = (2\sigma)^{\frac{1-d}{2}} |\det(A)|^{-\frac{1}{2}} e^{i\frac{\pi}{4}\sgn(A)}
     e^{-i\frac{\skp{\xi}{A^{-1}\xi}}{4\sigma}} 
  $$
  where $\sgn(A)$ denotes the signature of $A$, i.e. 
  the number of its positive eigenvalues minus the number of its negative eigenvalues.
\end{prop}
\begin{proof}
  Let $(K_R)$ be a sequence of compact sets with $K_R \nearrow \R^{d-1}$ as $R\to\infty$. Then we have for
  all $h\in \mathcal{S}(\R^{d-1})$ by Fubini's Theorem
  \begin{align*}
    \skp{\mathcal{F}(e^{i\sigma \skp{x}{Ax}})}{h}_{\mathcal{S}'(\R^{d-1})}
    &= \skp{e^{i\sigma \skp{x}{Ax}}}{ \mathcal{F}^{-1}h}_{\mathcal{S}'(\R^{d-1})} \\
    &= \int_{\R^{d-1}} e^{i\sigma \skp{x}{Ax}} \ov{(\mathcal{F}^{-1}h)(x)}\dx \\
    &= \lim_{R\to\infty} \int_{K_R} e^{i\sigma \skp{x}{Ax}} \ov{(\mathcal{F}^{-1}h)(x)}\dx  \\
    &= (2\pi)^{\frac{1-d}{2}} \lim_{R\to\infty}\int_{\R^{d-1}}  h(\xi) \Big(\int_{K_R} e^{i (\sigma
    \skp{x}{Ax}-\skp{x}{\xi})} \dx\Big) \,d\xi.
  \end{align*}
  We show that the integral over $K_R$ converges as $R\to\infty$. To this
  end we write $A=Q^TDQ$ for an orthogonal matrix $Q$ and a diagonal matrix
  $D=\diag(\mu_1,\ldots,\mu_m,-\mu_{m+1},\ldots,-\mu_{d-1})$ containing the eigenvalues of $A$ where
  $m\in\{1,\ldots,d-1\}$ and all $\mu_j$ are positive. Then we have
  $$
    \sgn(A) = m-(d-1-m) = 2m+1-d
  $$
  and the matrix $S:=Q^T\diag(|\mu_1|^{-1/2},\ldots,|\mu_{d-1}|^{-1/2}) \sigma^{-1/2}$ satisfies
  $$
    \det(S) = \sigma^{\frac{1-d}{2}} |\det(A)|^{-\frac{1}{2}},\qquad
    \sigma\skp{Sx}{ASx} = \sum_{j=1}^{d-1} \frac{\mu_j}{|\mu_j|}x_j^2 = \sum_{j=1}^m x_j^2-\sum_{j=m+1}^{d-1}
    x_j^2  =: |x'|^2-|x''|^2.
  $$
  From this we obtain by a change of coordinates
  \begin{align*}
     &\int_{K_R} e^{i (\sigma \skp{x}{Ax}-\skp{x}{\xi})} \dx \\
     &= \int_{S^{-1}K_R} |\det(S)| e^{i (\sigma \skp{Sx}{ASx}-\skp{Sx}{\xi})} \dx \\
     &=  \sigma^{\frac{1-d}{2}} |\det(A)|^{-\frac{1}{2}}
         \int_{S^{-1}K_R}   e^{i (|x'|^2-|x''|^2 - \skp{x}{S^T\xi})}   \dx \\
     &=  \sigma^{\frac{1-d}{2}} |\det(A)|^{-\frac{1}{2}} \int_{S^{-1}K_R}   e^{i
       (|x'-\frac{1}{2}(S^T\xi)'|^2-|x''+\frac{1}{2}(S^T\xi)''|^2 - \frac{|(S^T\xi)'|^2-|(S^T\xi)''|^2}{4})}
       \dx   \\
     &= \sigma^{\frac{1-d}{2}}|\det(A)|^{-1/2}  e^{-i\frac{|(S^T\xi)'|^2-|(S^T\xi)''|^2}{4}}
      \int_{K_R'} e^{i (|y'|^2-|y''|^2)} \dy   \\
     &= \sigma^{\frac{1-d}{2}}|\det(A)|^{-\frac{1}{2}} e^{-i\frac{\skp{\xi}{A^{-1}\xi}}{4\sigma}}
      \int_{K_R'} e^{i (|y'|^2-|y''|^2)} \dy,
  \end{align*}
  where  the compact set $K_R'$ is defined by $K_R':=S^{-1}K_R + \frac{1}{2}((S^T\xi)',-(S^T\xi)'')^T$.
  From
  $$
    \int_{M_1}^{M_2} e^{\pm iz^2} \,dz \to \sqrt{\pi} e^{\pm i \frac{\pi}{4}} \quad\text{as
    }M_1\to-\infty,M_2\to\infty
  $$
  we deduce
  \begin{align*}
    \lim_{R\to\infty} \int_{K_R} e^{i (\sigma \skp{x}{Ax}-\skp{x}{\xi})} \dx
     &=  \sigma^{\frac{1-d}{2}}|\det(A)|^{-\frac{1}{2}} e^{-i\frac{\skp{\xi}{A^{-1}\xi}}{4\sigma}}
     \cdot \Big( \sqrt{\pi} e^{i\frac{\pi}{4}} \Big)^m \Big( \sqrt{\pi}
     e^{-i\frac{\pi}{4}}\Big)^{d-m-1} \\
     &= \sigma^{\frac{1-d}{2}}|\det(A)|^{-\frac{1}{2}} e^{-i\frac{\skp{\xi}{A^{-1}\xi}}{4\sigma}}
     \cdot \Big(\pi\Big)^{\frac{d-1}{2}}    e^{i\frac{\pi}{4}(2m+1-d)}\\
     &= \Big(\frac{\pi}{\sigma}\Big)^{\frac{d-1}{2}} |\det(A)|^{-\frac{1}{2}} e^{i\frac{\pi}{4}\sgn(A)}
     e^{-i\frac{\skp{\xi}{A^{-1}\xi}}{4\sigma}}.
  \end{align*}
   Hence,
  \begin{align*}
    \skp{\mathcal{F}(e^{i\sigma \skp{x}{Ax}})}{h}_{\mathcal{S}'(\R^{d-1})}
    &= (2\pi)^{\frac{1-d}{2}}   \int_{\R^{d-1}}
    \left(\Big(\frac{\pi}{\sigma}\Big)^{\frac{d-1}{2}} |\det(A)|^{-\frac{1}{2}} e^{i\frac{\pi}{4}\sgn(A)}
     e^{-i\frac{\skp{\xi}{A^{-1}\xi}}{4\sigma}} \right) h(\xi) \,d\xi \\
    &=  \Bigskp{  (2\sigma)^{\frac{1-d}{2}} |\det(A)|^{-\frac{1}{2}} e^{i\frac{\pi}{4}\sgn(A)}
     e^{-i\frac{\skp{\xi}{A^{-1}\xi}}{4\sigma}} }{h}_{\mathcal{S}'(\R^{d-1})},
  \end{align*}
  which is all we had to show.
\end{proof}

Two further technical estimates are needed.

\begin{prop} \label{prop:estimate}
  Let $A\in\R^{(d-1)\times(d-1)}$ be symmetric and invertible. Then, for all $s>\frac{d-1}{2}$ and
  $\alpha,\beta\in (0,1)$ there is a $C>0$ such that for all $f\in C^{0,\beta}(I;H^{s+2\alpha}(\R^{d-1}))$ and
  $|\sigma|\geq 1$ we have
   \begin{align*}
    \Big|  \int_{\R^{d-1}} (e^{-i\frac{\skp{\xi}{A^{-1}\xi}}{4\sigma}}-1) \hat f_t(\xi) 
    \,d\xi  \Big|
    &\leq C |\sigma|^{-\alpha} \|f_t\|_{H^{s+2\alpha}(\R^{d-1})}, \\
    \Big|  \int_{\R^{d-1}} (e^{-i\frac{\skp{\xi}{A^{-1}\xi}}{4\sigma}}-1) (\hat f_t(\xi)-\hat f_0(\xi)) 
    \,d\xi  \Big|    
    &\leq C|t|^\beta|\sigma|^{-\alpha} \|f\|_{C^{0,\beta}(I;H^{s+2\alpha}(\R^{d-1}))}.
  \end{align*}
\end{prop}
\begin{proof}
  %|\int_0^t ie^{is}\,ds| \leq |t|$ and $|e^{it}-1|\leq 2$ we deduce $
  %  |e^{it}-1|\leq 2^{1-\alpha}|t|^\alpha$ for all $t\in\R,\alpha>0
  From $|e^{it}-1|\leq C|t|^\alpha$ we get for $|\sigma|\geq 1$
  \begin{align*}
    \Big|  \int_{\R^{d-1}} (e^{-i\frac{\skp{\xi}{A^{-1}\xi}}{4\sigma}}-1) \hat f_t(\xi)
    \,d\xi   \Big| 
    &\leq    \int_{\R^{d-1}} \big|e^{-i\frac{\skp{\xi}{A^{-1}\xi}}{4\sigma}}-1\big| |\hat f_t(\xi)|
    \,d\xi \\
    &\leq    \int_{\R^{d-1}} C|\sigma|^{-\alpha} |\xi|^{2\alpha} |\hat f_t(\xi)| \,d\xi \\
    &\leq C|\sigma|^{-\alpha}  
    \Big(\int_{\R^{d-1}} (1+|\xi|^2)^{-s}\, \,d\xi  \Big)^{1/2}
    \Big(\int_{\R^{d-1}} (1+|\xi|^2)^{s+2\alpha} |\hat f_t(\xi)|^2 \,d\xi\Big)^{1/2}  \\
    &\leq C |\sigma|^{-\alpha} \|f_t\|_{H^{s+2\alpha}(\R^{d-1})}. 
  \end{align*}
  In the last inequality the assumption $2s> d-1$ was used. The second estimate is a direct
  consequence of the first because of $\|f_t-f\|_{H^{s+2\alpha}(\R^{d-1})} \leq C|t|^\beta
  \|f\|_{C^{0,\beta}(I;H^{s+2\alpha}(\R^{d-1}))}$.
%     \\
%   \intertext{At this point, the assumption $2s> d-1$ was used. The difference of the integrals is
%   estimated as follows: }
%      &\Big|  \int_{\R^{d-1}} (e^{-i\frac{\skp{\xi}{A_t^{-1}\xi}}{4\sigma}}-1) \hat f_t(\xi) 
%     \,d\xi  
%     - \int_{\R^{d-1}} (e^{-i\frac{\skp{\xi}{A_0^{-1}\xi}}{4\sigma}}-1) \hat f_0(\xi) 
%     \,d\xi  \Big| \\
%     &=  \Big|  \int_{\R^{d-1}} (e^{-i\frac{\skp{\xi}{(A_t^{-1}-A_0^{-1})\xi}}{4\sigma}}-1) 
%     e^{-i\frac{\skp{\xi}{A_0^{-1}\xi}}{4\sigma}}  \hat f_t(\xi) \,d\xi \Big| \\
%     &\quad+ \Big| \int_{\R^{d-1}} (e^{-i\frac{\skp{\xi}{A_0^{-1}\xi}}{4\sigma}}-1) (\hat f_t(\xi)-\hat f_0(\xi)) 
%     \,d\xi  \Big| \\
%     &\leq 
%      \int_{\R^{d-1}} \big|e^{-i\frac{\skp{\xi}{(A_t^{-1}-A_0^{-1})\xi}}{4\sigma}}-1\big| 
%     |\hat f_t(\xi)| \,d\xi  \\
%     &\quad+ \int_{\R^{d-1}} \big| e^{-i\frac{\skp{\xi}{A_0^{-1}\xi}}{4\sigma}}-1\big| |\hat f_t(\xi)-\hat
%     f_0(\xi)| \,d\xi.      
%   \end{align*}
%   Using  
%   $$
%     \|A_t^{-1}-A_0^{-1}\|\leq C|t|,\quad 
%     |t|\|f_t\|_{H^{s+2\alpha}(\R^{d-1})}
%     + \|f_t-f_0\|_{H^{s+2\alpha}(\R^{d-1})}
%     \leq |t|\|f\|_{C^1(I;H^{s+2\alpha}(\R^{d-1}))}
%   $$
%   we get the second estimate.
\end{proof}

% $\int \hat f = (2\pi)^{(d-1)/2)}f(0)$, We follow Proposition 1 (p.331) and Proposition 4 in [Stein]
% (But we do not want to assume ''smooth'' and we also want to have a better control over the constants)
% % Erklaerung, was gemacht wird.
% Now we denote by $\|A^{-1}\|_2$ the spectral radius of $A^{-1}$,
%Notation $f_t,\phi_t$
%  (Having proved the second inequality for any given such family $(f_t)$ the first inequality follows from 
%   $(g_s):=(sf_t)$. )
  
  In the next step we use the above propositions in study the asymptotics of the quantity
  \begin{equation} \label{eq:defn_Deltat}
    \Xi(f_t) := \int_{\R^{d-1}} f_t(x) e^{i\sigma \skp{x}{Ax}} \dx  
    -  f_t(0)\big(\frac{\pi}{\sigma}\big)^{\frac{d-1}{2}} |\det(A)|^{-\frac{1}{2}}
    e^{i\frac{\pi}{4}\sgn(A)}
  \end{equation}
  as $\sigma\to\infty$.

\begin{prop} \label{prop:stationary_phase}
   Let $A\in \R^{(d-1)\times (d-1)}$ be symmetric and invertible. Then, for all
  $s>\frac{d-1}{2}$ and $\alpha,\beta\in (0,1)$ there is a  $C>0$ such that for all $f\in
  C^1(I;H^{s+2\alpha}(\R^{d-1}))$ and $\sigma\geq 1$ we have
  \begin{align*}
    |\Xi(f_t)| &\leq C |\sigma|^{\frac{1-d}{2}-\alpha} \|f_t\|_{H^{s+2\alpha}(\R^{d-1})}, \\
    |\Xi(f_t)-\Xi(f_0)| &\leq C|t|^\beta |\sigma|^{\frac{1-d}{2}-\alpha}
    \|f\|_{C^{0,\beta}(I;H^{s+2\alpha}(\R^{d-1}))}.
  \end{align*}
\end{prop}
\begin{proof}
  We set $m:=  (2\sigma)^{\frac{1-d}{2}} |\det(A)|^{-\frac{1}{2}} e^{i\frac{\pi}{4}\sgn(A)}$.
%   Notice that $\sgn(A_t)$ is constant since no eigenvalue of $A_t$ approaches zero as $t$ varies over  $I$
%   because $\|A_t^{-1}\|+\|A_t\|$ is bounded. The assumptions on the family $(A_t)$ imply
%   \begin{equation} \label{eq:statphase_estimatefora}
%     |m(t)|\leq C|\sigma|^{\frac{1-d}{2}},\qquad  |m(t)-m(0)|\leq C|t||\sigma|^{\frac{1-d}{2}}.
%   \end{equation}
  From Proposition \ref{prop:FT_Fresnelphase} we get
  \begin{align*}
    \int_{\R^{d-1}} f_t(x) e^{i\sigma \skp{x}{Ax}} \,dx
    &= \skp{  e^{i\sigma \skp{x}{Ax}} }{f_t}_{\mathcal{S}'(\R^{d-1})} \\
    &= \skp{  \mathcal{F}(e^{i\sigma \skp{x}{Ax}}) }{\hat f_t}_{\mathcal{S}'(\R^{d-1})} \\
    &= m\skp{ e^{-i\frac{\skp{\xi}{A^{-1}\xi}}{4\sigma}} }{ \hat f_t}_{\mathcal{S}'(\R^{d-1})}
    \\
    &= m\Big( \skp{1}{
    \hat f_t}_{\mathcal{S}'(\R^{d-1})} +
    \skp{e^{-i\frac{\skp{\xi}{A^{-1}\xi}}{4\sigma}} -1 }{
    \hat f_t}_{\mathcal{S}'(\R^{d-1})}
    \Big)   \\
%     &= m(t) \Big( (2\pi)^{\frac{d-1}{2}} f_t(0) +
%     \skp{e^{-i\frac{\skp{\xi}{A_t^{-1}\xi}}{4\sigma}} -1 }{
%     \hat f_t}_{\mathcal{S}'(\R^{d-1})}
%     \Big)   \\
    &= m(2\pi)^{\frac{d-1}{2}} f_t(0) +
    m\skp{e^{-i\frac{\skp{\xi}{A^{-1}\xi}}{4\sigma}} -1 }{ \hat f_t}_{\mathcal{S}'(\R^{d-1})}.
  \end{align*}
  Therefore, \eqref{eq:defn_Deltat} implies
  $$
    \Xi(f_t) = m\skp{e^{-i\frac{\skp{\xi}{A^{-1}\xi}}{4\sigma}} -1 }{ \hat f_t}_{\mathcal{S}'(\R^{d-1})} 
  $$
  so that both estimates follow from Proposition \ref{prop:estimate} since $\Xi$ is linear.
%    and
%   \eqref{eq:statphase_estimatefora}. Similarly, we get again from Proposition~\ref{prop:estimate}
%   \begin{align*}
%     |\Delta_t(f)-\Delta_0(f)|
%     &= \Big| m(t)\skp{e^{-i\frac{\skp{\xi}{A_t^{-1}\xi}}{4\sigma}} -1 }{ \hat f_t}_{\mathcal{S}'(\R^{d-1})}
%     - m(0)\skp{e^{-i\frac{\skp{\xi}{A_0^{-1}\xi}}{4\sigma}} -1 }{ \hat f_0}_{\mathcal{S}'(\R^{d-1})}\Big| \\
%     &\leq |m(t)-m(0)| \big| \skp{e^{-i\frac{\skp{\xi}{A_t^{-1}\xi}}{4\sigma}} -1 }{ \hat
%     f_t}_{\mathcal{S}'(\R^{d-1})}\big| \\
%     &\; + |m(0)| \big|\skp{e^{-i\frac{\skp{\xi}{A_t^{-1}\xi}}{4\sigma}} -1 }{ \hat
%     f_t}_{\mathcal{S}'(\R^{d-1})}-\skp{e^{-i\frac{\skp{\xi}{A_0^{-1}\xi}}{4\sigma}} -1
%     }{ \hat f_0}_{\mathcal{S}'(\R^{d-1})}\big| \\
%     &\leq C|t||\sigma|^{\frac{1-d}{2}} \cdot C|\sigma|^{-\alpha}\|f_t\|_{H^{s+2\alpha}(\R^{d-1})} 
%       + C|\sigma|^{\frac{1-d}{2}} \cdot C|t||\sigma|^{-\alpha}\|f\|_{C^1(I;H^{s+2\alpha}(\R^{d-1}))}  \\
%     &\leq  C|t||\sigma|^{\frac{1-d}{2}-\alpha} \|f\|_{C^1(I;H^{s+2\alpha}(\R^{d-1}))}.
%   \end{align*}
\end{proof}
 
  \medskip
  
  \noindent \textbf{Proof of Proposition~\ref{prop:resonant_pointwise_decay}:} 
  By~\eqref{eq:defn_LambdaPsi} we have $\lambda_s(k)=\Lambda(k+2\pi s),\psi_s(x,k)=\Psi(x,k+2\pi s)$ for all
  $x\in\Omega,k\in B,s\in\Z^d$. Moreover, by (A2) we can find a $\rho_1>0$ such that the Fermi surfaces
  $F_\tau$ are regular, compact hypersurfaces with positive Gaussian curvature as well provided $|\tau-\lambda|\leq \rho_1$.
  Since $\rho>0$ will later be chosen smaller than $\rho_1$, the properties  of $\chi$ in~\eqref{eq:Def_chi}
  imply that we can apply the coarea formula to obtain for all $x,y\in\R^d$
  \begin{align*}
    K^\eps_2(x,y) 
    &=  \avint_B\sum_{s\in\Z^d} \frac{\chi(\lambda_s(k)-\lambda)}{\lambda_s(k)-\lambda-i\eps} 
    \psi_s(x,k)\ov{\psi_s(y,k)} \,dk \\
    &= \avint_B \sum_{s\in\Z^d} \frac{\chi(\Lambda(k+2\pi s)-\lambda)}{\Lambda(k+2\pi s)-\lambda-i\eps} 
    \Psi(x,k+2\pi s)\ov{\Psi(y,k+2\pi s)}  \,dk  \\
    &=   \int_{\R^d} \frac{\chi(\Lambda(k)-\lambda)}{|B|(\Lambda(k)-\lambda-i\eps)} 
    \Psi(x,k)\ov{\Psi(y,k)}  \,dk  \\
    &=  \int_\R \frac{\chi(\tau-\lambda)}{\tau-\lambda-i\eps} \Big( \int_{F_\tau}
    \frac{\Psi(x,k)\ov{\Psi(y,k)}}{|B||\nabla\Lambda(k)|} \,d\mathcal{H}^{d-1}(k) \Big) \,d\tau \\
    &=  \int_\R \frac{\chi(\tau-\lambda)}{\tau-\lambda-i\eps} \Big( \int_{F_\tau} h_{x,y}(k)
    e^{i\sigma_{x,y} \skp{v_{x,y}}{k}} \,d\mathcal{H}^{d-1}(k) \Big) \,d\tau \\
    &=  \int_\R \frac{\chi(\tau-\lambda)}{\tau-\lambda-i\eps} a_{x,y}(\tau)\,d\tau. 
  \end{align*}
  Here we used the shorthand notations $\sigma_{x,y} := |x-y|$, $v_{x,y}:=\frac{x-y}{|x-y|}$ as well as
  \begin{align}\label{eq:defn_axy_hxy}
    \begin{aligned}
    a_{x,y}(\tau)
    &:= \int_{F_\tau} h_{x,y}(k) e^{i\sigma_{x,y} \skp{v_{x,y}}{k}} \,d\mathcal{H}^{d-1}(k),\\
     h_{x,y}(k)
     &:= \frac{\Psi(x,k)\ov{\Psi(y,k)}e^{-i\skp{x-y}{k}}}{|B||\nabla\Lambda(k)|}.
     \end{aligned}
  \end{align}
  As we will see below, for general periodic Schr\"odinger-type operators such as ours the terms
  $a_{x,y}(\tau)$ play the same role as the Herglotz waves in the case of the
  Laplacian (see Chapter 4.1 in~\cite{Ruiz_LN}). Actually, when $L=-\Delta$ the integral $a_{x,y}(\tau)$ is a
  Herglotz wave over the sphere of radius $\sqrt\tau$. 
  In view of Proposition~\ref{prop:Plemelj_formula}~(i) the only reasonable candidate for a limit of
  $K_2^\eps(x,y)$ as $\eps\to 0^\pm$ is given by
  \begin{equation} \label{eq:def_Kpm2}
    K^\pm_2(x,y)
    := p.v. \int_\R \frac{\chi(\tau-\lambda)}{\tau-\lambda} a_{x,y}(\tau)\,d\tau \pm i \pi a_{x,y}(\lambda).
  \end{equation}
  It therefore remains to find $\rho\in (0,\rho_1)$ such that 
  \begin{align}\label{eq:estimates_axy}
    |a_{x,y}(\lambda)| 
    \leq C(1+|x-y|)^{\frac{1-d}{2}}, \qquad 
    |a_{x,y}(\lambda+t)-a_{x,y}(\lambda)|
    \leq C t^\beta (1+|x-y|)^{\frac{1-d}{2}}
  \end{align} 
  holds whenever $|t|\leq \rho$. Having found such a $\rho$ the cut-off function $\chi$ is
  chosen according to~\eqref{eq:Def_chi} and Proposition~\ref{prop:Plemelj_formula} implies 
  \begin{align*}
    |K^\pm_2(x,y)|
    &\leq 2\pi  \Big( \int_0^\rho \frac{ |a_{x,y}(\lambda+t)-a_{x,y}(\lambda)|}{t}\,dt + 
    |a_{x,y}(\lambda)| \Big) \\
    &\leq 2\pi \Big( C(1+|x-y|)^{\frac{1-d}{2}} \int_0^\rho t^{\beta-1} \,dt + 
     C(1+|x-y|)^{\frac{1-d}{2}} \Big) \\
    &\leq C (1+|x-y|)^{\frac{1-d}{2}}, \\
    |K^\eps_2(x,y)-K^\pm_2(x,y)|
    &\leq \int_0^\rho \frac{2\eps}{\sqrt{\eps^2+t^2}}\frac{|a_{x,y}(\lambda+t)-a_{x,y}(\lambda)|}{t}\,dt \\
    &\leq C (1+|x-y|)^{\frac{1-d}{2}} \int_0^\rho \frac{\eps}{\sqrt{\eps^2+t^2}} t^{\beta-1}\,dt \\
    &\leq C \eps^\beta (1+|x-y|)^{\frac{1-d}{2}}
  \end{align*}
  so that Proposition~\ref{prop:resonant_pointwise_decay} is proved.
  The estimates \eqref{eq:estimates_axy} will be achieved via the method of stationary phase.
  
  \medskip 
     
  We only prove the much more difficult estimates \eqref{eq:estimates_axy} for large $\sigma_{x,y}=|x-y|$. For
  notational convenience we drop the subscripts, i.e. $\sigma=\sigma_{x,y},a=a_{x,y},v=v_{x,y},h=h_{x,y}$.
  Thanks to (A2) we find a $\rho_2\in (0,\rho_1)$ and nonempty bounded open sets
  $V^1,\ldots,V^m\subset\R^{d-1}$ such that for $|t|<\rho_2$ the Fermi surfaces $F_{\lambda+t}$ admit local
  graphical representations given via functions
  \begin{equation}\label{eq:regularity_phit}
    (t,z)\mapsto \phi^j_t(z) \in C^N( (-\rho_2,\rho_2)\times V^j)\quad\text{for }j=1,\ldots,m. 
  \end{equation}
  This means that we can find  permutation matrices $\pi_1,\ldots,\pi_m:\R^d\to\R^d$ and a $C^N$-partition of
  unity $\{\eta_1,\ldots,\eta_m\}$ associated with a covering of such graphical regions such that 
  $\supp(\eta_j)\subset\subset U$ and  
  \begin{align} \label{eq:formula_alambdat}
    \begin{aligned}
     a(\lambda+t)
     &= \int_{F_{\lambda+t}} h(k) e^{i\sigma\skp{v}{k}} \,d\mathcal{H}^{d-1}(k) \\
     &= \sum_{j=1}^m \int_{\R^{d-1}} \big(\eta_j h\big)\big(\pi_j(z,\phi_t^j(z))\big) 
     e^{i\sigma\skp{v}{\pi_j(z,\phi_t^j(z))}} \,dz   \\
     &= \sum_{j=1}^m \underbrace{\int_{\R^{d-1}} f^j_t(z) e^{i\sigma \Phi^j_{t,v}(z)}  \,dz}_{=:I^j_{t,v}}
  \end{aligned}
  \end{align}
  where 
  \begin{align} \label{eq:defn_Itv}
    \begin{aligned}
    f^j_t(z)&:= \big(\eta_j h\big)\big(\pi_j(z,\phi_t^j(z))\big) \sqrt{1+|\nabla\phi_t^j(z)|^2}, \\ 
    \Phi^j_{t,v}(z) &:= \skp{v}{\pi_j(z,\phi_t^j(z))}.
    \end{aligned}
  \end{align}
  The supports of the $f_t^j$ for $|t|<\rho_2$ are contained in the projection of $\supp(\eta_j\circ\pi_j)$
  onto the first $d-1$ coordinates. The latter set may without loss of generality assumed to be a
  closed ball, for otherwise we cover the compact set $\supp(\eta_j\circ\pi_j)$ by finitely many closed balls
  and refine the partition of unity accordingly. So we may assume that there are open balls $B_j$ such that  
  \begin{equation} \label{eq:supports}
    \supp(f^j_t)\subset \ov{B_j}\subset\subset V^j\subset\R^{d-1} \quad\text{for }
    |t|\leq \rho_2,\;j=1,\ldots,m.
  \end{equation}
  By \eqref{eq:regularity_phit} the Gaussian curvature depends continuously on $t$ and, given that
  $F_\lambda$ has positive Gaussian curvature by (A2), we find that there is a $\rho_3\in (0,\rho_2)$ such
  that the Gaussian curvature  $\mathcal{K}_t$ on $F_{\lambda+t}$  satisfies
  \begin{equation} \label{eq:Gauss_curvature_graph}
    \mathcal{K}_t(\pi_j(z,\phi_t^j(z))) =
    \frac{\det(D^2\phi_t^j(z))}{(1+|\nabla\phi_t^j(z)|^2)^{\frac{d+1}{2}}} \geq c>0 
    \quad\text{for }z\in \ov{B_j},|t|\leq \rho_3,\;j=1,\ldots,m.
  \end{equation}
  For every fixed $j=1,\ldots,m$ we now establish uniform estimates for the integrals $I_{t,v}^j$ with respect
  to unit vectors $v$ from regimes:
  \begin{align} \label{eq:def_R1jR2j}
    \begin{aligned}
    R_1^j &:= \{ v\in S^{d-1}: w:=\pi_j^{-1}v \text{ satisfies }w_d=0 \text{ or }
    w_d\neq 0,-w'/w_d \notin \nabla\phi^j_0(B_j^*)\},  \\
    R_2^j &:= R_2^{j,+}\cup R_2^{j,-}, \\
    R_2^{j,\pm}&:=\{ v\in S^{d-1}: w:=\pi_j^{-1}v \text{ satisfies } \pm w_d>0,\; -w'/w_d \in
    \nabla\phi^j_0(B_j^*)\}.
  \end{aligned}
  \end{align}
  Here we used the notation $w=(w',w_d)$ with $w'=(w_1,\ldots,w_{d-1})\in\R^{d-1}$ and
  the open ball $B_j^*\supset\supset B_j$ will be chosen sufficiently small and independently of the
  $f^j_t$ below, see part~(B). Notice that the sets $R_2^{j,+},R_2^{j,-}$ are disjoint and each of
  them is connected.
%   Notice that $R_1^j$ is relatively closed and thus compact, while $R_2^j$ is relatively open, which is a
% consequence of \eqref{eq:Gauss_curvature_graph}. For   notational convenience we drop the index $j$ from
% now on and assume for simplicity that $\pi_j$ is the   identity.
  
  \medskip  
  
  \noindent \textit{(A) Uniform estimates on $R_1^j$.}\\
  Due to \eqref{eq:defn_Itv} and $\ov B_j\subset\subset B_j^*$ there are 
  $\rho_4\in (0,\rho_3)$ and $c>0$ such that the lower bound $|\nabla\Phi^j_{t,v}(z)|\geq c>0$ holds for all
  $z\in \ov{B_j},v\in R_1^j,|t|\leq \rho_4$. So Proposition~\ref{prop:nonstationary_phase} and
  $\supp(f^j_t)\subset \ov{B_j}$ (see \eqref{eq:supports}) yields for those $t,v$ and any fixed
  $\alpha\in (0,1)$ the estimates
  \begin{align*}
    |I_{t,v}^j|
    &\leq C|\sigma|^{1-N}\|f_t^j\|_{W^{N-1,1}(\R^{d-1})} \\  
    &\stackrel{\eqref{eq:regularity_phit},\eqref{eq:defn_Itv}}{\leq} 
    C|\sigma|^{1-N}\|h\|_{C^{N-1}(\ov U)} \\ 
    &\stackrel{\eqref{eq:defn_axy_hxy}}{\leq} C|\sigma|^{1-N},  \\
    |I_{t,v}^j-I_{0,v}^j|
    &\leq C|t|^\beta |\sigma|^{\alpha+1-N}\|f^j\|_{C^{0,\beta}(I;W^{N-1,1}(\R^{d-1}))} \\ 
    &\stackrel{\eqref{eq:regularity_phit},\eqref{eq:defn_Itv}}{\leq}  
    C|t|^\beta |\sigma|^{\alpha+1-N}\|h\|_{C^{N-1,\beta}(\ov U)} \\
    &\stackrel{\eqref{eq:defn_axy_hxy}}{\leq} C|t|^\beta |\sigma|^{\alpha+1-N}. 
  \end{align*}
  Here we used $\supp(\eta_j)\subset\subset U$ and that the norms of the maps $(t,z)\mapsto \Phi^j_{t,v}(z)$
  in $W^{N,\infty}(\ov B_j)$ are bounded independently of the unit vector $v$. Notice also that  
  $(x,y)\mapsto h_{x,y}(k)=h(k)$ is $\Z^d\times\Z^d$-periodic for every fixed 
   $k\in \ov U$, so that the finiteness of $\sup_{x\in\Omega} \|\Psi(x,\cdot)\|_{C^{N-1,\beta}(\bar U)}$ from
   (A2) implies $\sup_{x\in\R^d} \|\Psi(x,\cdot)\|_{C^{N-1,\beta}(\bar U)}<\infty$, which we used in the last
   inequality.
   
  \medskip

  \noindent \textit{(B) Uniform estimates on $\ov R_2^j$.}\\ 
  We choose the open balls $B_j^*,B_j^{**}$ such that $V^j\supset \supset B_j^{**}\supset\supset
  B_j^*\supset\supset B_j$ and
  \begin{equation}\label{eq:nondegeneracy_phit}
    \det(D^2\phi^j_t(z))\geq c>0 \quad\text{for all }z\in B_j^{**} \text{ and }|t|\leq\rho_4, 
  \end{equation}
  which assures that $\nabla\phi_t$ is a $C^{N-1}$-diffeomorphism on the balls $B_j^*,B_j^{**}$ (because
  these are convex).
  In particular,  from $B_j^{**}\supset\supset B_j^*$ and \eqref{eq:def_R1jR2j} it follows that there is
  $\rho_5\in (0,\rho_4)$ such that for all $v\in\ov R_2^j$ and $|t|\leq \rho_5$ the points 
  $z^j_{t,v}:= (\nabla\phi^j_t|_{B_j^{**}})^{-1}(-(\pi_j^{-1}v)'/(\pi_j^{-1}v)_d)\in B_j^{**}$ are
  well-defined with 
  $$
    \nabla\Phi^j_{t,v}(z^j_{t,v}) = 0 \qquad\text{whenever } v\in \ov R_2^j,\;|t|\leq \rho_5,
  $$
  see~\eqref{eq:defn_Itv}. Having thus determined the unique point of stationary phase in $B_j^{**}$ 
  we now make a local coordinate transformation around $z^j_{t,v}$ which makes 
  the phase function $\Phi^j_{t,v}$ look like a quadratic form.  
  Since $z^j_{t,v}\in B_j^{**}$ and $z^j_{0,v}\in B_j^*$ for all $|t|\leq \rho_5,v\in \ov R_2^j$ (see
  \eqref{eq:def_R1jR2j}) and $\Phi^j_{t,v}\in C^N(B_j^{**})$, the Morse Lemma provides  
  $\delta_j^*,\delta_j^{**}>0$ and $\rho_6\in (0,\rho_5)$ and 
  $C^{N-2}$-diffeomorphisms $\psi^j_{t,v}:B_{\delta_j^*}(0)\to
  \psi^j_{t,v}(B_{\delta_j^*}(0))$ with 
  $$
    \psi^j_{t,v}(0)=0,\qquad
    B_{\delta_j^{**}}(0) \subset \psi^j_{t,v}(B_{\delta_j^*}(0)),\qquad
    z^j_{t,v}+\psi^j_{t,v}(B_{\delta_j^*}(0))\subset B_j^{**}
  $$ 
  and
  \begin{align} \label{eq:local_coordinates}
    \begin{aligned}
    & \Phi^j_{t,v}(z^j_{t,v}+\psi^j_{t,v}(y)) - \Phi^j_{t,v}(z^j_{t,v}) =
    \pm\Big(-y_1^2-\ldots-y_m^2+y_{m+1}^2+\ldots+y_{d-1}^2\Big) =: \pm \skp{y}{Ay} \\ 
    &\text{whenever}\qquad y\in B_{\delta_j^*}(0)\subset\R^{d-1}, v\in \ov R_2^{j,\pm},\;|t|<\rho_6.
  \end{aligned}
  \end{align}
  Notice that $m$ and hence the matrix $A$ are independent of $t,v$ since $R_2^{j,+},R_2^{j,-}$ are
  connected. Moreover, $\delta_j^*, \delta_j^{**}>0$ may be
  chosen independently of $t,v$ since $|\det(\Jac(\psi^j_{t,v})(0))|$ is bounded from below and from above
  for $|t|\leq \rho_6,v\in \ov R_2^j$ as follows from \eqref{eq:regularity_phit}, \eqref{eq:nondegeneracy_phit}
  and
  \begin{align}\label{eq:formula_A}
    \begin{aligned}
    A 
    &= \frac{1}{2}\Jac(\psi^j_{t,v})(0)^T D^2\Phi^j_{t,v}(z^j_{t,v})\Jac(\psi^j_{t,v})(0) \\
    &\stackrel{\eqref{eq:defn_Itv}}{=} 
    \pm \frac{|(\pi_j^{-1}v)_d|}{2}\Jac(\psi^j_{t,v})(0)^T D^2\phi^j_t(z^j_{t,v})\Jac(\psi^j_{t,v})'(0) .
    \end{aligned}
  \end{align}
  Let us finally set $\rho:=\rho_6$ so that it remains to find uniform bounds for integrals
  $I^j_{t,v}$ for $|t|\leq \rho$ and $v\in\ov R_2^j$.
  
  \medskip
  
  To this end we choose a cut-off function $\chi_j\in C_0^\infty(\R^{d-1})$ such that $\chi_j
  \equiv 1$ near $0$ with support in $B_{\delta_j^{**}}(0)$ and set $I^j_{t,v}=I_{t,v}^{j,1}+I_{t,v}^{j,2}$
  where
  \begin{align*}
    I_{t,v}^{j,1} 
    &= \int_{\R^{d-1}} f^j_t(z) (1-\chi_j(z- z^j_{t,v})) e^{i\sigma\Phi^j_{t,v}(z)}\,dz,  \\
    I_{t,v}^{j,2} 
    &= \int_{\R^{d-1}} f^j_t(z) \chi_j(z- z^j_{t,v}) e^{i\sigma\Phi^j_{t,v}(z)}\,dz.  
  \end{align*}
  We first discuss $I_{t,v}^{j,1}$. As in (A) the phase function $\Phi^j_{t,v}$ satisfies
  $|\nabla\Phi^j_{t,v}|\geq c>0$ on the support of $z\mapsto f^j_t(z) (1-\chi_j(z- z^j_{t,v}))$, since
  the only stationary point of $\Phi^j_{t,v}$ in $B_j^{**}$ (and hence in $\supp(f^j_t)$)
  is $z^j_{t,v}$, but $\chi_j\equiv 1$ near zero. So the same estimates as in part (A)
  yield
  \begin{align} \label{eq:estimates_Itvj1}
    |I_{t,v}^{j,1}|\leq C|\sigma|^{1-N}, \qquad
    |I_{t,v}^{j,1}-I_{0,v}^{j,1}| \leq C|t|^\beta |\sigma|^{\alpha+1-N}. 
  \end{align}
  The estimates for $I^{j,2}_{t,v}$ are based on Proposition~\ref{prop:FT_Fresnelphase} and
  Proposition~\ref{prop:stationary_phase}. Performing the change of variables
  from~\eqref{eq:local_coordinates} we obtain
  \begin{align} \label{eq:estimate_Itv2}
    \begin{aligned}
    &I^{j,2}_{t,v} 
    =  e^{i\sigma \Phi^j_{t,v}(z^j_{t,v})}  \int_{\R^d} g^j_{t,v}(y) 
    e^{\pm i\sigma \skp{y}{Ay}}\,dy  \\
    &\text{where}\qquad  
    g^j_{t,v}(y) := f^j_t(z^j_{t,v}+\psi^j_{t,v}(y)) \chi_j(\psi^j_{t,v}(y))|\det(\Jac(\psi^j_{t,v})'(y))|,\quad
    v\in R_2^{j,\pm}.
  \end{aligned}
  \end{align}
  In view of \eqref{eq:defn_Deltat} we first calculate
  $g^j_{t,v}(0)|\det(A)|^{-1/2}$. Exploiting 
  \begin{equation} \label{eq:Jacobian}
    1+|\nabla\phi^j_t(z^j_{v,t})|^2 = 1+ |v'|^2|v_d|^{-2} = |v_d|^{-2}
  \end{equation} 
  we get  
  \begin{align*}
    g^j_{t,v}(0)  |\det(A)|^{-1/2} 
    &\stackrel{\eqref{eq:estimate_Itv2}}{=}
     f^j_t(z^j_{t,v}) \chi_j(0)|\det(\Jac(\psi^j_{t,v})(0))| |\det(A)|^{-1/2}  \\
	&\stackrel{\eqref{eq:formula_A}}{=}  2^{\frac{d-1}{2}} f^j_t(z_{t,v})
	|\det(D^2\Phi^j_{t,v}(z^j_{t,v}))|^{-1/2}  \\
    &\stackrel{\eqref{eq:defn_Itv}}{=} 2^{\frac{d-1}{2}} \big(
    (\eta_j h)\circ\pi_j\big)\big(z^j_{t,v},\phi^j_t(z^j_{t,v})\big) \sqrt{1+|\nabla\phi^j_t(z^j_{v,t})|^2}
    |\det(v_d D^2\phi^j_t(z^j_{t,v}))|^{-1/2} \\
    &\stackrel{\eqref{eq:Jacobian}}{=} 2^{\frac{d-1}{2}}(\eta_j
    h)\big(\pi_j(z^j_{t,v},\phi^j_t(z^j_{t,v}))\big) \left(
    \frac{\det(D^2\phi^j_t(z^j_{t,v}))}{(1+|\nabla\phi^j_t(z^j_{t,v})|^2)^{\frac{d+1}{2}}}  \right)^{-1/2}
      \\
    &\stackrel{\eqref{eq:Gauss_curvature_graph}}{=} 2^{\frac{d-1}{2}} (\eta_j
    h)\big(\pi_j(z^j_{t,v},\phi^j_t(z^j_{t,v}))\big)
    \mathcal{K}_t\big(\pi_j(z^j_{t,v},\phi^j_t(z^j_{t,v}))\big)^{-1/2}.
  \end{align*}
  So the definition of $\Xi$ and~\eqref{eq:estimate_Itv2} imply for $v\in R_2^{j,\pm}$
  \begin{align}  \label{eq:formula_Itv2}
    \begin{aligned}
    I_{t,v}^{j,2}  
    &= e^{i(\sigma\Phi^j_{t,v}(z^j_{t,v})\pm \frac{\pi}{4}\sgn(A))}
    \Big(\frac{2\pi}{\sigma}\Big)^{\frac{d-1}{2}} (\eta_jh)\big( \pi_j(z_{t,v},\phi_t^j(z^j_{t,v}))\big)
    \mathcal{K}_t\big(\pi_j(z^j_{t,v},\phi_t^j(z^j_{t,v}))\big)^{-1/2} \\
    &+ e^{i\sigma\Phi^j_{t,v}(z^j_{t,v})}  \Xi (g^j_{t,v}). 
    \end{aligned}
  \end{align}
  The second term is estimated with the aid of
  Proposition~\ref{prop:stationary_phase}, so we choose $s\in\R,\alpha\in (0,1)$ such that
  $\frac{d-1}{2}<s<s+2\alpha\leq N-1$. This is possible due to $N>\frac{d+1}{2}$, see assumption (A2). So the
  Proposition yields for all $|\sigma|\geq 1$   
  \begin{align}
    \begin{aligned}
     |\Xi(g^j_{t,v})| 
    &\leq  C |\sigma|^{\frac{1-d}{2}-\alpha} \|g^j_{t,v}\|_{H^{s+2\alpha}(\R^{d-1})} 
     \;\leq\;  C|\sigma|^{\frac{1-d}{2}-\alpha} \|f^j_t\|_{H^{s+2\alpha}(\R^{d-1})} 
      \\ 
    &\leq  C |\sigma|^{\frac{1-d}{2}-\alpha}\|f^j_t\|_{H^{N-1}(\R^{d-1})}          
    \;\leq\; C|\sigma|^{\frac{1-d}{2}-\alpha} \|h\|_{C^{N-1}(\ov U)}   \\
    &\leq  C|\sigma|^{\frac{1-d}{2}-\alpha}, \\   
    |\Xi(g^j_{t,v}) - \Xi(g^j_{0,v})| 
    &\leq   C|\sigma|^{\frac{1-d}{2}-\alpha} \|f^j_t-f^j_0\|_{H^{N-1}(\R^{d-1})}
    \leq C|t|^\beta |\sigma|^{\frac{1-d}{2}-\alpha}. \label{eq:estimates_Ijtv2}
    \end{aligned}
  \end{align}
  Combining the estimates for $I_{t,v}^{j,1},I_{t,v}^{j,2}$ resulting from
  \eqref{eq:estimates_Itvj1} and $N>\frac{d+1}{2}$ and \eqref{eq:formula_Itv2},\eqref{eq:estimates_Ijtv2}
  we get the desired estimates. \qed 
  
  \bigskip
  
  We point out that the computations of the previous Proposition even reveal the asymptotics of
  $K_2^\pm(x,y)$ as $|x-y|\to\infty$. To see this we recall from the above proof 
  $$
    a_{x,y}(\lambda+t)
    = \int_{F_{\lambda+t}} h_{x,y}(k)\,d\mathcal H^{d-1}(k) 
    \stackrel{\eqref{eq:formula_alambdat}}{=} \sum_{j=0}^m I^j_{t,v}. 
  $$
  The estimates of these integrals represented the main part of the Proposition.
  For given $x\neq y$ we define the resonant points on the Fermi surfaces  
  \begin{align*}
    \mathfrak R_{\lambda+t}^\pm(x,y) &:= \Big\{k\in F_{\lambda+t} :
    \nu_{\lambda+t}(k)=\pm\frac{x-y}{|x-y|}\Big\}, \qquad\text{where } \\
    \nu_{\lambda+t}(\pi_j(z,\phi_t^j(z)))
    &= \frac{\pi_j(\nabla\phi_t^j(z),-1)}{\sqrt{1+|\nabla\phi_t^j(z)|^2}}.
  \end{align*}
  So $\nu_{\lambda+t}$ is the outer unit normal vector field along $F_{\lambda+t}$.
  The reason for this definition is 
  $$
    \nu_{\lambda+t}(k)=\pm \frac{x-y}{|x-y|}
    \quad\Leftrightarrow\quad
    k=\pi_j(z^j_{t,v},\phi^j_t(z^j_{t,v}))) \text{ for }v=\frac{x-y}{|x-y|}\in R_2^{j,\mp},\;j=0,\ldots,m.
  $$
  Notice that $\sign((\pi_jv)_d)=\skp{\nu(k)}{\frac{x-y}{|x-y|}}$. 
  So, writing $\sigma=\sigma_{x,y}=|x-y|,v=v_{x,y}=\frac{x-y}{|x-y|}$ as above, we get that the dominant part
  of $a_{x,y}(\lambda+t)$ as $|\sigma|=|x-y|\to\infty$ is (see~\eqref{eq:formula_Itv2})
  \begin{align*}
     &\Big(\frac{2\pi}{\sigma}\Big)^{\frac{d-1}{2}}  
       \sum_{j=0}^m e^{i(\sigma\Phi^j_{t,v}(z^j_{t,v})+ \sign((\pi_j^{-1}v)_d) \frac{\pi}{4})} (\eta_j
       h_{x,y})\big(\pi_j(z^j_{t,v},\phi^j_{t,v}(z^j_{t,v}))\big) \mathcal
       K_t\big(\pi_j(z^j_{t,v},\phi^j_{t,v}(z^j_{t,v}))\big)^{-1/2} \\ 
    &= \Big(\frac{2\pi}{|x-y|}\Big)^{\frac{d-1}{2}} 
       \sum_{j=0}^m \sum_{k\in\mathfrak R_{\lambda+t}^\pm(x,y)} e^{i(\sigma\skp{v}{k}\mp \frac{\pi}{4}\sgn(A))}
       (\eta_j h_{x,y})(k)  \mathcal K_t(k)^{-1/2}  \\
    &= \Big(\frac{2\pi}{|x-y|}\Big)^{\frac{d-1}{2}} 
    \sum_{k\in\mathfrak R_{\lambda+t}^\pm(x,y)} e^{i(\skp{x-y}{k}\mp \frac{\pi}{4}\sgn(A))}
        h_{x,y}(k)  \mathcal K_t(k)^{-1/2}.   
  \end{align*}
  This formulas allows to identify a farfield at the frequency $\lambda+t$. 
  
  \medskip
  
  Let us demonstrate what the computations reveal in the case of the Laplacian.
  From~\eqref{eq:eigenpairs_constantpotential} we deduce that in this special case the Fermi surface
  $F_\lambda=\{k\in\R^d: \Lambda(k)=|k|^2=\lambda\}$ is the sphere of radius $\sqrt\lambda$ centered at the
  origin. So  $\nu(k)=k/|k|$ for all $k\in F_\lambda$ and the principal curvatures are $\lambda^{-1/2}>0$,
  which implies $\sgn(A)=d-1$ and $\mathcal K_0 \equiv \lambda^{(1-d)/2}$. This yields 
  $$
    \mathfrak R_\lambda^\pm(x,y) = \Big\{ \pm \sqrt\lambda \frac{x-y}{|x-y|}\Big\}
  $$ 
  as well as
  $$
    h_{x,y}(k) 
    \stackrel{\eqref{eq:defn_axy_hxy}}{=}
     \frac{\Psi(x,k)\ov{\Psi(y,k)}e^{-i\skp{x-y}{k}}}{|B||\nabla\Lambda(k)|}
    \stackrel{\eqref{eq:eigenpairs_constantpotential}}{=}
    \frac{e^{i\skp{x}{k}}e^{-i\skp{y}{k}}e^{-i\skp{x-y}{k}}}{(2\pi)^d|2k|} 
    = \frac{1}{2(2\pi)^d \sqrt\lambda}.
  $$
  From~\eqref{eq:def_Kpm2} we therefore get that the point evaluation part of $K_2^\pm(x,y)$ is
  \begin{align*}
     \pm i\pi a_{x,y}(\lambda)  
    &\sim  \pm i\pi\cdot \Big(\frac{2\pi}{|x-y|}\Big)^{\frac{d-1}{2}}  
    \sum_{k\in\mathfrak R_\lambda^\pm(x,y)} e^{i(\skp{x-y}{k}-\frac{d-1}{4}\pi)}  
    \cdot \frac{\lambda^{\frac{d-1}{4}}}{(2\pi)^d2\sqrt\lambda}  \\ 
   &=  \pm i \frac{1}{2\sqrt\lambda}
    \frac{\Real(e^{i(\sqrt\lambda |x-y|-\frac{d-1}{4}\pi)})}{(2\pi\sqrt\lambda |x-y|)^{\frac{d-1}{2}}} \\
   &= \pm  i \Imag\left(\frac{1}{2\sqrt\lambda}  
    \frac{e^{i(\sqrt\lambda |x-y|-\frac{d-3}{4}\pi)}}{(2\pi\sqrt\lambda |x-y|)^{\frac{d-1}{2}}}\right)
    \qquad\text{as }|x-y|\to\infty.    
  \end{align*} 
  We see that these asymptotics are consistent with the well-known asymptotics of the imaginary part of the
  outgoing respectively incoming Green's function $G_\pm$ for the Helmholtz operator $-\Delta-(\lambda\pm i0)$ for
  $\lambda>0$. Indeed, we have as well
  \begin{align*}
    i\Imag(G_\pm(x,y))
    &\sim \pm i\Imag\left(\frac{i}{4} \Big(\frac{2\pi|x-y|}{\sqrt \lambda}\Big)^{\frac{2-d}{2}}
    H^{(1)}_{\frac{d-2}{2}}(\sqrt \lambda|x-y|) \right)\\
    &\sim \pm i\Imag\left(\frac{i}{4} \Big(\frac{2\pi|x-y|}{\sqrt \lambda}\Big)^{\frac{2-d}{2}}
    \Big(\frac{2}{\pi \sqrt \lambda |x-y| }\Big)^{1/2} e^{i(\sqrt \lambda|x-y|- \frac{d-1}{4}\pi)} \right)\\ 
    %&\sim  \pm \frac{1}{2} \Big(\frac{2\pi|x-y|}{\sqrt \lambda}\Big)^{\frac{2-d}{2}} \Big(\frac{1}{2\pi
    %\sqrt{\lambda}|x-y|}\Big)^{1/2} e^{i(\sqrt \lambda |x-y| -\frac{d-3}{4}\pi)} \\
    %&\sim  \frac{1}{2\sqrt a} \Big( \frac{2\pi|x-y|}{\sqrt a}\Big)^{\frac{1-d}{2}}   
    %   e^{i(\sqrt a|x-y|-(d-3)\pi/4)}  \\
    &\sim \pm i\Imag\left( \frac{1}{2\sqrt\lambda}       
       \frac{e^{i(\sqrt \lambda|x-y|-\frac{(d-3)\pi}{4})}}{(2\pi\sqrt\lambda |x-y|)^{\frac{d-1}{2}}}\right)
       \qquad\text{as }|x-y|\to\infty.  
  \end{align*}
  so that the above computations show that the imaginary parts of both asymptotics coincide. The
  corresponding computations concerning the asymptotics of the real part seem to be much more delicate and
  thorough discussion of those remains to be done elsewhere.
  
  \medskip
    
    Let us finally comment on the fact that in (A2) we imposed the positivity of the Gaussian curvature
    on the whole of the Fermi surface $F_\lambda$. The considerations above lead to the observation that 
    the decay rate  $|x-y|^{\frac{1-d}{2}}$ of the fundamental solution is valid along
    all rays $x-y \in\R \nu(k)$ where $\nu$ is the normal vector field along the Fermis surface 
    and $k$ belongs to a positively curved pieces of $F_\lambda$. The same decay rate is actually expected to
    hold for negatively curved pieces since the method of stationary phase works as well and the same
    computations as above, up to changes of signs, remain valid.
    So the only problematic points appear to be those where the curvature of the Fermi surface vanishes,
    since the coefficient function $\mathcal K_t^{-1/2}$ becomes infinite at these points. It seems reasonable
    to expect that that depending on the number (and possibly order) of vanishing principal curvatures, the
    fundamental solution has an even weaker decay rate, see~\cite{Littman_Fourier}
    or~\cite{Ste_harmonic_analysis} pp.360 for related results.
     %Understanding the scattering properties of
    %optical materials described by such Fermi surfaces is a challenge for the future.

%   \medskip
%   
%   The almost explicit construction of $K^\pm_2(x,y)$ allows for a more detailed analysis of its properties,
%   especially concerning radiation conditions an farfield expansions for solutions of $L u -\lambda u=f$ in
%   $\R^d$. In view of the above proof we think that the asymptotic properties of such a
%   solutions $u$ are governed by the terms $a_{x,y}(\tau)$ for $\tau\approx\lambda$. In these terms the
%   integrals over the Fermi surfaces have to be analyzed and the formula \eqref{eq:formula_Itv2} together with
%   the following estimates suggest that the leading terms at infinity are approximately given by 
%   $$
%     a_{x,y}(\tau)
%     \approx \big(\frac{2\pi}{|x-y|}\big)^{\frac{d-1}{2}}  
%     \sum_{k\in F_\tau \text{normal}^* \text{ to }x-y} 
%     \frac{\zeta_k}{\sqrt{K_\tau(k)}} h(k) 
%   $$
%   where $K_\tau$ denotes the Gaussian curvature of $F_\tau$, $h_{x,y}(k)$ is defined as above and $\zeta_k$  
%   is a complex prefactor with modulus. The condition ''$k$ normal$^*$ to $x-y$'' is a
%   reformulation of the condition $\nabla\phi_t(z)=-v'/v_d$ because of $v=v_{x,y}=\frac{x-y}{|x-y|}$. The star
%   $^*$ refers to the fact that only positive multiples of the normal are taken. \r{TODO}.

\section{Proof of Proposition \ref{prop:estimate_FBtransform_kernel}}  \label{sec:proof_prop_II}

  \begin{prop} \label{prop:gj}
    Let $R_j$ be defined as in \eqref{eq:def_Repsj} and set 
    $$
      g_j(\xi) := \sum_{m\in R_j} e^{im\xi} \qquad \text{for }\xi\in\R^d,\;j\in \N.
    $$ 
    Then for all compact subsets $K\subset\R^{d-1}$ and $\delta>0$ there is a $C>0$ such that for all
    $\xi'\in\R^{d-1},s,t\in\R$ and $j\in\N$ the following estimates hold:
    $$
      \int_K |g_j(\xi',s)|\,d\xi'\leq C 2^{j(1+\delta)} ,\qquad
      \int_K |g_j(\xi',s)-g_j(\xi',t)|\,d\xi' \leq C2^{j(1+\delta)}|s-t|^{\delta/2}. 
    $$
  \end{prop}
  \begin{proof}
    By definition of $R_j$ the function $g_j$ can be written as
    $$
      g_j(\xi) =  \prod_{p=1}^d D_j(\xi_p)-\prod_{p=1}^d D_{j-1}(\xi_p), \qquad\text{where } 
      D_j(z) := \sum_{m=-2^j}^{2^j} e^{imz}  = \frac{\sin((2^j+\frac{1}{2})z)}{\sin(\frac{z}{2})}.
    $$
    The Dirichlet kernels satisfy the estimates $|D_j(z)| \leq C 2^j$ and hence for
    $\xi'=(\xi_1,\ldots,\xi_{d-1})\in K$ 
    $$
      |g_j(\xi',s)| 
       \leq \sum_{\iota\in\{j-1,j\}}  \Big(\prod_{p=1}^{d-1} |D_\iota(\xi_p)|\Big)|D_\iota(s)| 
       \leq  C 2^j \sum_{\iota\in\{j-1,j\}}\prod_{p=1}^{d-1} |D_\iota(\xi_p)|. 
    $$
    Similarly, the estimate  
    $$
      |D_j(s)-D_j(t)| 
      \leq \sum_{p=-2^j}^{2^j} |e^{ipt} (e^{ip(s-t)}-1)|
      \leq \sum_{p=-2^j}^{2^j} 2|p(s-t)|^{\delta/2}   
      \leq C2^{j(1+\delta/2)} |s-t|^{\delta/2} 
    $$
    implies 
    \begin{align*}
%        |g_j(\xi',s)| 
%        &\leq \sum_{\iota\in\{j-1,j\}}  \Big(\prod_{p=1}^{d-1} |D_\iota(\xi_p)|\Big)|D_\iota(s)| 
%        \leq  C 2^j \sum_{\iota\in\{j-1,j\}}\prod_{p=1}^{d-1} |D_\iota(\xi_p)| \\ 
%     \intertext{as well as}
       |g_j(\xi',s)-g_j(\xi',t)|
       &= \Big| \Big(\prod_{p=1}^{d-1} D_j(\xi_p)\Big) D_j(s) -  \Big(\prod_{p=1}^{d-1} D_{j-1}(\xi_p)\Big)
       D_{j-1}(s) \\
       &\quad -  \Big(\prod_{p=1}^{d-1} D_j(\xi_p)\Big) D_j(t)+  \Big(\prod_{p=1}^{d-1} D_{j-1}(\xi_p)\Big)
       D_{j-1}(t) \Big|  \\
      &\leq  \sum_{\iota\in\{j-1,j\}} \Big(\prod_{p=1}^{d-1} |D_\iota(\xi_p)|\Big) |D_\iota(s)-D_\iota(t)|
      \\
      &\leq C 2^{j(1+\delta/2)}|s-t|^{\delta/2}   \sum_{\iota\in\{j-1,j\}} \Big(\prod_{p=1}^{d-1}
      |D_\iota(\xi_p)|\Big).
    \end{align*}
    Integrating the first estimate with respect to $\xi'$ over $K\subset [-M,M]^{d-1}$ gives
    \begin{align*}
      \int_K |g_j(\xi',s)|\,d\xi'   
      \leq C 2^j \sum_{\iota\in\{j-1,j\}} \Big(\int_{-M}^M |D_\iota|\Big)^{d-1} 
      \leq C 2^j \sum_{\iota\in\{j-1,j\}} \iota^{d-1} 
      \leq C 2^{j(1+\delta)}  
    \end{align*}
    where the final $C$ depends on $M$ and thus on the compact set $K$, but not on $j$. The estimate for the
    integral of the Dirichlet kernel over $[-M,M]$ can be found in \cite{Kach_ST_theorem}~(Lemma~7).
    The estimate  for the other integral is similar.
  \end{proof}

  \medskip
  
  \noindent 
  \textbf{Proof of Proposition~\ref{prop:estimate_FBtransform_kernel}:} We have to show that for all 
  $\delta>0$ there is a $C_{\delta}>0$ such that for all $\eps\in\R\sm\{0\}$ the following inequality holds: 
  \begin{align*}
    \sup_{x,y\in\Omega, l\in B} \big| U(K^{\eps,j}_{2}(\cdot,y))(x,l)\big| 
    &\leq C_{\delta} 2^{j(1+\delta)}  \;\;\qquad\text{for all } j\in\N_0 \text{ and } \\
    \sup_{x,y\in\Omega, l\in B} \big| U(K^{\eps,j}_{2}(\cdot,y)-K^{\pm,j}_2(\cdot,y))(x,l)\big| 
    &\leq C \eps^\beta 2^{j(1+\delta)}  \qquad\text{for all } j\in\N_0 \text{ as }\eps\to 0^\pm.
  \end{align*}
  We only prove the first inequality in detail. The formulas for $K^\eps_2,K_2^{\eps,j}$ from
  \eqref{eq:def_Keps_12},\eqref{eq:def_Repsj} yield  for all $x,y\in\Omega$ and $l\in B$
  \begin{align*} 
    \begin{aligned}
    U(K^{\eps,j}_{2}(\cdot,y))(x,l)
    &= \sum_{m\in\Z^d}  e^{iml}  K_{2}^\eps(x-m,y) 1_{R_j}([x-m]-[y])  \\
    &= \sum_{m\in\Z^d}  e^{iml} 1_{R_j}(-m) \avint_B
    \sum_{s\in\Z^d} \frac{\chi(\lambda_s(k)-\lambda)}{\lambda_s(k)-\lambda-i\eps} \psi_s(x-m,k)\ov{\psi_s(y,k)}
    \,dk  \\
    &=  \avint_B \sum_{s\in\Z^d}  \frac{\chi(\lambda_s(k)-\lambda)}{\lambda_s(k)-\lambda-i\eps}
    \psi_s(x,k)\ov{\psi_s(y,k)} \Big( \sum_{m\in R_j}e^{im(l-k)}  \Big)\,dk \\
    &=  \avint_B \sum_{s\in\Z^d}  \frac{\chi(\lambda_s(k)-\lambda)}{\lambda_s(k)-\lambda-i\eps}
    \psi_s(x,k)\ov{\psi_s(y,k)} g_j(l-k) \,dk.   
    \end{aligned}
  \end{align*}
  In order to simplify this expression further we use assumption (A2). From \eqref{eq:defn_LambdaPsi}
  and~\eqref{eq:Def_chi} we get $\lambda_s(k)=\Lambda(k+2\pi s),\psi_s(x,k)=\Psi(x,k+2\pi s)$ for all
  $s\in\Z^d$ such that $\chi(\lambda_s(k)-\lambda)\neq 0$. Using the $2\pi\Z^d$-periodicity of $g_j$ we arrive at
  \begin{align*} 
    U(K^{\eps,j}_{2}(\cdot,y))(x,l)
    &= \sum_{s\in\Z^d} \int_{B+2\pi s} \frac{\chi(\Lambda(k)-\lambda)}{|B|(\Lambda(k)-\lambda-i\eps)}
    \Psi(x,k)\ov{\Psi(y,k)} g_j(l-k)  \,dk \\
    &= \int_{\R^d} \frac{\chi(\Lambda(k)-\lambda)}{|B|(\Lambda(k)-\lambda-i\eps)}
    \Psi(x,k)\ov{\Psi(y,k)} g_j(l-k) \,dk \\
    &= \int_\R \frac{\chi(\tau-\lambda)}{\tau-\lambda-i\eps}
    \Big(\int_{F_\tau} \frac{\Psi(x,k)\ov{\Psi(y,k)} g_j(l-k)}{|B||\nabla\Lambda(k)|}
    \,d\mathcal{H}^{d-1}(k)\Big)\,d\tau \\
    &= \int_\R \frac{\chi(\tau-\lambda)}{\tau-\lambda-i\eps}
    \Big(\int_{F_\tau} h_{x,y}(k) \,d\mathcal{H}^{d-1}(k)\Big)\,d\tau 
    %\\  &= \int_\R \frac{a_{x,y}(\tau)}{\tau-\lambda-i\eps}  \,d\tau
  \end{align*} 
  where 
  \begin{align*}
    h_{x,y}(k)
    := \frac{\Psi(x,k)\ov{\Psi(y,k)} g_j(l-k)}{|B||\nabla\Lambda(k)|}.
    %,\qquad
    %a_{x,y}(\tau)
    %:= \chi(\tau-\lambda) \int_{F_\tau} h_{x,y}(k) \,d\mathcal{H}^{d-1}(k) 
  \end{align*}
  In the third equality above we used the coarea formula.  Assumption~(A2) gives 
  $$
     |h_{x,y}(k)|\leq C|g_j(l-k)|,\qquad |h_{x,y}(k)-h_{x,y}(\tilde k)|\leq C|g_j(l-k)-g_j(l-\tilde k)| 
  $$
  for some $C>0$ and all $x,y\in\R^d,l\in B$ and $j\in\N$.   
  In view of Proposition \ref{prop:Plemelj_formula}~(ii) we may bound the
  expression $U(K^{\eps,j}_{2}(\cdot,y))(x,l)$ by estimating the difference
  \begin{align*}
    \int_{F_{\lambda+t}}  h_{x,y}(k) \,d\mathcal{H}^{d-1}(k) 
    - \int_{F_\lambda}  h_{x,y}(k) \,d\mathcal{H}^{d-1}(k).
  \end{align*}
  
  \medskip
  
  As in the proof of Proposition~\ref{prop:resonant_pointwise_decay} we may content ourselves with proving the
  estimates on pieces of the Fermi surfaces that are parametrized over the first $d-1$ Euclidean coordinates
   according to $k=(z,\phi_t(z))$ for $z$ belonging to some open bounded set $V\subset\R^{d-1}$ where
   $(t,z)\mapsto \phi_t(z)$ is of class $C^N$. In particular, we have $\|\phi_t-\phi_0\|_{C^1(V)}\leq  C|t|$ for all
  $t\in I$.
   So we get 
  \begin{align*}
%   & \Big| \int_{F_{\lambda+t}\cap U} h(k)\,d\mathcal{H}^{d-1}(k) -
%   \int_{F_\lambda\cap U} h(k)\,d\mathcal{H}^{d-1}(k) \Big|  \\
%   &\leq
  &\int_{V} 
    \Big| h_{x,y}(z,\phi_t(z)) (1+|\nabla\phi_t(z)|^2)^{1/2} -
    h_{x,y}(z,\phi_0(z)) (1+|\nabla\phi_0(z)|^2)^{1/2}\Big| \,dz  \\
  &\leq  \int_{V}   \Big( |h_{x,y}(z,\phi_t(z))|
    \big|(1+|\nabla\phi_t(z)|^2)^{1/2}-(1+|\nabla\phi_0(z)|^2)^{1/2}\big|
    \\
    &\qquad + \big|h_{x,y}(z,\phi_t(z))- h_{x,y}(z,\phi_0(z))\big|(1+|\nabla\phi_0(z)|^2)^{1/2}
    \Big) \,dz  \\
  &\leq C \int_V |h_{x,y}(z,\phi_t(z))|  |\nabla\phi_t(z)-\nabla\phi_0(z)| \,dz
    + C\int_V |h_{x,y}(z,\phi_t(z))- h_{x,y}(z,\phi_0(z))|   \,dz \\  
  %\intertext{Using Proposition~\ref{prop:gj} and that the $\psi_s$ appearing in the definition of $\Psi_j$
  % are bounded (there are only finitely many so that we need not invoke assumption \ldots) we arrive at the
  %estimate } 
  %& \Big| \int_{F_{\lambda+t}\cap U}h_{x,y,l,j}(k)\,d\mathcal{H}^{d-1}(k) -
  %\int_{F_\lambda\cap U} h_{x,y,l,j}(k)\,d\mathcal{H}^{d-1}(k) \Big|  \\
  &\leq C |t| \int_V |g_j(l'-z,l_d-\phi_t(z))| \,dz 
  + C \int_V |g_j(l'-z,l_d-\phi_t(z))- g_j(l'-z,l_d-\phi_0(z))|\,dz   \\
  &\leq C ( 2^{j(1+\delta)} |t|  +  2^{j(1+\delta)} |t|^{\delta/2}) \\
  &\leq  C 2^{j(1+\delta)} |t|^{\delta/2}.
  \end{align*}
  In the second last inequality we used Proposition~\ref{prop:gj}. 
  Proceeding similarly, we get from the same proposition
  $$
%     \Big| \int_{F_{\lambda}\cap U} h_{x,y}(k)\,d\mathcal{H}^{d-1}(k) \Big|
%     &\leq 
    \int_V |h_{x,y}(z,\phi_t(z))|(1+|\nabla\phi_t(z)|^2)^{1/2} \,dz
    \leq C \int_V |g_j(l'-z,l_d-\phi_t(z))| \,dz
    \leq C 2^{j(1+\delta)}.
  $$
  Combining Proposition~\ref{prop:Plemelj_formula}~(ii) and the above estimates we arrive at  
  \begin{align*}
    |U(K^{\eps,j}_{2}(\cdot,y))(x,l)|
    &\leq C \Big| \int_{F_\lambda}  h_{x,y}(k) \,d\mathcal{H}^{d-1}(k) \Big|  \\
    &\;+ C\int_0^\rho \frac{1}{t}\Big|\int_{F_{\lambda+t}} h_{x,y}(k) \,d\mathcal{H}^{d-1}(k) -
     \int_{F_\lambda}  h_{x,y}(k) \,d\mathcal{H}^{d-1}(k)\Big|\,dt \\
    &\leq C 2^{j(1+\delta)}. 
   \end{align*}
   This implies the first of the asserted estimates. The second estimate is proved in the same manner where
   the estimate of Proposition~\ref{prop:Plemelj_formula}~(ii) is replaced by the one from~(i). \qed

\section{Proof of Lemma~\ref{lem:example}} \label{sec:Proof_Lemma}

  In this section we prove Lemma~\ref{lem:example}. To this end we determine a class of nontrivial
  Schr\"odinger operators $L=-\Delta+V(x)$ and frequencies $\lambda\in\sigma_{ess}(L)$ such that
  (A1),(A2),(A3) holds. To check (A1) is a triviality, so we subsequently discuss (A3) and (A2). Concerning
  (A3), we first prove this assumption holds for separable potentials $V\in L^\infty(\R^d)$ given by
  $V(x)=V_1(x_1)+\ldots+V_d(x_d)$ and in particular for all $V$ described in Lemma~\ref{lem:example}. We will use that in this case each eigenpair $(\lambda_s(k),\psi_s(\cdot,k))$ of~\eqref{eq:FB_ev_problem} is given
  by 
  \begin{equation} \label{eq:eigenfunctions_separable}
    \lambda_s(k)=E_{l_1}^{V_1}(k_1)+\ldots+E_{l_d}^{V_d}(k_d),\qquad 
    \psi_s(x,k) = \phi_{l_1}^{V_1}(x_1,k_1)\cdot\ldots\cdot\phi_{l_d}^{V_d}(x_d,k_d)
  \end{equation} 
  for some $l=(l_1,\ldots,l_d)\in\N^d$ where  $(E_{l_j}^{V_j}(k_j),\phi_{l_j}^{V_j}(\cdot,k_j))$ denotes the
  complete set of eigenpairs associated
  with the one-dimensional eigenvalue problems 
  \begin{equation} \label{eq:Floquet_EVproblem_1D}
    - \phi'' + V_j \phi = E \phi \qquad\text{in }(0,1), \quad 
    \phi(1)=e^{ik_j}\phi(0),\;\phi'(0)=e^{ik_j}\phi'(0).
  \end{equation}
  Here, the labeling of the eigenpairs is chosen such that 
  $$
    E_1^{V_j}(k_j)\leq E_2^{V_j}(k_j)\leq  E_3^{V_j}(k_j)\leq \ldots  \qquad\text{for } -\pi\leq k_j\leq \pi
  $$
  in order to keep with the notation from Theorem XIII.89 (see Lemma~\ref{lem:RS_results}) in the book of Reed
  and Simon \cite{RS_analysis_of_operators}. This characterization of the Floquet-Bloch eigenpairs makes it possible to reduce the
  original eigenvalue problem~\eqref{eq:FB_ev_problem} to the one-dimensional
  problems~\eqref{eq:Floquet_EVproblem_1D}. So we start with an equiboundedness result concerning the latter
  problem, which is due to Il'in and Joo \cite{IlinJoo_uniform_estimates}, see also Theorem 2.1 in
  \cite{Kom_unif_bounded} for a short proof of this result.
  
  \begin{prop}[Teorema 1, \cite{IlinJoo_uniform_estimates}] \label{prop:IlinJoo}
    Let $q\in L^1([a,b])$ for $a,b\in\R$, $a<b$. Then there is a $C>0$ such that all solutions
    $u\in W^{2,1}([a,b])$ of 
    $$
      -u''+qu = \lambda u \quad\text{in }(a,b)
    $$
    with $\lambda\geq 0$ satisfy $\|u\|_{L^\infty([a,b])}\leq C \|u\|_{L^2([a,b])}$.
  \end{prop}

  Using~\eqref{eq:eigenfunctions_separable} this result carries over to the
  Floquet-Bloch eigenfunctions associated with bounded separable potentials.

\begin{lem}\label{lem:(A3)}
  Let $V\in L^\infty(\R^d)$ be a separable $\Z^d$-periodic potential and $(\psi_s(\cdot,k))_{s\in\Z^d}$ an
   $L^2(\Omega;\C)$-orthonormal basis consisting of Floquet-Bloch eigenfunctions associated
   with~\eqref{eq:FB_ev_problem}. Then there is a positive $C>0$ such that 
   $$
    \|\psi_s(\cdot,k)\|_{L^\infty(\Omega;\C)} \leq C \qquad\text{for all }s\in\Z^d, k\in B.
  $$
\end{lem}
\begin{proof}
  By Proposition~\ref{prop:IlinJoo} there are $C_1,\ldots,C_d>0$ such that
  \begin{align*}
    \|\psi_s(\cdot,k)\|_{L^\infty(\Omega;\C)}
    &\stackrel{\eqref{eq:eigenfunctions_separable}}{=} \|\phi_{l_1}^{V_1}(\cdot,k_1)\|_{L^\infty([0,1];\C)}\cdot
     \ldots\cdot\|\phi_{l_d}^{V_d}(\cdot,k_d)\|_{L^\infty([0,1];\C)}  \\
    &\leq C_1 \|\phi_{l_1}^{V_1}(\cdot,k_1)\|_{L^2([0,1];\C)}\cdot
      \ldots\cdot C_d\|\phi_{l_d}^{V_d}(\cdot,k_d)\|_{L^2([0,1];\C)} \\
    &\stackrel{\eqref{eq:eigenfunctions_separable}}{=} C \|\psi_s(\cdot,k)\|_{L^2(\Omega;\C)} \\
    &= C.
  \end{align*}
  Notice that the restriction $\lambda\geq 0$ from Proposition~\ref{prop:IlinJoo} is in fact not needed since
  only finitely many eigenfunctions associated with negative eigenvalues exist.  
\end{proof}
 
 We remark that for other general elliptic eigenvalue problems one can not expect the equiboundedness of
 the corresponding eigenfunctions. For instance, it is pointed out in Example~2.7 in~\cite{Kom_unif_bounded}
 that the radially symmetric eigenfunctions of the Laplacian on a three-dimensional ball associated with homogeneous
 Dirichlet boundary conditions are not equibounded in $L^\infty$. Having thus found a criterion for
 assumption (A3) we now discuss~(A2). We make use of the following result about the band
 structure of one-dimensional Schr\"odinger operators with periodic potentials:
  
  \begin{lem}[see Theorem XIII.89 \cite{RS_analysis_of_operators}] \label{lem:RS_results}
    For piecewise continuous and $1$-periodic potentials $V_i:\R\to\R$, $i=1,\ldots,d$ and eigenvalue
    functions $(E_{l_i}^{V_i})_{l_i\in\N}$ as above the following holds:
    \begin{itemize}
      \item[(i)] The functions $k_i\mapsto E_{l_i}^{V_i}(k_i)$ are symmetric about $0$ and
      analytic on $(-\pi,0)\cup (0,\pi)$.
      \item[(ii)] The functions $k_i\mapsto \phi_{l_i}^{V_i}(\cdot,k_i)\in L^2([0,1];\C)$ are
      continuous on $[-\pi,\pi]$ and analytic on $(-\pi,0)\cup (0,\pi)$.
      \item[(iii)] For odd (respectively even) $l_i\ni\N$ the function $E_{l_i}^{V_i}$ is strictly increasing
      (respectively decreasing) from $0$ to $\pi$ with 
      \begin{align*}
        E_1^{V_i}(0) < E_1^{V_i}(\pi)\leq E_2^{V_i}(\pi)<E_2^{V_i}(0)\leq 
        E_3^{V_i}(0)< E_3^{V_i}(\pi)\leq  E_4^{V_i}(\pi)<E_4^{V_i}(0)\leq \ldots
        %\ldots \\
        %&\leq E^{V_i}_{2n-1}(0)<E^{V_i}_{2n-1}(\pi)\leq E^{V_i}_{2n}(\pi)<E^{V_i}_{2n}(0)\leq \ldots. 
      \end{align*}
    \end{itemize}
  \end{lem}
  
  Notice that Theorem~XIII.89~\cite{RS_analysis_of_operators} is formulated for $2\pi$-periodic
  potentials, but the result is the same for $1-$periodic potentials by rescaling. We refer to Figure
  XIII.13~\cite{RS_analysis_of_operators} for an illustration of the situation.
  Let us comment on two subtleties. The first is that the band functions $E_{l_i}^{V_i}$ are even strictly
  monotone on $[0,\pi]$ in the sense that their derivatives only vanish at~$0$ or $\pi$. The second
  is that the possibly weaker regularity at $-\pi,0,\pi$ results from eventual intersections
  of the bands at these points of the Brillouin zone, i.e. when $E_j(\pi)=E_{j+1}(\pi)$ for odd $j$ or
  $E_j(0)=E_{j+1}(0)$ for some even $j$. As explained in the Introduction, this regularity issue is
  linked to the ordering of the eigenfunctions. Given that the interior of the first band does not intersect
  the other bands (see the first inequalities in part (iii) of the lemma), we may sharpen the statement about
  $E_1^{V_i}$ in $(-\pi,\pi)$ as follows:
  
  \begin{lem}\label{lem:RS_results improved}
    Under the assumptions of Lemma~\ref{lem:RS_results} the functions
    $k_i\mapsto E_1^{V_i}(k_i)$ and $k_i\mapsto \phi_1^{V_i}(\cdot,k_i)\in H^2([0,1];\C)$ are
    analytic on $(-\pi,\pi)$ with $(E_1^{V_i})'>0$ on $(0,\pi)$.
%     and the eigenvector functions $k\mapsto \phi_1^{V_i}(\cdot,k)\in
%     W^{2,\infty}([0,1];\C)$ are analytic on $(-\pi,\pi)$. If $V_1,V_2$ are continuously differentiable then
%     these maps are analytic as maps from $(-\pi,\pi)$ to $H^3([0,1];\C)$.
  \end{lem}
  \begin{proof}
    The proof of Theorem~XIII.89~(d),(f) on p.294 \cite{RS_analysis_of_operators} not only yields part (i)
    of Lemma~\ref{lem:RS_results}, but also proves the analyticity of $E_1^{V_i}$ on $(-\pi,\pi)$. From this,
    one deduces the analyticity of $k_i\mapsto \phi_1^{V_i}(\cdot,k_i) \in L^2([0,1];\C)$ and exploiting the
    ODE~\eqref{eq:Floquet_EVproblem_1D} this entails the analyticity of $k_i\mapsto \phi_1^{V_i}(\cdot,k_i)
    \in H^2([0,1];\C)$. Moreover, the computations from p.296 in~\cite{RS_analysis_of_operators} reveal the
    formula $$
      D(E_1^{V_i}(k_i)) = 2\cos(k_i) 
    $$
    where $D\in C^2(\R)$ is the Hill discriminant defined on p.296~\cite{RS_analysis_of_operators}.
    Differentiating this identity with respect to $k_i$ shows that the derivative of $E_1^{V_i}$ does
    not vanish on $(0,\pi)$. So the monotonicity statement from
    Lemma~\ref{lem:RS_results} actually implies $(E_1^{V_i})'>0$ on $(0,\pi)$.
%     Accordingly,
%     the eigenvector functions $k\mapsto \phi_1^{V_i}(\cdot,k)\in L^2([0,1];\C)$ are
%     analytic on $(-\pi,\pi)$. Exploiting the boundary value problem satisfied by $\phi_1^{V_i}(\cdot,k)$ and
%     $V\in L^\infty([0,1])$ we find that bounded linear functionals $H^2([0,1];\C)\to\R$ acting on
%     $\phi_1^{V_1}(\cdot,k)$ can be rewritten (using integration by parts) as a bounded linear functional
%     $L^2([0,1];\C)\to \R$ acting on $\phi_1^{V_1}(\cdot,k)$. So $k\mapsto \phi_1^{V_i}(\cdot,k)\in
%     H^2([0,1];\C)$ is analytic on $(-\pi,\pi)$. If $V_1,V_2$ are $C^1$ then the differentiated ODE implies the
%     analyticity with range $H^3([0,1];\C)$.
  \end{proof}
  
   We now use this lemma in order to verify (A2) for a class of two-dimensional
   Schr\"odinger operators and certain sufficiently low frequencies $\lambda$ in the spectrum.
     
  \begin{thm}\label{thm:A2condition}
    Let $V(x):=V_1(x_1)+V_2(x_2)$ for piecewise continuous and $1$-periodic functions
    $V_1,V_2:\R\to\R$ such that 
     $(E_1^{V_1})''(k_1),(E_1^{V_2})''(k_2)$ are positive for $k=(k_1,k_2)\in F_\lambda$ 
%     assume $\lambda\in I$ where
%     \begin{equation} \label{eq:A2_condition}
%        I:= \Big\{ \tau\in\R: E_1^{V_1}(0)+E_1^{V_2}(0)
%       < \tau
%       < \min\big\{ E_1^{V_1}(0)+E_1^{V_2}(\pi), E_1^{V_1}(\pi)+E_1^{V_2}(0)\big\} \Big\}.
%     \end{equation}
     where
    \begin{equation} \label{eq:A2_condition}
       E_1^{V_1}(0)+E_1^{V_2}(0)
      < \lambda
      < \min\big\{ E_1^{V_1}(0)+E_1^{V_2}(\pi), E_1^{V_1}(\pi)+E_1^{V_2}(0)\big\}.
    \end{equation}
    Then $L:=-\Delta+V_1(x_1)+V_2(x_2)$ satisfies assumption (A2) at
    the frequency $\lambda$.
  \end{thm}
  \begin{proof}
      We define for $s\in\Z^2$
      $$ 
        \psi_s(x,k)  := \phi^{V_1}_{l_1(s)}(x_1,k_1) \phi^{V_2}_{l_2(s)}(x_2,k_2),\qquad  
       \lambda_s(k) := E_{l_1(s)}^{V_1}(k_1)+E_{l_2(s)}^{V_2}(k_2)  
      $$  
      such that $\Z^2\ni s\mapsto (l_1(s),l_2(s))\in\N\times\N$ is a bijection with $l_1(0,0)=l_2(0,0)=1$.
      From \eqref{eq:defn_LambdaPsi} we recall that the functions $\Lambda,\Psi$ then satisfy
      $\Lambda(k+2\pi s)=\lambda_s(k)$ and $\Psi(x,k+2\pi s)=\psi_s(x,k)$ whenever $x\in\Omega,k\in
      B,s\in\Z^d$.
%     We define 
%     $$
%        \lambda_s(k) := E_1^{V_1}(k_1)+E_1^{V_2}(k_2) \quad\qquad\text{if }s=(0,0)
%     $$ 
%     and the other eigenvalues  $\lambda_s(k)$ for $s\in\Z^2,s\neq (0,0)$ are arbitrarily chosen so that they
%     are in one-to-one correspondence with $E_{l_1}^{V_1}(k_1)+E_{l_2}^{V_2}(k_2)$ for $(l,m)\in\N^2,(l,m)\neq
%     (1,1)$. Similarly, 

    \medskip 
    
    First we show that the Fermi surface $F_\lambda=\{\lambda\in\R^d: \Lambda(k)=\lambda\}$ is contained in
    the interior of the Brillouin zone $B=[-\pi,\pi]^2$.
    %Notice that in general this need not be the case. 
    Indeed, if $\Lambda(k)=\lambda$ then there are $l_1,l_2\in\N$ such that
    $E_{l_1}^{V_1}(k_1)+E_{l_2}^{V_2}(k_2)=\lambda$. Lemma~\ref{lem:RS_results}~(iii)
    and~\eqref{eq:A2_condition} imply
    \begin{align*}
      E_{l_1}^{V_1}(k_1) 
      &= \lambda - E_{l_2}^{V_2}(k_2) 
      \leq  \lambda - E_1^{V_2}(0)
      < E_1^{V_1}(\pi), \\
      E_{l_2}^{V_2}(k_2) 
      &=  \lambda - E_{l_1}^{V_1}(k_1) 
      \leq  \lambda - E_1^{V_1}(0)
      < E_1^{V_2}(\pi).
    \end{align*}
    So Lemma~\ref{lem:RS_results}~(iii) gives
    $l_1=l_2=1$ and $|k_1|,|k_2|< \pi-\delta < \pi$ for some $\delta>0$. So we have
    \begin{equation}\label{eq:Flambda}
       F_\lambda = \big\{ k\in \R^n : \Lambda(k) = \lambda\big\} \subset    (-\pi+\delta,\pi-\delta)^2=:U
    \end{equation}
    as well as
    $$
      \Lambda(k) = E_1^{V_1}(k_1) + E_1^{V_2}(k_2)\quad\text{and}\quad 
      \Psi(x,k) = \phi^{V_1}_1(x_1,k_1) \phi^{V_2}_1(x_2,k_2) \quad\text{for }x\in\Omega,k\in \ov U.
    $$
    From Lemma~\ref{lem:RS_results improved} we get that $\Lambda$ is analytic on $U$.
    Furthermore, all derivatives of $\phi^{V_i}_1(x_i,k_i)$ with respect to
    $k_i$ are continuous as maps from $(-\pi,\pi)$ to $H^2([0,1];\C)$ and thus continuous with range
    $L^\infty([0,1];\C)$. This implies for $N=3$ and arbitrary $\beta\in (0,1)$ 
    $$
      %\Lambda\in C^3(\ov U)\qquad\text{and}\qquad 
      \sup_{x\in \Omega}\|\Psi(x,\cdot)\|_{C^{N,\beta}(\ov U)}
      \leq \sup_{x_1,x_2\in [0,1]} \|  \phi^{V_1}_1(x_1,\cdot)\|_{C^{N,\beta}([-\pi+\delta,\pi-\delta])}\| 
      \phi^{V_2}_1(x_2,\cdot)\|_{C^{N,\beta}([-\pi+\delta,\pi-\delta])} <\infty. 
    $$ 
    This proves (a).

    \medskip
    
    In order to check~(b) we denote by $Z_1,Z_2$ the strictly increasing inverses of $E_1^{V_1},E_1^{V_2}$ on
    $[0,\pi)$, which exist due to Lemma~\ref{lem:RS_results}~(iii). Then we have $F_\lambda=\gamma_1\oplus \ldots\oplus
    \gamma_4$ for the positively oriented curves
    \begin{align*}
      \gamma_1(r) &:= \big(   Z_1(\lambda-r),  \phantom{-} Z_2(r) \big) 
      &&\hspace{-1cm}\in [0,\pi)\times [0,\pi)  
      &&\big( E_1^{V_2}(0)\leq r\leq \lambda-E_1^{V_1}(0) \big),    \\
      \gamma_2(r) &:= \big(- Z_1(r), Z_2(\lambda-r) \big) 
      &&\hspace{-1cm}\in (-\pi,0] \times [0,\pi)  
      &&\big( E_1^{V_1}(0)\leq r\leq \lambda-E_1^{V_2}(0) \big),    \\
      \gamma_3(r) &:= \big(- Z_1(\lambda-r),-Z_2(r) \big) 
      &&\hspace{-1cm}\in (-\pi,0]\times (-\pi,0]  
      &&\big( E_1^{V_2}(0)\leq r\leq \lambda-E_1^{V_1}(0) \big),    \\
      \gamma_4(r) &:= \big(Z_1(r),-Z_2(\lambda-r) \big) 
      &&\hspace{-1cm}\in [0,\pi)\times (-\pi,0]  
      &&\big( E_1^{V_1}(0)\leq r\leq \lambda-E_1^{V_2}(0) \big).    
    \end{align*}
    running through the four quadrants of the Brillouin zone.  
    By Lemma~\ref{lem:RS_results improved}  we have  
    $$
      \nabla \Lambda(k) 
      = \big( (E_1^{V_1})'(k_1),(E_1^{V_2})'(k_2)\big)
      \neq (0,0) 
      \quad\text{for }k=(k_1,k_2)\in F_\lambda \subset U
    $$  
    because of $(0,0)\notin F_\lambda$. So  we
    deduce from~\eqref{eq:Flambda} that $F_\lambda$ is a closed, compact, regular and analytic curve as the
    zero set of the regular analytic function $\Lambda$.  Finally, the Gaussian curvature of $F_\lambda$ at the point $k\in F_\lambda$
    is, according to Proposition~3.1 in~\cite{Gold_curvature}, given by
    \begin{align*}
      \kappa 
      &= \frac{ \big( (-E_1^{V_2})'(k_2) \;\; (E_1^{V_1})'(k_1)\big)
      \matII{(E_1^{V_1})''(k_1)}{0}{0}{(E_1^{V_2})''(k_2)} \vecII{-(E_1^{V_2})'(k_2)}{(E_1^{V_1})'(k_1)}   }{
      \big( (E_1^{V_1})'(k_1)^2+ (E_1^{V_2})'(k_2)^2\big)^{3/2}} \\
      &= \frac{ (E_1^{V_1})''(k_1) (E_1^{V_2})'(k_2)^2 + (E_1^{V_2})''(k_2) (E_1^{V_1})'(k_1)^2  }{
      \big( (E_1^{V_1})'(k_1)^2+ (E_1^{V_2})'(k_2)^2\big)^{3/2}}
      > 0.
    \end{align*}
    This establishes property (b) and the claim is proved.
  \end{proof}
  
  A reasonable criterion for the positivity assumption on the second derivatives of the first band function
  $E_1^{V_i}$ does not seem to be available in the literature. As Figure~\ref{Fig:Fermi_surfaces}~(b)
  suggests, this assumption may not hold for heavily oscillating potentials. On the other hand,
  Figure~\ref{Fig:Fermi_surfaces}~(a) indicates that the assumption may be exptected to hold for potentials
  that are close to constant ones. Notice that $E_1^{\mu_i}(k_i)=\mu_i+k_i^2$ for the constant potentials
  $\mu_i\in\R$, so $(E_1^{\mu_i})''(k_i)=2$. In the following proof of Lemma~\ref{lem:example} we make this
  rigorous using the Implicit Function Theorem.
  
  \medskip

%   \begin{lem}\label{lem:example}
%      Let $\mu_1,\mu_2\in\R$ and $\eps>0$. Then there is a $\delta>0$ such that the two-dimensional
%      Schr\"odinger operator $L=-\Delta+V_1(x_1)+V_2(x_2)$ satisfies (A1),(A2),(A3) at all frequencies
%      $\lambda\in (\mu_1+\mu_2+\eps,\mu_1+\mu_2+\pi^2-\eps)$ provided $V_1,V_2$ are $1$-periodic, piecewise
%      continuous with $\|V_1-\mu_1\|_{L^\infty([0,1])},\|V_2-\mu_2\|_{L^\infty([0,1])}<\delta$.  
%   \end{lem}
%   
  
  \noindent{\it Proof of Lemma~\ref{lem:example}:}\;  Let $\eps,\mu_1,\mu_2>0$ be given and 
  $\lambda\in (\mu_1+\mu_2+\eps,\mu_1+\mu_2+\pi^2-\eps)$ as in the statement of Lemma~\ref{lem:example}, set
  $L:=-\Delta+V_1(x_1)+V_2(x_2)$ for $1$-periodic, piecewise  continuous functions $V_1,V_2$.  
  Then $L$ is $\Z^2$-periodic and has the regularity properties required by (A1). In
  Lemma~\ref{lem:(A3)} we verified~(A3). In view of Theorem~\ref{thm:A2condition} it therefore remains to check that  
  for $\|V_i-\mu_i\|_{L^\infty([0,1])}$ sufficiently small
  we have $(E_1^{V_1})''(k_1),(E_1^{V_2})''(k_2)>0$ whenever $k=(k_1,k_2)\in F_\lambda\subset U$.
  To verify this claim we apply the Implicit Function Theorem. For any given interval  $I\subset\subset
  [-\pi,\pi]$ we define the continuously differentiable function  
  \begin{align*}
    F&:C^2(I,H^2_{per}([0,1];\C))\times C^2(I;\C)\times L^\infty([0,1]) \to  C^2(I,L^2([0,1];\C))\times
    C^2(I;\C), \\
      F&(\zeta,E,W)(x,k) := \Big( -\zeta''(x,k)-2ik\zeta'(x,k)+(W(x)-E(k)+k^2)\zeta(x,k), \int_0^1
      \zeta(y,k)^2\,dy - 1\Big).
  \end{align*}
  The motivation for this function comes from the fact that the equality $F(\zeta,E,W)=0$ holds if and only
  if $\phi(x,k):=\zeta(x,k)e^{ikx}$ solves~\eqref{eq:Floquet_EVproblem_1D}. In particular, we have
  $F(\zeta_1^{V_i},E_1^{V_i},V_i)=0$ for $\zeta_1^{V_i}(x,k):=\phi_1^{V_i}(x,k)e^{-ikx}$. So the dependence of
  $\phi_1^{V_i}$ and $E_1^{V_i}$ on the potential $V_i\approx \mu_i$ may be investigated through the Implicit
  Function Theorem applied to $F$ around its zero $(\zeta_1^{\mu_i},E_1^{\mu_i},\mu_i)$ given by
  $\zeta_1^{\mu_i}=1,E_1^{\mu_i}(k)=\mu_i+k^2$.    
    
  \medskip
    
  We claim that the partial derivative  
  $$
      F_{(\zeta,E)}(\zeta_1^{\mu_i},E_1^{\mu_i},\mu_i)[(\zeta,E)](x,k)
      = \Big(-\zeta''(x,k)-2ik\zeta'(x,k) - E(k), 2 \int_0^1 \zeta(y,k)\dy \Big)
  $$ 
  is an isomorphism. Indeed, the equation
  \begin{equation} \label{eq_ProofLemma_IFT}
    F_{(\zeta,E)}(\zeta_1^{\mu_i},E_1^{\mu_i},\mu_i)[(\zeta,E)]=(f,\sigma)\in Y\times C^2(I;\C)
  \end{equation} 
  is, by solving the corresponding linear ODE, equivalent to 
  $$
      \zeta(x,k)  = A(k)+B(k)e^{-2ikx} - \int_0^x \int_0^z e^{2ik(y-z)}\big(E(k)+f(y,k)\big)\,dy\,dz, 
  $$
  where $A,B,E\in C^2(I;\C)$ are subsequently defined by 
    \begin{align*}
      E(k) &:= - \int_0^1 f(x,k)\,dx,  \\
      %-\int_0^1 e^{-2ik(1-y)}f(y,k)\,dy -2ik \int_0^1\int_0^z e^{-2ik(z-y)}f(y,k)\,dy\,dz, \\
      B(k)&:= \frac{1}{e^{-2ik}-1} \int_0^1 \int_0^z e^{2ik(y-z)}\big(E(k)+f(y,k)\big)\,dy\,dz, \\
      A(k)&:= \frac{\sigma(k)}{2} -  B(k)\int_0^1 e^{-2ikx}\,dx + \int_0^1 \int_0^x \int_0^z
      e^{2ik(y-z)}\big(E(k)+f(y,k)\big)\,dy\,dz \,dx.
    \end{align*}
    Here the first two definitions assure that $\zeta\in C^2(I,H^2_{per}([0,1];\C))$ and the formula for
    $A(k)$ guarantees that the second equation in~\eqref{eq_ProofLemma_IFT} holds.
    Notice that the denominator in the formula for $B(k)$ shows why we have to restrict to intervals
    $I\subset\subset [-\pi,\pi]$. The Implicit Function Theorem then implies that
    $(\zeta_1^{V_i},E_1^{V_i})$ depends continuously on $V_i$ with respect to the spaces introduced above. 
    
    \medskip
    
    In particular, the above result to the interval $I:=[-\pi+\tfrac{\eps}{8\pi},\pi-\tfrac{\eps}{8\pi}]$. So
    the Implicit Function Theorem provides a $\delta>0$ such that $\|V_i-\mu_i\|_{L^\infty([0,1])}<\delta$
    implies
    \begin{equation}\label{eq:E1_approx}
      \| E_1^{V_i}-E_1^{\mu_i} \|_{C^2(I;\C)} <
       \min\big\{\frac{\eps}{4},1\big\},\quad\text{where } 
        E_1^{\mu_i}(k_i)=\mu_i+k_i^2  \qquad (i=1,2).
    \end{equation}
    For such potentials $V_1,V_2$ the Fermi surface  satisfies $F_\lambda\subset
    I\times I$.
    Indeed, $k\notin I\times I$ implies $\max\{|k_1|,|k_2|\}\geq \pi-\frac{\eps}{8\pi}$ and thus by Lemma~\ref{lem:RS_results}~(iii)
    \begin{align*}
      \Lambda(k)
      &= E_1^{V_1}(k_1) + E_1^{V_2}(k_2) \\ 
      &\geq \max\Big\{  
       E_1^{V_1}\big(\pi- \frac{\eps}{8\pi}\big) + E_1^{V_2}(0),
       E_1^{V_1}(0) + E_1^{V_2}\big(\pi- \frac{\eps}{8\pi}\big)
       \Big\} \\
       &\stackrel{\eqref{eq:E1_approx}}{\geq} \max\Big\{ 
       E_1^{\mu_1}\big(\pi- \frac{\eps}{8\pi}\big) + E_1^{\mu_2}(0) - \frac{\eps}{2},
       E_1^{\mu_1}(0) + E_1^{\mu_2}\big(\pi- \frac{\eps}{8\pi}\big) - \frac{\eps}{2} \Big\} \\
       &\stackrel{\eqref{eq:E1_approx}}{=}  \mu_1 + \mu_2 + \big(\pi- \frac{\eps}{8\pi}\big)^2 -
       \frac{\eps}{2} \\
       &> \mu_1 + \mu_2 + \pi^2 - \eps \\
       &> \lambda \hspace{5cm} \text{whenever }k\notin I\times I. 
    \end{align*}
    Furthermore, \eqref{eq:E1_approx} implies 
    $(E_1^{V_1})'',(E_1^{V_2})''>0$ on $I$ and so that the above computation implies
    $$
      (E_1^{V_1})''(k_1),(E_1^{V_2})''(k_2)>0 \quad\text{for }k=(k_1,k_2)\in F_\lambda.
    $$
    So Theorem~\ref{thm:A2condition} proves that (A2) holds and the proof of Lemma~\ref{lem:example} is
    finished.
%     Furthermore, we have 
%     $$
%       %E_1^{V_i}(\pi)    \geq 
%        E_1^{V_i}\big(\pi- \frac{\eps}{8\pi}\big)
%       \geq \mu_i + \big(\pi- \frac{\eps}{8\pi}\big)^2 - \frac{\eps}{4}  
%       > \mu_i + \pi^2 - \frac{\eps}{2} 
%     $$
%     and 
%     $$  
%       \mu_i - \frac{\eps}{4} \leq E_1^{V_i}(0)  \leq \mu_i + \frac{\eps}{4}   
%     $$
%     for $i=1,2$. This implies that $(E_1^{V_1})''(k_1),(E_1^{V_2})''(k_2)$ are positive for $k\in F_\lambda$.
    \qed

\section*{Acknowledgments}

  The author thanks Tom\'a\v{s} Dohnal (University of Dortmund) for stimulating discussions about
  Floquet-Bloch theory during the past years and for providing Figure~\ref{Fig:Fermi_surfaces}.  Additionally,
  he gratefully acknowledges financial support by the Deutsche Forschungs\-gemeinschaft (DFG) through
  CRC~1173 ''Wave phenomena: analysis and numerics''.

\bibliographystyle{plain}
\bibliography{biblio}

\end{document}